%% file: Oseen.tex
\RequirePackage{fix-cm}
\documentclass[smallextended]{svjour3}       
\smartqed  
\usepackage{amssymb,amsfonts,amsmath,bm,mathrsfs}
\usepackage{hyperref}
\usepackage{graphicx}
\usepackage{color}
\usepackage{eepic,epic}
\usepackage{epsfig,subfigure,epstopdf}
\usepackage[showonlyrefs]{mathtools}
\usepackage{cite}
\usepackage{enumerate,enumitem}

\usepackage{layouts}
\usepackage{stmaryrd}
\usepackage{graphicx}
\usepackage{ulem}
\usepackage{tikz}

\allowdisplaybreaks

\usepackage{lipsum,titletoc}

\titlecontents{section} 
[2.5em]                 
{\rmfamily}
{\contentslabel{2.3em}} 
{\hspace*{-2.3em}} {\titlerule*[1pc]{.}\contentspage}

\titlecontents{subsection} 
[4.8em]                 
{\rmfamily}             
{\contentslabel{2.3em}} 
{\hspace*{-2.3em}} {\titlerule*[1pc]{.}\contentspage}

\numberwithin{equation}{section}
\numberwithin{lemma}{section}
\newtheorem{assumption}[theorem]{\sc Assumption}
\newtheorem{algorithm}[theorem]{\sc Algorithm}

\input{macros.tex}
\def\makeheadbox{\hspace{-15pt}Version of \today}

\begin{document}

\title{A high-order unfitted finite element method for moving interface problems
}
\titlerunning{\,}        

\author{Chuwen Ma\and Weiying Zheng$^*$
}

\authorrunning{\,} 

\institute{Chuwen Ma\at School of Mathematical Science, University of Chinese Academy of Sciences. Institute of Computational Mathematics and Scientific/Engineering Computing, Academy of Mathematics and System Sciences, Chinese Academy of Sciences, Beijing, 100190, China. \\
              \email{{chuwenii@lsec.cc.ac.cn}}           
           \and
           Weiying Zheng \at
              LSEC, Institute of Computational Mathematics and Scientific/Engineering Computing, Academy of Mathematics and System
              Sciences, Chinese Academy of Sciences, Beijing, 100190, China.
              School of Mathematical Science, University of Chinese Academy of Sciences.
              The author was supported in part by the National Science Fund for Distinguished Young Scholars 11725106 and by China NSF major project 11831016. \\
            \email{zwy.@lsec.cc.ac.cn}
}

\date{Received: date / Accepted: date}

\maketitle

\begin{abstract}
We propose a $k^{\rm th}$-order unfitted finite element method ($2\le k\le 4$) to solve moving interface problem of the Oseen equations. Thorough error estimates for the discrete solutions are presented by considering errors from interface-tracking, time integration, and spatial discretization. In literatures on time-dependent Stokes interface problems, error estimates for the discrete pressure are usually sub-optimal, namely, $(k-1)^{\rm th}$-order, under the $L^2$-norm. We have obtained a $(k-1)^{\rm th}$-order error estimate for the discrete pressure under the $H^1$-norm. Numerical experiments for a severely deforming interface show that optimal convergence orders are obtained for $k = 3$ and $4$.
\\[5pt] 
\textbf{Keywords}: 
optimal control, 
wave equation, 
unbounded domain, 
boundary integral equation, 
well-posedness, 
convolution quadrature, 
stability, 
convergence, 
error estimate.
\end{abstract}



\tableofcontents

\section{introduction}
Partial differential equations (PDEs) on time-varying domains
are frequently encountered in various applications in biology, physics, and engineering,
such as blood-flows, fluid-structure interaction, free-surface problems, etc.
It is well-known that moving interface problems may cause challenges to
high-order numerical simulations and rigorous error analysis. In this paper, we study numerical methods for two-dimensional Oseen equations with a time-varying interface.

Let $\Omega\subset\bbR^2$ be an open rectangle whose boundary is denoted by $\Sigma=\partial\Omega$. For any $t\ge 0$, let $\Omega_1(t)$ and $\Omega_2(t)$ be two time-varying sub-domains of $\Omega$ occupied by two immiscible fluids. We assume that $\Omega$ is fixed and denote the interface by $\Gamma(t)=\partial\Omega_1(t) \subset\Omega$. Then $\partial\Omega_2(t)=\Sigma\cup\Gamma(t)$ (see Fig.~\ref{fig:pro fig}).
The linear interface problem of two-phase incompressible fluids is given as follows
\begin{subequations}\label{eq:model}
	\begin{align}
		\frac{\partial \Bu_i}{\partial t} +(\Bw \cdot\nabla )\Bu_i
		-\nu_i\Delta \Bu_i +\nabla p_i = \Bf_i,\quad 	\Div \Bu_i=0	
		\quad &\text{in}\;\; \Omega_i(t),  \label{eq:Oseen0}\\
		\Bu_i|_{t=0}=\Bu_{i,0}\quad &\text{in}\;\;\Omega_i(0),  \label{eq:Oseen1} \\
		\nu_1 \partial_\Bn\Bu_1 - p_1 \Bn =\nu_2 \partial_\Bn\Bu_2- p_2 \Bn,\quad
		\Bu_1 =\Bu_2 \quad &\text{on}\;\; \Gamma(t),\label{eq:Oseen2} \\
		\Bu_2 =\textbf{0} \quad &\text{on}\;\; \Sigma,\label{eq:Oseen3}
	\end{align}	
\end{subequations}
where $\Bu_i$, $p_i$, and $\Bf_i$, $i=1,2$ stand for the flow velocity, the pressure, and the applied body force in each phase,  respectively. Here $\Bw$ is the advection velocity which drives the variation of $\Gamma(t)$, namely,
\ben
\Gamma(t) =\big\{\BX(t;0,\Bx): \forall\,\Bx\in\Gamma(0)\big\},
\een
where, for any $t\ge s\ge 0$ and any $\Bx\in\Omega$, $\BX(t;s,\Bx)$ is the solution to the initial-value problem of ordinary differential equation (ODE)
\begin{equation}\label{ODE}
	\frac{\D}{\D t}\BX(t;s,\Bx) = \Bw(\BX(t;s,\Bx),t),\qquad
	\BX(s;s,\Bx) =\Bx.
\end{equation}
Moreover, \eqref{eq:Oseen1} stands for the transmission conditions for the solution across $\Gamma(t)$, and the unit normal $\Bn$ points to the interior of $\Omega_2(t)$.
The viscosities $\nu_1,\nu_2$ are positive constants. Without loss of generality, we assume $0<\nu_2<\nu_1=1$ and define $\nu\in L^\infty(\Omega)$ by $\nu=\nu_i$ in $\Omega_i(t)$, $i=1,2$.

\begin{figure}[htp]
	\centering
	\begin{tikzpicture}[scale =1.5]	
		\draw[thick,fill=red!30! white] (0,0) rectangle (2,2);
		\draw[fill=yellow!30!white, thick] (1,1) ellipse [x radius=0.6cm, y radius=0.5cm];
		\node[left] at (1.32, 1) {$\tiny{\Omega_1(t)}$};
		\node[right] at (1.3, 0.45) {$\tiny{\Gamma(t)}$};
		\node[right] at (1.3,1.6) {$\Omega_2(t)$};
	\end{tikzpicture}
	\caption{$\Omega_i(t)$, $i=1,2$ is colored in yellow and red, respectively, and $\Gamma(t)$ is
		the moving interface.}     \label{fig:pro fig}
\end{figure}
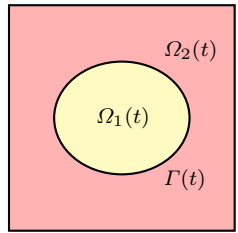

In the literature on interface problems,
numerical methods on unfitted
grids (referred to ``unfitted methods'' hereafter) are very popular during the past few decades.
To mention some of them, we refer
to \cite{bre12,mit05,pes77} for immersed boundary methods,
to \cite{lev94,li06} for immersed interface methods,
to \cite{li98,li03,lin09} for immersed finite element methods,
to \cite{bec09,hua17,liu20, han02,wu19} for Nitsche extended finite element methods,
to \cite{bor17,bur15} for cut finite element methods, and to \cite{adj19} for immersed discontinuous Galerkin method.
The essential idea is to double the degrees of freedom on interface elements so that interface conditions can be enforced explicitly in basis functions or weakly in discrete formulations.
Similar ideas are used in fictitious-domain methods
which enhance the stability of numerical solutions by penalizing face jumps of their normal derivatives \cite{bur10,bur12,jar09,mas14}.

Although unfitted methods have been well-developed for stationary problems,
they may encounter a significant challenge for solving dynamic interface problems.
Since the computational domain (or its sub-domains) is varying in time, numerical solutions computed in previous time steps are compounded with flow maps in the current time step.
Traditional methods for time integrations can not be applied directly to this case \cite{fri09,zun13}.
One way to deal with the issue is the space-time method, which uses discontinuous Galerkin method to the time variable and uses extended finite element method to the spatial variable \cite{leh13,leh15}. We refer to the very recent work \cite{guo21} which presents error estimates for immersed finite element method for the parabolic equation with a time-varying interface. In \cite{leh19}, Lehrenfeld and Olshanaskii propose an implicit Euler finite element method for the advection-diffusion equation on time-varying domains.
In \cite{wah20}, von Wahl and Richter and Lehrenfeld extend this method to solve the Stokes equations on a time-varying domain. By extending the discrete solution to a slightly large neighbourhood at every time step, Lou and Lehrenfeld propose a high-order method for solving the advection-diffusion equation \cite{lou21}. The method is based on isoparametric unfitted finite element and backward-differentiation formulas (BDF). In \cite{ma21},
the authors proposed a high-order finite element method for
solving the advection-diffusion equation on time-varying domain. Thorough error estimates are given for third- and fourth-order methods by taking account of errors from boundary tracking, time integration, and spatial discretization.

This paper extends the numerical study for varying-boundary problem of the advection-diffusion equation to varying-interface problem of the Oseen equations. The extension is essentially nontrivial, considering the discrete LBB condition for varying-interface problems and the time integration along characteristic curves.
We have overcome the difficulties by defining
a modified Stokes projection onto finite element spaces and introducing adequate penalty terms to the pressure variable. The main contributions of the paper are summarized as follows.
\begin{itemize}[leftmargin=5mm]
	\item [1.] For $2\le k\le 4$, we propose a $k^{\rm th}$-order numerical method for the varying-interface problem \eqref{eq:model}. The method uses $k^{\rm th}$-order time integration along characteristic curves and $k^{\rm th}$-order finite elements on unfitted Eulerian meshes. The interface is formed dynamically with the $(k+1)^{\rm th}$-order interface-tracking method in each time step.
	
	\item [2.] We present a thorough error analysis for the finite element method by taking into account all errors from spatial discretization, time integration, and interface-tracking process. Optimal error estimates are obtained for the discrete velocity under weighted $H^1$-norm.
	
	\item [3.] By incorporating adequate residual-based penalties into the discrete formulation, we have obtained the $(k-1)^{\rm th}$-order error estimates for the discrete pressure under weighted $H^1$-norm, while error estimates for the discrete pressure are of $(k-1)^{\rm th}$-order under the $L^2$-norm in the literature \cite{wah20}.
	
	\item [4.] Numerical experiment for a severely deforming interface show that optimal convergence orders are obtained for both the third-order and the fourth-order methods.
\end{itemize}
\vspace{2mm}

The rest of the paper is organized as follows.
In section 2, we reformulate the Oseen model in Lagrangian coordinates. Discrete flow maps are introduced via the $(k+1)^{\rm th}$-order Runge-Kutta (RK) scheme.
In section 3, we introduce the interface-tracking algorithm which generates the computational interface in each time step. Error estimates between the exact interface and the approximate one are cited from \cite{ma21}.
In section 4, we propose a $k^{\rm th}$-order unfitted finite element method for solving the interface problem in each time step. The well-posedness of the discrete problem is established.
In section 5, we introduce a modified Stokes projection operator and prove error estimates for the projection.
Section 6 is devoted to proving the stability of numerical solutions. In section~7, a thorough error analysis is presented by taking into account all errors from approximate interfaces, time integrations, and finite element discretization.
In section 8, optimal convergence orders are verified for both the third-order and the fourth-order methods by a numerical experiment with severely deforming interface.

\section{The interface-tracking algorithm}

In this section, we present an algorithm for constructing an approximate interface in each time step.
In practice, it is difficult to track the interface $\Gamma(t)$ with the true solution to the ODE \eqref{ODE}. A realistic approach is to track a finite number of control points on $\Gamma(0)$ by solving \eqref{ODE} numerically and construct an approximate interface based on these control points.

Let $T$ be the final time of evolution and let $t_n=n\tau$, $n=0,\cdots,N$, be a uniform partition of $[0,T]$ with time-step size $\tau=T/N$. Let $\BX^{n-j,n}:= \BX(t_n;t_{n-j},\bullet)$, $0\le j \le n$, denote exact flow maps at discrete time points.
The uniqueness of the solution to \eqref{ODE} implies that
$\BX^{n-j,n}$ is one-to-one and maps $\Omega_i(t_{n-j})$ to $\Omega_i(t_n)$ for $i=1,2$. For convenience,
we denote the inverse of $\BX^{n-j,n}$ by $\BX^{n,n-j} := (\BX^{n-j,n})^{-1}$.

\subsection{Approximate flow maps}

For any $\Bx^{n-1}\in\Omega$, the image
$\Bx^{n} = \BX^{n-1,n}_\tau(\Bx^{n-1})$ is calculated
by the RK-$(k+1)$ (the $(k+1)^{\rm th}$-order Runge-Kutta) scheme
\begin{equation}\label{eq:RK}
	\begin{cases}
		\Bx^{(i)} =\Bx^{n-1} +\tau \sum_{j=1}^{i-1}
		\alpha^{k+1}_{ij} \Bw(\Bx^{(j)},t^{(j)}),\quad
		t^{(j)}   = t_{n-1} + \gamma_j^{k+1} \tau,\vspace{1mm}\\
		\Bx^{n} =\Bx^{n-1} +\tau \sum_{i=1}^{n_{k+1}}
		\beta_{i}^{k+1} \Bw(\Bx^{(i)},t^{(i)}) .
	\end{cases}
\end{equation}
Here $n_{k+1}$ is the number of stages and $\alpha_{ij}^{k+1},\beta_i^{k+1},\gamma_i^{k+1}$ are coefficients of the RK-$(k+1)$ scheme,
satisfying $\alpha_{ij}^{k+1} =0$ for $j\geq i$.

We denote the map from $\Bx^{n-1}$ to $\Bx^{(i)}$ by
\ben
\phibf_{n-1}^{(i)}(\Bx^{n-1}) :=\Bx^{(i)}, \qquad
i=1,\cdots,n_{k+1}.
\een
Then $\BX^{n-1,n}_\tau$ can be represented explicitly as follows
\begin{equation}\label{Xn}
	\BX^{n-1,n}_\tau(\Bx) = \Bx +\tau\sum_{i=1}^{n_{k+1}}\beta_i^{k+1}
	\Bw\big(\phibf^{(i)}_{n-1}(\Bx),t^{(i)}\big).
\end{equation}
The multi-step mapping is defined by
\begin{equation}\label{Xmstep}
	\BX^{n-i,n}_\tau = \BX^{n-1,n}_\tau
	\circ \BX^{n-2,n-1}_\tau
	\circ\cdots\circ \BX^{n-i,n-i+1}_\tau,\qquad1\le i\le n.
\end{equation}
The inverse of $\BX^{n-i,n}_\tau$ is denoted by
$\BX^{n,n-i}_\tau := (\BX^{n-i,n}_\tau)^{-1}$.

\begin{lemma} [{\cite[Lemma~3.3]{ma21}}]
	There exists a constant $C>0$ independent of $\tau$ such that for $\nu=0,1$,
	\begin{align}	
		\N{\BX^{m,n}_\tau-\BX^{m,n}}_{\BW^{\mu,\infty}(\Omega)}
		\le C (n-m)\tau^{k+2-\mu},\qquad
		0\le m\le n\le N. 	\label{ieq:Xmn}
	\end{align}
\end{lemma}

\subsection{Interface-tracking algorithm}

Now we describe the algorithm which tracks the interface $\Gamma(t_n)$ approximately.
The algorithm is adopted from \cite{ma21}. Let $J$ be the number of control points on the initial boundary $\Gamma^0\equiv\Gamma(0)$, let $L$ be the arc length of $\Gamma^0$, and let $\eta:= L/J$ be the segment size for interface-tracking.
Let $\Cp^0= \left\{\Bp_j^0\in\Gamma^0: j=0,1,\cdots,J\right\}$ be a partition of $\Gamma^0$ with $\Bp_0^0=\Bp_{J}^0$.
Suppose there is a parametrization $\Gamma^0=\left\{\chibf_0(l): l\in [0,L]\right\}$ which satisfies  $\chibf_0\in\BC^4([0,L])$ and
\ben
\chibf_0(L_j)=\Bp^0_j,\qquad
L_j =j\eta ,\qquad 0\le j\le J .
\een
The set of nodes on $\Gamma^0$
is defined by $\Cl =\left\{L_j:\; j=0,1,\cdots,J\right\}$.
\vspace{2mm}

\begin{center}
\fbox{\parbox{0.975\textwidth}
{\begin{algorithm}[\hspace{-0.2mm}\cite{ma21}]\label{alg:spline}
Given $n\ge 1$ and $\Gamma^0_\eta \equiv\Gamma^0$, the interface-tracking algorithm for constructing $\Gamma^n_\eta$ from $\Gamma^{n-1}_\eta$ consists of two steps.
\begin{enumerate}[leftmargin=6mm]
\item Trace forward all control points in $\Cp^{n-1}$ to obtain the set of control points at $t=t_n$,
\ben
\Cp^n=\left\{\Bp^{n}_{j} =\BX^{n-1,n}_\tau(\Bp^{n-1}_{j}):\;	j=0,\cdots, J\right\}.
\een
		
\item Compute the cubic spline interpolation $\chibf_n\in\BC^2([0,L])$ based on $\Cl$ and $\Cp^n$. Define
\begin{equation}\label{eq:ln}
\Gamma^{n}_\eta:=\left\{\chibf_{n}(l): l\in [0,L]\right\}.
\end{equation}
\end{enumerate}
\end{algorithm}}}
\end{center}

%

\vspace{1mm}

In each time step, we clarify three interfaces, the exact interface $\Gamma(t_n)$,
the approximate interface $\Gamma^n$ obtained with the RK-$(k+1)$ scheme,
and the practical interface $\Gamma^n_\eta$ used for numerical computations. Then three interfaces are, respectively, parameterized as follows
\begin{align}\label{para-chi}
	\left\{
	\begin{array}{ll}
		\Gamma(t_n)= \left\{\hat\chibf_n(l): l\in [0,L]\right\},
		&\quad \hat\chibf_n=\BX^{0,n}\circ \chibf_0,	\vspace{1mm}\\
		\Gamma^n = \left\{\tilde\chibf_n(l): l\in [0,L]\right\},
		&\quad \tilde\chibf_n=\BX^{0,n}_\tau\circ \chibf_0,
		\vspace{1mm}\\
		\Gamma^n_\eta= \left\{\chibf_n(l): l\in [0,L]\right\},
		&\quad \chibf_n\in\BC^2([0,L]),
	\end{array}
	\right. \qquad 1\le n\le N.
\end{align}

\subsection{Error estimates for interface tracking}

Let $L^2(\Omega)$ denote the space of square-integrable functions on $\Omega$, $L^2_0(\Omega)$ the subspace of functions whose integrals on $\Omega$ are zero, and $H^m(\Omega)$  the subspace of functions whose partial derivatives of order up to $m$ belong to $L^2(\Omega)$. The inner product on $\Ltwo$ is denoted by
\ben
(u,v)_{\Omega} =\int_\Omega uv,\qquad \forall\,u,v\in\Ltwo.
\een
Moreover, $L^\infty(\Omega)$ denotes the space of functions which are essentially bounded on $\Omega$ and $W^{m,\infty}(\Omega)$ the subspace of functions whose partial derivatives of order up to $m$ belong to $L^\infty(\Omega)$. We denote vector-valued quantities by boldface symbols, such as $\Ltwov=(\Ltwo)^2$,
and denote matrix-valued quantities by blackboard bold symbols,
such as $\bbL^{2}(\Omega)=(\Ltwo)^{2\times2}$.
Hereafter, $(\cdot,\cdot)_{\Omega}$ denotes the inner products on $L^2(\Omega)$,  $\BL^2(\Omega)$, and $\bbL^2(\Omega)$, in their respective circumstances.

Throughout the paper, the notation $f \lesssim g$ means
$f\le Cg$ where $C$ is a generic constant and independent of sensitive quantities,
such as the segment size $\eta$,
the spatial mesh size $h$, the time-step size $\tau$,
the number of time steps $n$, and the viscosity $\nu$.
So are the constants $C_i$, $i=0,1,\cdots$, in the following texts.

To focus on high-order error estimates, we make the assumptions on the driving velocity $\Bw$ and the exact interfaces $\Gamma(t_n)$: \vspace{1mm}

\begin{assumption}\label{ass-1}
	Assume that
	\begin{itemize}[leftmargin=5mm]
		\item $\Bw\in C^r([0,T]\times\bar\Omega)$ and $\Bw(\cdot,t)\in \BC_0^r(\Omega)$, $r\geq 4$, for all $t\ge 0$, and that
		\vspace{1mm}
		
		\item the parametrization of $\Gamma(t_n)$ satisfies
		$\N{\hat{\chibf}_n}_{\BC^4([0,L])}\lesssim 1$ for all $0\leq n\leq N$.
	\end{itemize}
\end{assumption}
\vspace{2mm}
Let the Jacobi matrices of $\BX^{n-i,n}$,  $\BX^{n-i,n}_\tau$,
and $\BX^{n,n-i}$, $\BX^{n,n-i}_\tau$ be denoted by
\ben
\begin{array}{ll}
	\displaystyle
	\bbJ^{n-i,n}:= \frac{\partial \BX^{n-i,n}}{\partial\Bx},\quad
	& \bbJ^{n,n-i}:= \left(\bbJ^{n-i,n}\right)^{-1}, \vspace{2mm}\\
	\displaystyle
	\bbJ^{n-i,n}_\tau:= \frac{\partial \BX^{n-i,n}_\tau}{\partial\Bx},\quad
	&\bbJ^{n,n-i}_\tau:= \left(\bbJ^{n-i,n}_\tau\right)^{-1}.
\end{array}
\een
From \eqref{Xn}, it is easy to see that
\begin{equation}\label{Jnn1}
	\bbJ^{n-1,n}_\tau(\Bx)
	= \bbI + \tau \sum_{i=1}^{n_{k+1}} \beta_i^{k+1}
	\nabla \Bw\big(\phibf_{n-1}^{(i)}(\Bx) ,t^{(i)}\big)
	\nabla \phibf_{n-1}^{(i)}(\Bx), \qquad \forall\,\Bx\in\Omega,
\end{equation}
where
\begin{equation}
	\nabla \phibf^{(1)}_{n-1} =\bbI,\qquad
	\nabla \phibf^{(i)}_{n-1}= \bbI + \tau \sum_{j=1}^{i-1} \alpha_{ij}^{k+1} \nabla \Bw\big(\Bx^{(j)},t^{(j)}\big)
	\nabla \phibf^{(j)}_{n-1},\quad i\ge 2.
\end{equation}
From \cite[Lemma~3.3]{ma21}, we cite some useful estimates for the Jacobi matrices
\begin{align}	
	&\N{\bbJ^{m,n}}_{\bbW^{3,\infty}(\Omega)} +
	\N{\bbJ^{m,n}_\tau}_{\bbW^{3,\infty}(\Omega)}\lesssim 1,\qquad 0\le m, n\le N,
	\label{ieq:Jmn}\\
	&\N{\bbJ^{n-i,n}_\tau-\bbI}_{\bbL^{\infty}(\Omega)}
	+\N{\bbJ^{n,n-i}_\tau-\bbI}_{\bbL^{\infty}(\Omega)}
	\lesssim\tau, \qquad 0\le i\le k+1.	\label{ieq:Jtau}
\end{align}

\begin{theorem}\label{thm:chi}
	Let Assumption~\ref{ass-1} be satisfied. Then $\N{\tilde{\chibf}_n}_{\BC^4([0,L])}\lesssim 1$ and
	\begin{align}\label{chi-err}
		\begin{cases}
			\N{\chibf_n-\tilde\chibf_n}_{\BC^{\mu}([0,L])}
			\lesssim \eta^{4-\mu},\vspace{1mm} \\
			\N{\chibf_n-\hat\chibf_n}_{\BC^{\mu}([0,L])}
			\lesssim \eta^{4-\mu} +\tau^{k+1-\mu},
		\end{cases}
		\quad \mu=0,1.
	\end{align}
\end{theorem}
\begin{proof}
	From \eqref{Xmstep} and \eqref{para-chi}, we have $\tilde\chibf_n(l) = \BX^{n-1,n}_\tau\circ\tilde\chibf_{n-1}$. The chain rule implies
	\begin{align*}
		&\tilde\chibf_n' = \prod_{m=1}^n \bbJ^{m-1,m}_\tau \tilde\chibf_0',\qquad
		\tilde\chibf_n'' = \big\{ \big(\tilde\chibf_{n-1}'\cdot\nabla\big) \bbJ^{n-1,n}_\tau \big\} \tilde\chibf_{n-1}'
		+\bbJ^{n-1,n}_\tau \tilde\chibf_{n-1}''.
	\end{align*}

	Since $\Bw\in\BC^r([0,T]\times\bar\Omega)$ with $r\geq 4$, from \eqref{Jnn1} we infer that $\bbJ^{n-1,n}_\tau\in \BC^3(\bar\Omega)$ and $\SN{\bbJ^{n-1,n}_\tau}_{\bbW^{j,\infty}(\Omega)}\lesssim \tau$ for $j=1,2,3$. Together with \eqref{ieq:Jmn}--\eqref{ieq:Jtau}, they show that
	\begin{align*}
		&\N{\tilde\chibf_n'}_{\BC([0,L])} \le (1+C\tau)^n \N{\tilde\chibf_0'}_{\BC([0,L])}
		\lesssim 1, \\
		&\N{\tilde\chibf_n''}_{\BC([0,L])} \le
		(1+C\tau)\big(C\tau +\N{\tilde\chibf_{n-1}''}_{\BC([0,L])}\big)
		\lesssim 1 + \N{\tilde\chibf_{0}''}_{\BC([0,L])}
		\lesssim 1.
	\end{align*}
	Furthermore, $\N{\tilde\chibf_n'''}_{\BC([0,L])}$ and $\big\|\tilde\chibf_n^{(4)}\big\|_{\BC([0,L])}$ can be estimated similarly.

	Since $\chibf_n$ is the cubic spline interpolation of $\tilde\chibf_n$, the first inequality of \eqref{chi-err} follows directly from standard error estimates for cubic spline interpolations.
	The second inequality can be proven by triangular inequality and standard error estimates for the RK-$(k+1)$ scheme.
\end{proof}
\vspace{1mm}

\begin{theorem}[\hspace{-0.2mm}{\cite[Theorem~3.5]{ma21}}]\label{thm:bdr}
	Suppose $r\ge \max(k+2,4)$. For any $0\le n,m\le N$ and $|m-n|\leq k$, there holds for $\mu=0,1$
	\begin{align}\label{eq:Gamma-tn-err}		
		\N{\chibf_n - \BX_{\tau}^{m,n}\circ\chibf_m}_{\BC^{\mu}([0,L])}
		&\lesssim \tau\eta^{-\mu}\sum_{i=0}^{k+1} \big(\tau^{i}\eta^{\min(4,r-i)}+ \tau^{k+1}\big).
	\end{align}
\end{theorem}

\subsection{A semi-discrete scheme}

For any $\Bx_0\in\Omega_0$, we use the continuous flow map and write $\Bx\equiv \Bx(t)=\BX(t;0,\Bx_0)$. The material derivative of $\Bu_i$ is defined as
\begin{equation}\label{eq:DuDt}
	\frac{\D}{\D t}\Bu_i(\Bx(t),t)
	= \frac{\partial \Bu_i}{\partial t}(\Bx,t)
	+\Bw(\Bx,t) \cdot\nabla_\Bx \Bu_i(\Bx,t),\qquad i=1,2.
\end{equation}
The momentum equation of \eqref{eq:model} can be written equivalently as follows
\begin{align}\label{eq:compact}
	\frac{\D\Bu_i}{\D t} -\nu_i\Delta \Bu_i +\nabla p_i = \Bf_i
	\quad \text{in}\;\; \Omega_{i}(t),  \qquad i=1,2.
\end{align}

We apply the $k^{\rm th}$-order backward differentiation formula (BDF-$k$) to \eqref{eq:compact} and
obtain a semi-discrete approximation to problem \eqref{eq:model}:
\begin{center}
	\fbox{\parbox{0.975\textwidth}
		{Given $(\Bu^m_1,\Bu^m_2)$ for $n-k\le m< n$, solve the coupled problems for $(\Bu_1^n,\Bu_2^n)$
			\begin{subequations}\label{eq:semi-disc}
				\begin{align}
					\frac{1}{\tau}\Lambda^k\bbU_i^n -\nu_i\Delta \Bu^n_i +\nabla p_i^n= \Bf^n_i,\quad
					\Div \Bu^n_i =0 \quad &\text{in}\;\; \Omega_{\eta,i}^n,
					\quad i=1,2, \label{eq:semi-disc-0}\\
					\nu_1 \partial_\Bn\Bu_1^n - p_1^n \Bn =\nu_2 \partial_\Bn\Bu_2^n- p_2^n\Bn,\quad
					\Bu_1^n =\Bu_2^n \quad &\text{on}\;\; \Gamma^n_\eta,
					\label{eq:semi-disc-1}\\
					\Bu^n_2 =\textbf{0} \quad &\text{on}\; \;\Sigma,\label{eq:semi-disc-2}
				\end{align}	
			\end{subequations}
			where $\bbU^n_i = \big[\BU^{n-k,n}_i,\cdots,\BU^{n,n}_i\big]$,
			$\BU_i^{m,n}=\Bu_i^m\circ\BX^{n,m}_\tau$, and
			$\Lambda^k\bbU^n_i	= \sum_{j=0}^k \lambda_j^k \BU^{n-j,n}_i$.}}
\end{center}\vspace{1mm}
Here $\Bf^n_i=\Bf_i(t_n)$ and $\lambda_0^k,\cdots,\lambda_k^k$ are coefficients for the BDF-$k$ (see \cite{liu13,ma21}).

\section{The finite element method}

The purpose of this section is to propose a fully discrete finite element scheme for solving \eqref{eq:model}. Suppose $\Gamma^n_\eta\subset \Omega$ is the approximate interface obtained with Algorithm~\ref{alg:spline}.
Let $\Omega^n_{\eta,1}$ denote the open domain surrounded by $\Gamma^n_\eta$,
namely, $\Gamma^n_\eta=\partial\Omega^n_{\eta,1}$.
Define
\ben
\Omega_{\eta,2}^n = \Omega\backslash\bar \Omega_{\eta,1}^n,\qquad
\Omega^n_\eta =\Omega\backslash\Gamma^n_\eta = \Omega^n_{\eta,1}\cup \Omega^n_{\eta,2}.
\een

\subsection{A semi-discrete scheme}

For any $\Bx_0\in\Omega_0$, we use the continuous flow map and write $\Bx\equiv \Bx(t)=\BX(t;0,\Bx_0)$. The material derivative of $\Bu_i$ is defined as
\begin{equation}\label{eq:DuDt}
	\frac{\D}{\D t}\Bu_i(\Bx(t),t)
	= \frac{\partial \Bu_i}{\partial t}(\Bx,t)
	+\Bw(\Bx,t) \cdot\nabla_\Bx \Bu_i(\Bx,t),\qquad i=1,2.
\end{equation}
The momentum equation of \eqref{eq:model} can be written equivalently as follows
\begin{align}\label{eq:compact}
	\frac{\D\Bu_i}{\D t} -\nu_i\Delta \Bu_i +\nabla p_i = \Bf_i
	\quad \text{in}\;\; \Omega_{i}(t),  \qquad i=1,2.
\end{align}

We apply the $k^{\rm th}$-order backward differentiation formula (BDF-$k$) to \eqref{eq:compact} and
obtain a semi-discrete approximation to problem \eqref{eq:model}:
\begin{center}
	\fbox{\parbox{0.975\textwidth}
		{Given $(\Bu^m_1,\Bu^m_2)$ for $n-k\le m< n$, solve the coupled problems for $(\Bu_1^n,\Bu_2^n)$
			\begin{subequations}\label{eq:semi-disc}
				\begin{align}
					\frac{1}{\tau}\Lambda^k\bbU_i^n -\nu_i\Delta \Bu^n_i +\nabla p_i^n= \Bf^n_i,\quad
					\Div \Bu^n_i =0 \quad &\text{in}\;\; \Omega_{\eta,i}^n,
					\quad i=1,2, \label{eq:semi-disc-0}\\
					\nu_1 \partial_\Bn\Bu_1^n - p_1^n \Bn =\nu_2 \partial_\Bn\Bu_2^n- p_2^n\Bn,\quad
					\Bu_1^n =\Bu_2^n \quad &\text{on}\;\; \Gamma^n_\eta,
					\label{eq:semi-disc-1}\\
					\Bu^n_2 =\textbf{0} \quad &\text{on}\; \;\Sigma,\label{eq:semi-disc-2}
				\end{align}	
			\end{subequations}
			where $\bbU^n_i = \big[\BU^{n-k,n}_i,\cdots,\BU^{n,n}_i\big]$,
			$\BU_i^{m,n}=\Bu_i^m\circ\BX^{n,m}_\tau$, and
			$\Lambda^k\bbU^n_i	= \sum_{j=0}^k \lambda_j^k \BU^{n-j,n}_i$.}}
\end{center}\vspace{1mm}
Here $\Bf^n_i=\Bf_i(t_n)$ and $\lambda_0^k,\cdots,\lambda_k^k$ are coefficients for the BDF-$k$ (see \cite{liu13,ma21}).

\subsection{Finite element meshes}
Let $\Ct_h$ be the uniform partition of $\Omega$ into closed squares of side-length $h$. It generates the covers of $\Omega^n_{\eta,1}$, $\Omega^n_{\eta,2}$, and $\Gamma^n_{\eta}$, which are, respectively, denoted by
\begin{align*}
	\Ct^n_{h,i} := \left\{K\in \Ct_h: K \cap \bar{\Omega}_{\eta,i}^n \neq \emptyset\right\}, 	\quad
	\Ct^n_{h,B} := \left\{K\in \Ct_h:\;
	\mathrm{length}(K\cap\Gamma^n_{\eta}) >0\right\}.
\end{align*}
For each $i$, we define a domain containing $\Omega^n_{\eta,i}$ and a domain contained in $\Omega^n_{\eta,i}$
\begin{equation}\label{fic-domain}
	\Omega^n_{h,i} :=\mathrm{interior}(
	\cup\{K: K\in\Ct^n_{h,i}\}),\quad
	\omega^n_{h,i} :=\mathrm{interior}(\cup\{K: K\in\Ct^n_{h,i}\backslash\Ct^n_{h,B}\}) .
\end{equation}
Let $\Ce_h$ denote the set of all interior edges of $\Ct_h$ and let $\Ce_{i,B}^{n}$ denote
the set of edges of $\Ct^n_{h,B}$ which are not on the boundary of $\Omega^n_{h,i}$, namely,
\ben
\Ce_{i,B}^{n}= \big\{E\in\Ce_h: \; E \not\subset\partial\Omega^n_{h,i}
\;\; \hbox{and}\;\; \exists K\in \Ct^n_{h,B}\;\;
\hbox{s.t.}\;\; E\subset\partial K\big\}.
\een
The sets are illustrated in Fig~\ref{fig:test}.
Now we make two mild assumptions on the mesh.
\vspace{1mm}

\begin{assumption}\label{ass-2}
	We assume that
	\begin{itemize}
		\item each element in $\Ct^n_{h,i}\backslash\Ct^n_{h,B}$ has at most two edges on  $\partial\omega^n_{h,i}$ for $i=1,2$, and that
		\vspace{1mm}
		
		\item there exists an integer $I>0$ such that, for $i=1,2$ and any $K\in\Ct_{h,B}^n$,
		one can find at most $I$ elements $\big\{K_j^{(i)}\big\}_{j=1}^{I}\subset\Ct^n_{h,i}$ (see Fig.~\ref{fig:A3}),
		which satisfy
		\ben
		K^{(i)}_{1}=K,\qquad K^{(i)}_{I}\subset\Omega^n_{\eta,i},\qquad
		K^{(i)}_{j}\cap K^{(i)}_{j+1}\in \Ce^n_{i,B}.
		\een
	\end{itemize}
\end{assumption}
\vspace{1mm}

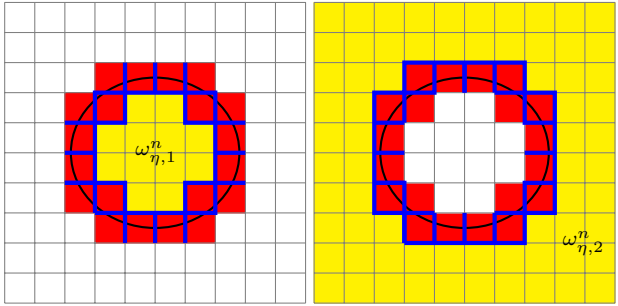
\begin{figure}[htbp]
	\centering
	\begin{tikzpicture}[scale =2.] \filldraw[red](0.2*3,0.2*2)--(0.2*3,0.2*3)--(0.2*2,0.2*3)--(0.2*2,0.2*7)--(0.2*3,0.2*7)--(0.2*3,0.2*8)--(0.2*4,0.2*8)--(0.2*7,0.2*8)--(0.2*7,0.2*7)--(0.2*8,0.2*7)--(0.2*8,0.2*3)--(0.2*7,0.2*3)--(0.2*7,0.2*2)--(0.2*3,0.2*2);  		\filldraw[yellow](0.2*3,0.2*6)--(0.2*4,0.2*6)--(0.2*4,0.2*7)--(0.2*5,0.2*7)--(0.2*6,0.2*7)--(0.2*6,0.2*6)--(0.2*7,0.2*6)--(0.2*7,0.2*5)--(0.2*7,0.2*4)--(0.2*6,0.2*4)--(0.2*6,0.2*3)--(0.2*5,0.2*3)--(0.2*4,0.2*3)--(0.2*4,0.2*4)--(0.2*3,0.2*4)--(0.2*3,0.2*5)--(0.2*3,0.2*6);
		\draw[black, thick] (1,1) ellipse [x radius=0.56cm, y radius=0.5cm];
		\filldraw[step =0.2cm,gray,thin] (0,0) grid (2cm,2cm);
		\draw [blue, ultra thick] (0.2*3,0.2*3)--(0.2*3,0.2*7)--(0.2*7,0.2*7)--(0.2*7,0.2*3)--(0.2*3, 0.2*3);
		\draw [blue,ultra thick] (0.2*2,0.2*4)--(0.2*3,0.2*4);
		\draw [blue, ultra thick] (0.2*2,0.2*5)--(0.2*3,0.2*5);
		\draw [blue,ultra thick] (0.2*2,0.2*6)--(0.2*3,0.2*6);
		\draw [blue,ultra thick] (0.2*2,0.2*4)--(0.2*3,0.2*4);
		\draw [blue,ultra thick] (0.2*4,0.2*2)--(0.2*4,0.2*4)--(0.2*3, 0.2*4);
		\draw [blue,ultra thick] (0.2*4,0.2*8)--(0.2*4,0.2*6)--(0.2*3, 0.2*6);
		\draw [blue,ultra thick] (0.2*6,0.2*2)--(0.2*6,0.2*4)--(0.2*8, 0.2*4);
		\draw [blue,ultra thick] (0.2*6,0.2*8)--(0.2*6,0.2*6)--(0.2*8, 0.2*6);
		\draw [blue,ultra thick] (0.2*5,0.2*2)--(0.2*5,0.2*3);
		\draw [blue,ultra thick] (0.2*7,0.2*5)--(0.2*8,0.2*5);
		\draw [blue,ultra thick] (0.2*5,0.2*7)--(0.2*5,0.2*8);
		\node[left] at (0.2*6,0.2*5) {$\omega_{\eta,1}^n$};
	\end{tikzpicture}\label{fig:sub1}
	\begin{tikzpicture}[scale =2.]
		\filldraw[yellow](0,0)--(0.2*10,0)--(0.2*10,0.2*10)--(0,0.2*10)--(0,0);
		\filldraw[red](0.2*3,0.2*2)--(0.2*3,0.2*3)--(0.2*2,0.2*3)--(0.2*2,0.2*7)--(0.2*3,0.2*7)--(0.2*3,0.2*8)--(0.2*4,0.2*8)--(0.2*7,0.2*8)--(0.2*7,0.2*7)--(0.2*8,0.2*7)--(0.2*8,0.2*3)--(0.2*7,0.2*3)--(0.2*7,0.2*2)--(0.2*3,0.2*2);  		
		\filldraw[white](0.2*3,0.2*6)--(0.2*4,0.2*6)--(0.2*4,0.2*7)--(0.2*5,0.2*7)--(0.2*6,0.2*7)--(0.2*6,0.2*6)--(0.2*7,0.2*6)--(0.2*7,0.2*5)--(0.2*7,0.2*4)--(0.2*6,0.2*4)--(0.2*6,0.2*3)--(0.2*5,0.2*3)--(0.2*4,0.2*3)--(0.2*4,0.2*4)--(0.2*3,0.2*4)--(0.2*3,0.2*5)--(0.2*3,0.2*6);
		\draw[black, thick] (1,1) ellipse [x radius=0.56cm, y radius=0.5cm];
		\filldraw[step =0.2cm,gray,thin] (0,0) grid (2cm,2cm);
		\draw[blue,ultra thick](0.2*3,0.2*2)--(0.2*7,0.2*2)--(0.2*7,0.2*3)--(0.2*8,0.2*3)--(0.2*8,0.2*7)--(0.2*7,0.2*7)--(0.2*7,0.2*8)--(0.2*3,0.2*8)--(0.2*3,0.2*7)--(0.2*2,0.2*7)--(0.2*2,0.2*3)--(0.2*3,0.2*3)--(0.2*3,0.2*2);
		\draw [blue,ultra thick] (0.2*4,0.2*8)--(0.2*4,0.2*7)--(0.2*3,0.2*7)--(0.2*3,0.2*6);
		\draw [blue,ultra thick] (0.2*2,0.2*4)--(0.2*3,0.2*4);
		\draw [blue, ultra thick] (0.2*2,0.2*5)--(0.2*3,0.2*5);
		\draw [blue,ultra thick] (0.2*2,0.2*6)--(0.2*3,0.2*6);
		\draw [blue,ultra thick] (0.2*2,0.2*4)--(0.2*3,0.2*4)--(0.2*3,0.2*3)--(0.2*4,0.2*3);
		\draw [blue,ultra thick] (0.2*6,0.2*7)--(0.2*7,0.2*7)--(0.2*7, 0.2*6)--(0.2*8,0.2*6);
		\draw [blue,ultra thick] (0.2*6,0.2*3)--(0.2*7,0.2*3)--(0.2*7, 0.2*4)--(0.2*8,0.2*4);
		\draw [blue,ultra thick] (0.2*6,0.2*8)--(0.2*6,0.2*7);
		\draw [blue,ultra thick] (0.2*5,0.2*8)--(0.2*5,0.2*7);
		\draw [blue,ultra thick] (0.2*4,0.2*2)--(0.2*4,0.2*3);
		\draw [blue,ultra thick] (0.2*3,0.2*2)--(0.2*3,0.2*3);
		\draw [blue,ultra thick] (0.2*5,0.2*2)--(0.2*5,0.2*3);
		\draw [blue,ultra thick] (0.2*6,0.2*2)--(0.2*6,0.2*3);
		\draw [blue,ultra thick] (0.2*7,0.2*5)--(0.2*8,0.2*5);
		\draw [blue,ultra thick] (0.2*5,0.2*7)--(0.2*5,0.2*8);	
		\node[right] at (0.2*8,0.2*2) {$\omega_{\eta,2}^n$};
	\end{tikzpicture}\label{fig:sub2}
	\caption{(I) $\Ct_{h,B}^n$: the set of red squares.
		(II) $\Ce_{1,B}^n$ (blue edges in the left figure) and $\Ce_{2,B}^n$ (blue edges in the right figure).
		(III) Left figure: $\Omega_{h,1}^n$ (red and yellow squares) and  $\omega_{h,1}^n$ (yellow squares).
		(IV) Right figure: $\Omega_{h,2}^n$ (red and yellow squares) and  $\omega_{h,2}^n$ (yellow squares).} \label{fig:test}
\end{figure}

\vspace{1mm}

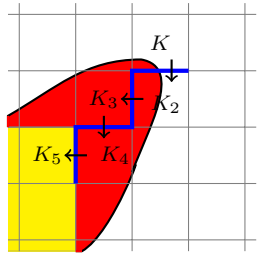
\begin{figure}
	\centering
	\begin{tikzpicture}[scale =1.5]
		\draw [thick,fill=red] (-0.1,0.5*2)--(-0.1,0.5*2+0.1) to [out=30,in=180](0.5*2+0.08,0.5*3+0.1)
		to [out=-10, in=70] (1,0.6) to [out=67,in=30] (0.556,-0.1)--(0.5,-0.1);
		\draw [gray,fill=yellow](-0.1,0.5*2)--(0.5*1,0.5*2)--(0.5*1,-0.1)--(-0.1,-0.1);
		\draw [yellow] (0.5,-0.1)--(-0.1,-0.1);
		\node[align=right] at (1.25,1.75) {$K$};
		\node[align=right] at (1.3,1.2) {$K_2$};
		\node[align=right] at (0.75,1.25) {$K_3$};
		\node[align=right] at(0.85,0.75) {$K_4$};
		\draw [step =0.5cm,gray,thin] (0,0) grid (2cm,2cm);
		\draw [gray](0,0)--(-0.1,0);
		\draw [gray](0,0.5*1)--(-0.1,0.5*1);
		\draw [gray](0,0.5*2)--(-0.1,0.5*2);
		\draw [gray](0,0.5*3)--(-0.1,0.5*3);
		\draw [gray](0,0.5*4)--(-0.1,0.5*4);
		\draw [gray](0,0.5*4)--(0,0.5*4+0.1);
		\draw [gray](0.5*1,0.5*4)--(0.5*1,0.5*4+0.1);		
		\draw [gray](0.5*2,0.5*4)--(0.5*2,0.5*4+0.1);
		\draw [gray](0.5*3,0.5*4)--(0.5*3,0.5*4+0.1);
		\draw [gray](0.5*4,0.5*4)--(0.5*4,0.5*4+0.1);
		\draw [gray](0.5*4,0)--(0.5*4+0.1,0);
		\draw [gray](0.5*4,0.5*1)--(0.5*4+0.1,0.5*1);
		\draw [gray](0.5*4,0.5*2)--(0.5*4+0.1,0.5*2);
		\draw [gray](0.5*4,0.5*3)--(0.5*4+0.1,0.5*3);
		\draw [gray](0.5*4,0.5*4)--(0.5*4+0.1,0.5*4);
		\node[align=left] at(0.25,0.75) {$K_5$};
		\draw[blue, ultra thick] (0.5*3,0.5*3)--(0.5*2,0.5*3)--(0.5*2,0.5*2)--(0.5*1,0.5*2)--(0.5,0.5*1);
		\draw[thick,->](0.75,1.1)--(0.75,0.9);
		\draw[thick,->](0.6,0.75)--(0.4,0.75);
		\draw[red,thick] (0.556,-0.1)--(0.5,-0.1);
		\draw[red,thick] (-0.1,1.1)--(-0.1,0.5*2);
		\draw[thick,->] (1.35,1.6)--(1.35,1.4);
		\draw[thick,->] (1.1,1.25)--(0.9,1.25);
		\draw[gray](0,0)--(0,-0.1);
		\draw[gray](0.5*2,0)--(0.5*2,-0.1);
		\draw[gray](0.5*3,0)--(0.5*3,-0.1);
		\draw[gray](0.5*4,0)--(0.5*4,-0.1);
	\end{tikzpicture}
	\caption{An illustration of Assumption~\ref{ass-2} with $I=5$.}
	\label{fig:A3}
\end{figure}


\subsection{Finite element spaces}

Let $Q_k$ be the space of polynomials whose degrees are no more than $k$ for each variable. The finite element spaces with and without homogeneous boundary conditions are defined over $\Ct_h$ as
\begin{align*}
	V_{h}(k,\Omega)&= \left\{q\in \Hone: q|_K\in Q_{k}(K), 
	\forall\, K\in \Ct_h\right\},\\
	V_{h,0}(k,\Omega)&= V_{h}(k,\Omega)\cap\zbHone.
\end{align*}
Correspondingly, $\BV_{h}(k,\Omega)$ and $\BV_{h,0}(k,\Omega)$ denote finite element spaces of vector-valued functions.
For a subdomain $D\subset\Omega$, we write
\ben
\BV_h(k,D) = \BV_{h}(k,\Omega)|_{D},\qquad
\BV_{h,0}(k,D) = \BV_{h,0}(k,\Omega)|_{D}.
\een
The spaces for the discrete velocity and the discrete pressure are, respectively, defined as
\begin{align*}
	\boldsymbol{\Cv}^n_h:=\,& \big\{(\Bv_{h,1},\Bv_{h,2}): \Bv_{h,1}\in \BV_{h}(k,\Omega^n_{h,1}), \;
	\Bv_{h,2}\in \BV_{h,0}(k,\Omega_{h,2}^n)\big\},\\
	\mathcal{Q}^n_h:=\,& \big\{(q_{h,1},q_{h,2}): q_{h,i}\in
	V_h(k-1,\Omega^n_{h,i}), \; i=1,2,\;
	\sum_{i=1,2}(q_{h,i},\nu_i^{-1})_{\Omega^n_{\eta,i}}=0\big\}.
\end{align*}
As a convention, we extend each $v_{h,i}\in V_h(k,\Omega^n_{h,i})$ to the exterior of $\Omega^n_{h,i}$ such that the extension, denoted still by $v_{h,i}$,
belongs to $V_{h}(k,\Omega)$
and vanishes at all degrees of freedom outside $\ol{\Omega^n_{h,i}}$.
It is easy to see that
\begin{equation}\label{ieq:vh-ext}
	\N{v_{h,i}}_{H^m(\Omega)}\lesssim \N{v_{h,i}}_{H^m(\Omega^n_{h,i})},
	\quad m=0,1,\quad i=1,2.
\end{equation}

\subsection{The discrete problem}

Following Hansbo, Larson, and Zahedi \cite{han14}, we define the ``harmonic weights''
\ben
\kappa_1 = \nu_2/(\nu_1+\nu_2),\qquad \kappa_2 =\nu_1/(\nu_1+\nu_2).
\een
Let $\ul\Ba=(\Ba_1,\Ba_2)$ be a pair of vectoral (or scalar) functions satisfying $\Omega^n_{\eta,i} \subset \Dom{\Ba_i}$ for $i=1,2$. The magic formula shows
\begin{equation*}
	\jump{\ul \Ba\,\ul \Bb} = \jump{\ul \Ba}\avg{\ul \Bb} + \jump{\ul \Bb}\avgg{\ul \Ba},
\end{equation*}
where
the jump and  averages of $\ul\Ba$ across $\Gamma^n_\eta$ are, respectively, defined by
\ben
\jump{\ul\Ba} = \Ba_1-\Ba_2,\quad \avg{\ul\Ba} = \kappa_1 \Ba_1 + \kappa_2 \Ba_2,
\quad \avgg{\ul \Ba}=\kappa_2 \Ba_1 +\kappa_1 \Ba_2 \quad  \hbox{on}\;\; \Gamma^n_\eta.
\een
We abuse the notation and define
$\nu = \nu_i$ in $\Omega_{\eta,i}^n$, $i=1,2$, for all $0\leq n\leq N$.
It is easy to see that
\begin{align}\label{avg-nu}
	\kappa_1\nu_1=\kappa_2\nu_2 = \nu_1\nu_2/(\nu_1+\nu_2) = \avg{\nu}/2.
\end{align}
For any edge $E\in\Ce_h$, let $\Bn_E$ be the unit normal on $E$ and define the jump of $v$ across $E$ by
\ben
\jump{v}(\Bx) =\lim_{\varepsilon\to 0+}
\left[v(\Bx-\varepsilon\Bn_E) - v(\Bx+\varepsilon\Bn_E)\right],
\quad \forall\,\Bx\in E.
\een
For a function $\alpha\in L^\infty(\Omega^n_\eta)$, 
we define 
the weighted inner product of $\ul\Ba$ and $\ul\Bb$ as
\ben
\Aprod[\Omega^n_\eta]{\alpha\ul\Ba,\ul\Bb}
= (\alpha\Ba_1,\Bb_1)_{\Omega_{\eta,1}^n}
+(\alpha\Ba_2,\Bb_2)_{\Omega_{\eta,2}^n} .
\een
Let $\Delta_h$ denote the discrete Laplacian operator on piecewise regular functions,
\ben
\Delta_h\Bv_h\in\Ltwov\;\;\hbox{satisfying}\;\;
(\Delta_h\Bv_h)\big|_K \equiv \Delta(\Bv_h|_K),\quad
\forall\, K\in\Ct_h.
\een
It is clear that $\Delta_h\Bv=\Delta\Bv$ if $\Bv\in\BH^2(\Omega)$.
We also use the symbols and write
\ben
\Div\ul{\Bv_h}=(\Div\Bv_{h,1},\Div\Bv_{h,2}),\;
\nabla\ul{\Bv_h}=(\nabla\Bv_{h,1},\nabla\Bv_{h,2}),\;
\Delta_h\ul{\Bv_h}=(\Delta_h\Bv_{h,1},\Delta_h\Bv_{h,2}).
\een

Now we introduce the fully discrete approximation to the semi-discrete problem \eqref{eq:semi-disc}. Suppose that the discrete solutions $\Bu^m_h\in{\boldsymbol{\Cv}}^m_h$ are obtained for $n-k\le m< n$. Define
\begin{align*}
	\ul{\BU_{h}^{m,n}}:=\ul{\Bu_{h}^{m}}\circ\BX_\tau^{n,m},\quad \ul{\bbU^n_h}=\big[\ul{\BU^{n-k,n}_h},\cdots,\ul{\BU^{n,n}_h}\big],\quad
	\Lambda^k\ul{\bbU^n_h} = \sum_{j=0}^k \lambda_j^k \ul{\BU^{n-j,n}_h}.
\end{align*}
\begin{center}
	\fbox{\parbox{0.975\textwidth}
		{The fully discrete problem reads: Find $\big(\ul{\Bu^n_h},\ul{p^n_h}\big) \in{\boldsymbol{\Cv}}^n_h\times\Cq^n_h$ such that
			for any $(\ul{\Bv_h},\ul{q_h})\in{\boldsymbol{\Cv}}^n_h\times\Cq^n_h$,
			\begin{subequations} \label{eq:disc}
				\begin{align}
					&\frac{1}{\tau}\Aprod[\Omega^n_\eta]{\Lambda^k \ul{\bbU_h^n},\ul{\Bv_h}}
					+\mathscr{A}_h^n(\ul{\Bu_h^n},\ul{\Bv_h})
					+\mathscr{B}^n_0(\ul{\Bv_h},\ul{p_h^n})
					= \Aprod[\Omega^n_\eta]{\ul\Bf^n,\ul{\Bv_h}}, \label{eq:disc-u}\\
					&\mathscr{B}^n_0(\ul{\Bu_h^n},\ul{q_h}) - \mathscr{J}_p^n(\ul{p_h^n},\ul{q_h})-
					\mathscr{R}_h^n(\ul{\bbU_h^n},\ul{p_h^n};\ul{q_h}) = 0.  \label{eq:disc-p}
				\end{align}
	\end{subequations}}}
\end{center}\vspace{2mm}
Here $\ul{\Bf}^n=\big(\Bf^n_1,\Bf^n_2\big)$ and $\mathscr{R}_h^n$ is the residual functional from the semi-discrete momentum equation
\begin{align*}
	\mathscr{R}_h^n(\ul\bbU_h^n,\ul{p_h^n};\ul{q_h}) =\,&\gamma_1\nu_2 h^2
	\Aprod[\Omega^n_\eta]{\tau^{-1} \Lambda^k\ul{\bbU_h^n}- \nu \Delta_h \ul{\Bu_h^n}
		+\nabla \ul{p_h^n}-\ul{\Bf}^n,\nu^{-1}\nabla \ul{q_h}},
\end{align*}
where $\gamma_1\in (0,1)$.
The bilinear forms in \eqref{eq:disc} are defined as follows
\begin{align*}
	\mathscr{A}_h^n(\ul{\Bu_h},\ul{\Bv_h})=\,&
	\Aprod[\Omega_{\eta}^n]{\nu\nabla\ul{\Bu_h}, \nabla\ul{\Bv_h}}
	-\mathscr{F}^n(\ul{\Bu_h},\ul{\Bv_h})
	+\mathscr{J}_0^n(\ul{\Bu_h},\ul{\Bv_h})
	+\mathscr{J}_{\Bu}^n(\ul{\Bu_h},\ul{\Bv_h}), \\
	\mathscr{F}^n(\ul{\Bu_h},\ul{\Bv_h})=\,&
	\int_{\Gamma_\eta^n}\left(
	\Avg{\nu\partial_{\Bn}\ul{\Bu_h}}\cdot \jump{\ul{\Bv_h}}
	+\Avg{\nu \partial_{\Bn} \ul{\Bv_h}}\cdot\jump{\ul{\Bu_h}}\right),\\
	\mathscr{J}_0^n(\ul{\Bu_h},\ul{\Bv_h}) =\,& \gamma_0 \avg{\nu}h^{-1}
	\int_{\Gamma_\eta^n} \jump{\ul{\Bu_h}}\cdot\jump{\ul{\Bv_h}}, \\
	\mathscr{B}_0^n(\ul{\Bv_h},\ul{q_h})
	=\,&-\Aprod[\Omega^n_\eta]{\Div \ul{\Bv_h}, \ul{q_h}}
	+ \int_{\Gamma_\eta^n}\big(\jump{\ul{\Bv_h}}\cdot\Bn\big) \Avg{\ul{q_h}}, \\
	\mathscr{J}_{\Bu}^n(\ul{\Bu_h},\ul{\Bv_h})=\,&
	\sum_{i=1}^2 \sum_{E\in \Ce_{i,B}^{n}}
	\sum_{l=1}^k h^{2l-1}/((l-1)!)^2\int_E \nu_i
	\jump{\partial_{\Bn}^l \Bu_{h,i}} \cdot
	\jump{\partial_{\Bn}^l \Bv_{h,i}}, \\
	\mathscr{J}^n_p(\ul{p_h},\ul{q_h}) =\,&
	\sum_{i=1}^2\sum_{E\in \Ce_{i,B}^{n}}
	\sum_{l=1}^{k-1} h^{2l+1}/(l!)^2\int_E \nu_i^{-1}
	\jump{\partial_{\Bn}^l p_{h,i}}
	\jump{\partial_{\Bn}^l q_{h,i}},
\end{align*}
where $\partial_{\Bn} \ul{\Bv_h}\equiv \big(\partial_{\Bn}\Bv_{h,1},\partial_{\Bn}\Bv_{h,2}\big)$ and $\partial_\Bn^l\Bv_{h,i}$ stands for the $l^{\rm th}$-order normal derivative of $\Bv_{h,i}$. The parameters $\gamma_0,\gamma_1$ are positive and independent of $\tau$, $h$, and $\eta$.
\vspace{1mm}

\begin{remark}
	The interface-zone penalties $\mathscr{J}_{\Bu}^n$ and $\mathscr{J}_p^n$ are used to enhance the stability of the discrete velocity and the discrete pressure, respectively. The residual-based bilinear form $\mathscr{R}_h^n$ is favorable to proving optimal $H^1$-error estimate for the discrete pressure.
\end{remark}
\vspace{1mm}

\begin{remark}
	The discrete scheme \eqref{eq:disc} is consistent with the semi-discrete problem \eqref{eq:semi-disc}. In fact, let $\ul\Bu^n=(\Bu^n_1,\Bu^n_2)$, $\ul p^n=(p^n_1,p^n_2)$ be the solution to \eqref{eq:semi-disc} and write
	\ben
	\ul{\BU^{m,n}}:=\big(\Bu_1^{m}\circ\BX_\tau^{n,m}, \Bu_2^{m}\circ\BX_\tau^{n,m}\big),
	\qquad
	\ul{\bbU^n}=\big[\ul{\BU^{n-k,n}},\cdots,\ul{\BU^{n,n}}\big].
	\een
	Then \eqref{eq:disc} still holds if we replace $\big(\ul{\Bu^n_h}, \ul{p^n_h}, \ul{\bbU^n_h}\big)$ with
	$\big(\ul\Bu^n,\ul p^n,\ul{\bbU^n}\big)$.
\end{remark}

\section{The well-posedness of discrete problems}

The purpose of this section is to prove the well-posedness of \eqref{eq:disc} in each time step.
First we introduce some product spaces and mesh-dependent norms.

\subsection{Mesh-dependent norms}

For a positive function $\alpha\in L^\infty(\Omega)$ and an integer $m\ge 0$, we define the norms on $H^m(\Omega_{\eta,1}^n)\times H^m(\Omega_{\eta,2}^n)$ and
$L^2(\Omega_{\eta,1}^n)\times L^2(\Omega_{\eta,2}^n)$, respectively, as
\begin{align*}
	\N{\ul{v}}_{m,\Omega_\eta^n} =\Big(\sum_{i=1,2}\N{v_i}^2_{H^m(\Omega^n_{\eta,i})}\Big)^{\frac12},
	\qquad
	\N{\alpha\ul{v}}_{0,\Omega_\eta^n} =\Big(\sum_{i=1,2}
	\N{\alpha v_i}^2_{L^2(\Omega^n_{\eta,i})}\Big)^{\frac12}.
\end{align*}
For convenience, we also introduce two product spaces
\begin{align*}
	\boldsymbol{\Cv}^n =\,& \big\{\ul\Bv\in \Honev[\Omega_{h,1}^n]
	\times \Honev[\Omega_{h,2}^n]: \Bv_2|_{\Sigma} =0,\;\hbox{and}\;
	\partial_{\Bn}\Bv_1, \partial_{\Bn}\Bv_2\in\BL^2(\Gamma_\eta^n)\big\},\\
	\Cq^n =\,& \big\{\ul{q}\in\Ltwo[\Omega_{h,1}^n]\times\Ltwo[\Omega_{h,2}^n]:
	\sum_{i=1}^2(q_{i},\nu_i^{-1})_{\Omega_{\eta,i}}=0,\;
	\hbox{and}\; q_1,q_2\in\Ltwo[\Gamma_\eta^n]\big\}.
\end{align*}
Clearly $\boldsymbol{\Cv}^n_h\subset \boldsymbol{\Cv}^n$ and $\Cq^n_h\subset \Cq^n$.
The mesh-dependent norms on ${\boldsymbol{\Cv}}^n$ and $\Cq^n$  are defined as
\begin{align*}
	\N{\ul{\Bv}}_{{\boldsymbol{\Cv}}} =\,& \Big\{\big\|\nu^{\frac12}\nabla\ul{\Bv}\big\|_{0,\Omega^n_\eta}^2
	+\mathscr{J}_0^n(\ul{\Bv},\ul{\Bv})
	+h\avg{\nu}^{-1}\NLtwov[\Gamma_\eta^n]{\avg{\nu\partial_{\Bn}\ul{\Bv}}}^2\Big\}^{\frac12},\\
	\N{\ul{q}}_{\Cq} =\,& \Big\{\big\|\nu^{-\frac12}\ul{q}\big\|_{0,\Omega^n_\eta}^2
	+h\avg{\nu}^{-1}\N{\Avg{\ul{q}}}_{L^2(\Gamma_{\eta}^n)}^2\Big\}^{\frac12}.
\end{align*}
We define three more norms for piecewise regular functions
\begin{align}
	&\tN{\ul{\Bv}}_{{\boldsymbol{\Cv}}} =\Big\{
	\big\|\nu^{\frac12}\nabla\ul{\Bv}\big\|_{0,\Omega^n_\eta}^2
	+\mathscr{J}_0^n(\ul{\Bv},\ul{\Bv})
	+\mathscr{J}_{\Bu}^n(\ul{\Bv},\ul{\Bv})\Big\}^{\frac12},
	\label{eq:norm-V1} \\
	&\tN{\ul{q}}_{\Cq} = \Big\{\big\|\nu^{-\frac12}\ul{q}\big\|_{0,\Omega^n_\eta}^2
	+\mathscr{J}_p^n(\ul{q},\ul{q})\Big\}^{\frac12}, \label{eq:norm-Q0} \\
	&\tN{\ul{q_h}}_{1,\Cq}=\Big\{\tN{\ul{q_h}}_{\Cq}^2 +\sum_{i=1,2}
	\dN{\nu^{-\frac12}\nabla q_{h,i}}_{L^2(\Omega^n_{h,i})}^2\Big\}^{\frac12}.
	\label{eq:norm-Q1}
\end{align}
They induce the product norm
\ben
\tN{(\ul{\Bv},\ul{q})}_{\boldsymbol{\Cv},\Cq} =\big(
\tN{\ul{\Bv}}_{{\boldsymbol{\Cv}}}^2 +\tN{\ul{q}}_{\Cq}^2\big)^{\frac12}.
\een

\begin{lemma}\label{lem:trace2}
	Let Assumption~\ref{ass-2} be satisfied.
	For any  $\ul{\Bv_h}\in {\boldsymbol{\Cv}}^n_h$ and $\ul{q_h}\in \Cq^n_h$, there hold
	\begin{align}
		h\avg{\nu}^{-1}\N{\Avg{\nu\partial_{\Bn}\ul{\Bv_h}}}_{\BL^2(\Gamma_\eta^n)}^2
		\lesssim\,& \big\|\nu^{\frac12}\nabla\ul{\Bv_h}\big\|_{0,\Omega^n_\eta}^2
		+\mathscr{J}_{\Bu}^n(\ul{\Bv_h},\ul{\Bv_h}), \label{eq:trace2}\\
		h\avg{\nu}^{-1}\N{\Avg{\ul{q_h}}}_{\BL^2(\Gamma_\eta^n)}^2
		\lesssim\,& \tN{\ul{q_h}}_{\Cq}^2. \label{eq:trace3}
	\end{align}
\end{lemma}
\begin{proof}
	Thanks to Assumption~\ref{ass-2}, we can cite \cite[Lemma~5.2]{ma21} for the inequalities
	\begin{align}
		\N{\nabla^{\mu}\Bv_{h,i}}_{\BL^2(\Omega^n_{h,i})}^2
		\lesssim\,& \N{\nabla^{\mu}\Bv_{h,i}}_{\BL^2(\omega_{h,i}^n)}^2
		+\sum_{E\in \Ce_{i,B}^n}\sum_{l=1}^k \frac{h^{2l+1-2\mu}}{((l-1)!)^2}
		\int_{E} \SN{\jump{\partial_{\Bn}^l \Bv_{h,i}}}^2,
		\label{norm-equal-v}\\
		\N{q_{h,i}}_{L^2(\Omega^n_{h,i})}^2
		\lesssim\,& \N{q_{h,i}}_{L^2(\omega_{h,i}^n)}^2
		+\sum_{E\in \Ce_{i,B}^n}\sum_{l=1}^{k-1} \frac{h^{2l+1}}{(l!)^2}
		\int_{E} \SN{\jump{\partial_{\Bn}^l q_{h,i}}}^2,
		\label{norm-equal-q}
	\end{align}
	where $\mu =0,1$ and $i=1,2$. The definitions of $\mathscr{J}_{\Bu}^n$ and $\mathscr{J}^n_p$ show that
	\begin{align}
		\sum_{i=1,2}\nu_i\N{\nabla^{\mu} \Bv_{h,i}}_{\BL^2(\Omega^n_{h,i})}^2
		\lesssim\,& \sum_{i=1,2}\nu_i \N{\nabla^{\mu}\Bv_{h,i}}_{\BL^2(\omega_{h,i}^n)}^2 +h^{2-2\mu}\mathscr{J}_{\Bu}^n(\ul{\Bv_h},\ul{\Bv_h}),
		\label{normEqual1}\\
		\sum_{i=1,2}\nu_i^{-1}\N{q_{h,i}}_{L^2(\Omega^n_{h,i})}^2
		\lesssim\,& \sum_{i=1,2}\nu_i^{-1}\N{q_{h,i}}_{L^2(\omega_{h,i}^n)}^2
		+\mathscr{J}_p^n(\ul{q_h},\ul{q_h}). \label{normEqual2}
	\end{align}

	For any $K\in\Ct^n_{h,B}$ and $\Gamma_K=K\cap\Gamma^n_\eta$, we cite \cite{guz18} for the trace inequality
	\begin{align}	\label{ieq:trace0}
		\NLtwo[\partial K]{v} + \NLtwo[\Gamma_K]{v}
		\lesssim h^{-\frac12}\NLtwo[K]{v}+h^{\frac12}\SNHone[K]{v},
		\; \forall\,v\in H^1(K).
	\end{align}
	It follows from \eqref{avg-nu} and the triangular inequality that
	\begin{align}
		&h\avg{\nu}^{-1}\N{\Avg{\nu\partial_{\Bn}\ul{\Bv_h}}}_{\BL^2(\Gamma_K)}^2
		\lesssim h\sum_{i=1,2}\kappa_i\nu_i \SNHonev[\Gamma_K]{\Bv_{h,i}}^2
		\lesssim \sum_{i=1,2}\nu_i \SN{\Bv_{h,i}}_{\BH^1(K)}^2, \label{normEqual4} \\
		&h\avg{\nu}^{-1}\N{\Avg{\ul{q_h}}}_{L^2(\Gamma_K)}^2
		\lesssim h \sum_{i=1,2}\kappa_i\nu_i^{-1}\N{q_{h,i}}_{L^2(\Gamma_K)}^2
		\lesssim \sum_{i=1,2}\nu_i^{-1} \N{q_{h,i}}_{L^2(K)}^2.\label{normEqual5}
	\end{align}
	Taking the sum of \eqref{normEqual4} over all $K\in \Ct_{h,B}^n$ and using \eqref{normEqual1} yield  \eqref{eq:trace2}. Taking the sum of \eqref{normEqual5} over all $K\in \Ct_{h,B}^n$ and use \eqref{normEqual2} yield \eqref{eq:trace3}.
\end{proof}
\vspace{1mm}

\begin{lemma}\label{lem:ch-up}
	For any $\ul{\Bv_h}\in {\boldsymbol{\Cv}}^n_h$ and $\ul{q_h}\in \Cq^n_h$, there hold
	\begin{align*}
		&h^2\sum_{i=1,2}\nu_i\NLtwov[\Omega^n_{h,i}]{\Delta_h\Bv_{h,i} }^2
		\lesssim \N{\nu^{\frac12}\nabla\ul{\Bv_h}}_{0,\Omega^n_\eta}^2
		+\mathscr{J}_{\Bu}^n(\ul{\Bv_h},\ul{\Bv_h}),\\
		&h^2\sum_{i=1,2}\nu_i^{-1}\NLtwo[\Omega^n_{h,i}]{\nabla q_{h,i}}^2
		\lesssim \tN{\ul{q_h}}_{\Cq}^2.
	\end{align*}
\end{lemma}
\begin{proof}
	The lemma is a direct result of \eqref{normEqual1}--\eqref{normEqual2} and inverse estimates.
\end{proof}
\vspace{1mm}

\begin{corollary}
	For any $\ul{\Bv_h}\in {\boldsymbol{\Cv}}^n_h$ and $\ul{q_h}\in \Cq^n_h$, there hold
	\begin{align}
		\N{\ul{\Bv_h}}_{{\boldsymbol{\Cv}}}^2
		\lesssim \tN{\ul{\Bv_h}}_{{\boldsymbol{\Cv}}}^2 &\eqsim
		\sum\limits_{i=1,2}\nu_i \SNHonev[\Omega_{h,i}^n]{\Bv_{h,i}}^2
		+ \mathscr{J}^n_0(\ul{\Bv_h},\ul{\Bv_h}),\label{norm-equi0}\\
		\N{\ul{q_h}}^2_{\Cq} \lesssim \tN{\ul{q_h}}_{\Cq}^2
		&\eqsim\sum_{i=1}^2\nu_i^{-1}\N{q_{h,i}}^2_{L^2(\Omega_{h,i}^n)}.\label{eq:qh eqsim}
	\end{align}
\end{corollary}
\begin{proof}
	The inequalities $\N{\ul{q_h}}_{\Cq} \lesssim \tN{\ul{q_h}}_{\Cq}$ and  $\N{\ul{\Bv_h}}_{{\boldsymbol{\Cv}}}
	\lesssim \tN{\ul{\Bv_h}}_{{\boldsymbol{\Cv}}}$ follow directly from Lemma~\ref{lem:trace2}.
	For $i=1,2$ and $1\le l\le k$, using inverse estimates, we have
	\begin{equation*}
		h^{2l-1}\int_{E}\SN{\jump{\partial_\Bn^l \Bv_{h,i}}}^2
		\lesssim h^{2l-2}\N{\nabla^l\Bv_{h,i}}_{\BL^{2}(K_1\cup K_2)}^2
		\lesssim\N{\nabla\Bv_{h,i}}_{\BL^{2}(K_1\cup K_2)}^2,\;
	\end{equation*}
	where $E\in \Ce^n_{i,B}$, $K_1,K_2\in\Ct^n_{h,i}$ are the two elements sharing $E$. Together with \eqref{normEqual1}, it yields
	\begin{equation*}
		\sum_{i=1,2}\nu_i\NLtwov[\Omega_{h,i}^n]{\nabla\Bv_{h,i}}^2\eqsim
		\dN{\nu^{\frac12}\nabla\ul{\Bv_h}}_{0,\Omega^n_\eta}^2
		+ \mathscr{J}_{\Bu}^n(\ul{\Bv_h},\ul{\Bv_h}),
	\end{equation*}
	which shows
	the equivalence $\tN{\ul{\Bv_h}}_{{\boldsymbol{\Cv}}} \eqsim
	\sum\limits_{i=1,2}\nu_i \SNHonev[\Omega_{h,i}^n]{\Bv_{h,i}}^2
	+ \mathscr{J}^n_0(\ul{\Bv_h},\ul{\Bv_h})$.
	Similarly, we can prove
	\begin{equation*}
		\sum_{i=1,2}\nu_i^{-1}\NLtwo[\Omega_{h,i}^n]{q_{h,i}}^2	\eqsim
		\tN{\ul{q_h}}_{\Cq}^2 .
	\end{equation*}
	The proof is finished.
\end{proof}

\subsection{Inf-sup conditions}

Inf-sup conditions play the key role in proving the well-posedness of mixed problems. In \cite[Lemma~3.11]{han14}, Hansbo, Larson, and Azhedi proved the inf-sup condition for mixed finite element approximation to Stokes interface problem, where the velocity is discretized with linear finite elements on a triangular mesh of size $h/2$, while the pressure is discretized with linear finite elements on a triangular mesh of size $h$.
Here we follow them to establish the inf-sup condition on $\boldsymbol{\Cv}^n_h\times \Cq^n_h$.
We also refer to \cite{ols06} which studies mixed finite element method for Stokes interface problem on body-fitted meshes.

\begin{lemma} \label{lem:infsup}
	Let Assumption~\ref{ass-2} be satisfied and let $h$ be small enough. For any $\ul{q_h}\in\Cq^n_h$, there exists a function
	$\ul{\Bv_h}\equiv (\Vv_h|_{\Omega^n_{h,1}},\Vv_h|_{\Omega^n_{h,2}})$,
	$\Vv_h\in\BV_{h,0}(k,\Omega)$ satisfies such that
	\begin{align*}
		\tN{\ul{\Bv_h}}_{\boldsymbol{\Cv}}^2
		\lesssim \tN{\ul{q_h}}_{\Cq}^2 \lesssim
		\Aprod[\Omega^n_\eta]{\Div\ul{\Bv_h},\ul{q_h}}
		+ \mathscr{J}^n_p(\ul{q_h},\ul{q_h}) .
	\end{align*}
\end{lemma}
\begin{proof}
	In \cite[Theorem~3.2]{bre91}, Brezzi and Falk proved the inf-sup condition for the $Q_k$--$Q_{k-1}$ Hood-Taylor triangular finite elements.
	A careful inspection shows that the theorem holds for any rectangular mesh in which each element has at most two edges on the boundary of domain. Therefore, we have, for any $\xi_h\in V_{h}(k-1,\omega^n_{h,i})\cap L^2_0(\omega^n_{h,i})$,
	\begin{align}\label{infsup-1}
		\sup_{\Bw_h\in \BV_{h}(k,\omega^n_{h,i})\cap\zbHonev[\omega^n_{h,i}]}
		\frac{(\Div\Bw_h,\xi_h)_{\omega^n_{h,i}}}{\NHonev[\omega^n_{h,i}]{\Bw_h}}
		\gtrsim \NLtwo[\omega^n_{h,i}]{\xi_h},\quad
		\quad i=1,2.
	\end{align}
	The rest of the proof uses \eqref{infsup-1} and is parallel to the proof of \cite[Lemma~3.11]{han14} where the inf-sup condition is proven for linear finite elements on triangular meshes. Here we only sketch the proof and omit the details.
	
	{\bf Step~1.} Define $\sigma_1\equiv \nu_1\SN{\Omega^n_{\eta,1}}^{-1}$ and $\sigma_2\equiv -\nu_2\SN{\Omega^n_{\eta,2}}^{-1}$ and write $\ul\sigma= (\sigma_1,\sigma_2)$.
	Let $\ul{q_0}$ be the projection of $\ul{q_h}$ onto the one-dimensional space $\Span{\ul{\sigma}}$ under the inner product $\Aprod[\Omega^n_\eta]{\nu^{-1}\cdot,\cdot}$. Following the proof of \cite[Lemma~3.9]{han14}, we can find a function $\Vv_0\in\BV_{h,0}(k,\Omega)$ such that
	$\ul{\Bv_0}\equiv (\Vv_0|_{\Omega^n_{h,1}},\Vv_0|_{\Omega^n_{h,2}})$ satisfies
	\begin{align*}
		\tN{\ul{\Bv_0}}_{\boldsymbol{\Cv}}^2\lesssim
		\big\|\nu^{-\frac12}\ul{q_0}\big\|_{0,\Omega^n_\eta}^2
		\leq \tN{\ul{q_0}}_{\Cq}^2
		\lesssim \Aprod[\Omega^n_\eta]{\Div\ul{\Bv_0},\ul{q_0}}.
	\end{align*}

	{\bf Step~2.} For $\ul{\hat q_h}=\ul{q_h} -\ul{q_0}$, following the proof of \cite[Lemma~3.10]{han14} and using \eqref{eq:qh eqsim}, we can find a $\ul{\hat\Bv_h}\in\boldsymbol{\Cv}^n_h$ which satisfies $\hat\Bv_{h,i}=0$ in $\Omega^n_{h,i}\backslash\omega^n_{h,i}$, $i=1,2$, and
	\begin{align*}
		\tN{\ul{\hat\Bv_h}}_{\boldsymbol{\Cv}}^2\lesssim
		\tN{\ul{\hat{q}_h}}_{\Cq}^2\eqsim
		\sum_{i=1,2}\nu^{-1}_i\NLtwo[\Omega^n_{h,i}]{\hat q_{h,i}}^2
		\lesssim \Aprod[\Omega^n_\eta]{\Div\ul{\hat\Bv_h},\ul{\hat q_h}}
		+ \mathscr{J}^n_p(\ul{\hat q_h},\ul{\hat q_h}).
	\end{align*}
	Define $\hat\Vv_h=\hat\Bv_{h,1}$ in $\Omega^n_{h,1}$ and $\hat\Vv_h=\hat\Bv_{h,2}$ elsewhere. It is clear that $\hat\Vv_h\in\BV_{h,0}(k,\Omega)$.

	{\bf Step~3.} Define $\ul{\Bv_h}=\ul{\Bv_0}+\alpha \ul{\hat\Bv_h}$ and $\ul{q_h} = \ul{q_0}+ \ul{\hat q_h}$ for some $\alpha>0$. By the Cauchy-Schwarz inequality and the equality $\Aprod[\Omega^n_\eta]{\Div\ul{\hat\Bv_h},\ul{q_0}}=0$, we have
	\begin{align*}
		\Aprod[\Omega^n_\eta]{\Div\ul{\Bv_h},\ul{q_h}}
		=\,&\Aprod[\Omega^n_\eta]{\Div\ul{\Bv_0},\ul{q_0}}
		+\alpha\Aprod[\Omega^n_\eta]{\Div\ul{\hat\Bv_h},\ul{\hat q_h}}
		+\Aprod[\Omega^n_\eta]{\Div\ul{\Bv_0},\ul{\hat q_h}}\\
		\ge\,& C_0 \tN{\ul{q_0}}_{\Cq}^2
		+C_1\alpha \tN{\ul{\hat q_h}}_{\Cq}^2
		-\alpha \mathscr{J}^n_p(\ul{\hat q_h},\ul{\hat q_h})
		-C_2\tN{\ul{q_0}}_{\Cq}\tN{\ul{\hat q_h}}_{\Cq} \\
		\ge\,&\big(C_0-2C_1^{-1}C_2^2\alpha^{-1}\big) \tN{\ul{q_0}}_{\Cq}^2
		+0.5C_1\alpha\tN{\ul{\hat q_h}}_{\Cq}^2
		-\alpha \mathscr{J}^n_p(\ul{\hat q_h},\ul{\hat q_h}),
	\end{align*}
	where $C_0,C_1,C_2$ are positive constants independent of $h$, $\tau$, and $\nu$.
	Since $\mathscr{J}^n_p(\ul{\hat q_h},\ul{\hat q_h})=\mathscr{J}^n_p(\ul{q_h},\ul{q_h})$,
	choosing $\alpha=4C_2^2/(C_0C_1)$ and using \eqref{normEqual2}, we find that
	\begin{align*}
		\tN{\ul{q_h}}_{\Cq}^2\le\,&
		2\tN{\ul{q_0}}_{\Cq}^2+2\tN{\ul{\hat q_h}}_{\Cq}^2
		\lesssim\Aprod[\Omega^n_\eta]{\Div\ul{\Bv_h},\ul{q_0}}
		+\mathscr{J}^n_p(\ul{q_h},\ul{q_h}) ,\\
		\tN{\ul{\Bv_h}}_{\boldsymbol{\Cv}}^2 \le\,&
		2\tN{\ul{\Bv_0}}_{\boldsymbol{\Cv}}^2 +
		2\tN{\ul{\hat\Bv_h}}_{\boldsymbol{\Cv}}^2 \lesssim
		\big\|\nu^{-\frac12}\ul{q_0}\big\|_{0,\Omega^n_\eta}^2
		+\tN{\ul{\hat q_h}}_{\Cq}^2
		\leq \tN{\ul{q_h}}_{\Cq}^2.
	\end{align*}
	Then $\Vv_h:=\Vv_0+\alpha\hat\Vv_h\in\BV_{h,0}(k,\Omega)$ and satisfies
	$\Vv_h=\Bv_{h,i}$ in $\Omega^n_{h,i}$ for $i=1,2$.
\end{proof}
\vspace{1mm}

\begin{lemma}\label{lem:Ah}
	Suppose $\gamma_0$ is large enough. For any $\ul{\Bu_h}, \ul{\Bv_h}\in {\boldsymbol{\Cv}}_h^n$
	and $\ul{q_h}\in\Cq_h^n$,
	\begin{align*}
		\SN{\mathscr{B}^n_0(\ul{\Bv_h},\ul{p_h})}
		\lesssim\,& \N{\ul{\Bv_h}}_{\boldsymbol{\Cv}}\N{\ul{q_h}}_{\Cq},\\
		\SN{\mathscr{A}_h^{n}(\ul{\Bu_h},\ul{\Bv_h})}
		\lesssim\,& \tN{\ul{\Bu_h}}_{\boldsymbol{\Cv}}
		\tN{\ul{\Bv_h}}_{\boldsymbol{\Cv}}, \qquad
		\mathscr{A}_h^n(\ul{\Bv_h},\ul{\Bv_h}) \ge
		0.9\tN{\ul{\Bv_h}}_{\boldsymbol{\Cv}}^2.
	\end{align*}
\end{lemma}
\begin{proof}
	The continuity of $\mathscr{B}^n_0$ is obvious.
	By Lemma~\ref{lem:trace2}, it is clear that, for any $\varepsilon>0$,
	\begin{align*}
		\SN{\mathscr{F}^n(\ul{\Bv_h},\ul{\Bv_h})}
		\le\,& (\varepsilon h)^{-1}\avg{\nu}
		\N{\jump{\ul{\Bv_h}}}_{\BL^2(\Gamma_{\eta}^n)}^2
		+ \varepsilon h \avg{\nu}^{-1}
		\N{\Avg{\nu\partial_{\Bn}\ul{\Bv_h}}}_{\BL^2(\Gamma_\eta^n)}^2 \\
		\le\,& (\varepsilon^{-1}\gamma_0^{-1} +C_0\varepsilon)
		\tN{\ul{\Bv_h}}_{\boldsymbol{\Cv}}^2 .
	\end{align*}
	Choosing $\varepsilon$ small enough and $\gamma_0$ large enough such that
	$\varepsilon^{-1}\gamma_0^{-1} +C_0\varepsilon\le 0.1$, we have $\mathscr{A}_h^n\big(\ul{\Bv_h},\ul{\Bv_h}\big)
	= \tN{\ul{\Bv_h}}^2_{\boldsymbol{\Cv}}-\mathscr{F}^n(\ul{\Bv_h},\ul{\Bv_h})
	\ge 0.9\tN{\ul{\Bv_h}}^2_{\boldsymbol{\Cv}}$. The proof for the continuity of $\mathscr{A}^n_h$ is similar.
\end{proof}
\vspace{1mm}

\subsection{The well-posedness}

To favor the analysis, we rewrite \eqref{eq:disc} into an equivalent variational form: find $\big(\ul{\Bu_h},\ul{p_h}\big)\in{\boldsymbol{\Cv}}^n_h\times\Cq^n_h$ such that
\begin{align}\label{eq:disc-alpha}
	\mathscr{K}_1((\ul{\Bu_h},\ul{p_h}),(\ul{\Bv_h},\ul{q_h}))
	=\Cg(\ul{\Bv_h},\ul{q_h}),\qquad \forall\, (\ul{\Bv_h},\ul{q_h})\in{\boldsymbol{\Cv}}^n_h\times\Cq^n_h,
\end{align}
where
\begin{align*}
	&\mathscr{K}_1((\ul{\Bu_h},\ul{p_h}),(\ul{\Bv_h},\ul{q_h})) =
	\tau^{-1}\lambda_0^k\Aprod[\Omega^n_\eta]{\ul{\Bu_h},\ul{\Bv_h}}
	+\mathscr{A}_h^n(\ul{\Bu_h},\ul{\Bv_h})
	+\mathscr{B}_0^n(\ul{\Bv_h},\ul{p_h})\\ 
	&\hspace{37mm}
	-\mathscr{B}_0^n(\ul{\Bu_h},\ul{q_h})
	+\mathscr{J}_p^n(\ul{p_h},\ul{q_h})
	+ \gamma_1\nu_2 h^2 \mathscr{B}^n_1(\ul{\Bu_h},\ul{q_h})\\
	&\hspace{37mm}
	+\gamma_1\nu_2 h^2\Aprod[\Omega^n_\eta]{\nu^{-1}\nabla \ul{p_h},\nabla \ul{q_h}}, \\
	&\mathscr{B}^n_1(\ul{\Bu_h},\ul{q_h}) = \sum_{i=1,2} \sum_{K\in \Ct_h}
	\nu^{-1}_i\int_{K\cap \Omega_{\eta,i}^n}
	\left(\tau^{-1} \lambda_0^k\Bu_{h,i}- \nu_i \Delta \Bu_{h,i}\right)
	\cdot\nabla q_{h,i},\\
	&\Cg(\ul{\Bv_h},\ul{q_h}) =\sum_{i=1,2}\int_{\Omega^n_{\eta,i}}
	\Big(\Bf^n_i -\frac{1}{\tau}\sum_{j=1}^k\lambda_j^k
	\BU_{h,i}^{n-j,n}\Big)\cdot
	\left(\Bv_{h,i}-\gamma_1\nu_2 h^2\nu_i^{-1}\nabla q_{h,i}\right) .
\end{align*}

\begin{center}
	\fbox{\parbox{0.975\textwidth}{
			\begin{theorem} \label{thm:alpha}
				Suppose that $\gamma_0,\gamma_1^{-1}$ are large enough and $h=O(\tau)$. Then problem \eqref{eq:disc-alpha} has a unique solution.
	\end{theorem}}}
\end{center}
\vspace{1mm}

\begin{proof}
	Since \eqref{eq:disc-alpha} is a linear problem on finite-dimensional space, it suffices to prove
	that there exits an $\alpha_0>0$ such that for any $(\ul{\Bu_h},\ul{p_h})\in \boldsymbol{\Cv}_h^n\times\mathcal{Q}_h^n$
	\begin{align}\label{insup-K1}
		\sup_{(\ul{\Bv_h},\ul{q_h})\in \boldsymbol{\Cv}_h^n\times\mathcal{Q}_h^n}
		\frac{\mathscr{K}_1((\ul{\Bu_h},\ul{p_h}),(\ul{\Bv_h},\ul{q_h}))}
		{\tN{(\ul{\Bv_h},\ul{q_h})}_{\boldsymbol{\Cv},\Cq}}
		> \alpha_0 \tN{(\ul{\Bu_h},\ul{p_h})}_{\boldsymbol{\Cv},\Cq}.
	\end{align}
	From Lemma~\ref{lem:infsup}, there exists a $\Vw_h\in\BV_{h,0}(k,\Omega)$ and a
	$\ul{\Bw_h}\in \boldsymbol{\Cv}^n_h$ such that $\Bw_{h,i}=\Vw_h$ in $\Omega^n_{h,i}$, $i=1,2$, and
	\begin{align}\label{discwell0}
		C_0\tN{\ul{p_h}}_{\Cq}^2 \le \Aprod[\Omega^n_\eta]{\Div\ul{\Bw_h},\ul{p_h}}
		+ \mathscr{J}^n_p(\ul{p_h},\ul{p_h}),\quad
		\tN{\ul{\Bw_h}}_{\boldsymbol{\Cv}}
		\le C_1\tN{\ul{p_h}}_{\Cq},
	\end{align}
	where $C_0,C_1$ are positive constants independent of $h,\tau$, and $\nu$.
	Choosing $(\ul{\Bv_h},\ul{q_h}) = (\ul{\Bu_h} + \alpha \ul{\Bw_h},\ul{p_h})$ for some $\alpha\in (0,1)$ to be specified later, we have
	\begin{align}
		\mathscr{K}_1((\ul{\Bu_h},\ul{p_h}),(\ul{\Bv_h},\ul{p_h})) \ge \,&
		\lambda_0^k\tau^{-1}\N{\ul{\Bu_h}}_{0,\Omega^n_\eta}^2
		+\mathscr{A}_h^n(\ul{\Bu_h},\ul{\Bu_h})
		+ C_0\alpha\tN{\ul{p_h}}_{\Cq}^2\notag\\
		&+\lambda_0^k\alpha \tau^{-1}\Aprod[\Omega^n_\eta]{\ul{\Bu_h},\ul{\Bw_h}}
		+\alpha\mathscr{A}_h^n(\ul{\Bu_h},\ul{\Bw_h})\notag \\
		&+\gamma_1\nu_2 h^2 \big\|\nu^{-\frac12}\nabla \ul{p_h}\big\|_{0,\Omega^n_\eta}^2
		+ \gamma_1\nu_2 h^2 \mathscr{B}^n_1(\ul{\Bu_h},\ul{p_h}). \label{est:K}
	\end{align}
	Using Lemma~\ref{lem:ch-up} and the assumption $h=O(\tau)$, we have
	\begin{align}
		\mathscr{B}^n_1(\ul{\Bu_h},\ul{p_h})
		&\le  \big\|\nu^{-\frac12}\big(\tau^{-1} \lambda_0^k\Bu_{h,i}
		- \nu\Delta_h\Bu_h\big)\big\|_{0,\Omega^n_\eta}
		\big\|\nu^{-\frac12}\nabla \ul{p_h}\big\|_{0,\Omega^n_\eta} \notag \\
		&\le C_2 h^{-2}\big(\big\|\nu^{-\frac12}\ul{\Bu_h}\big\|^2_{0,\Omega^n_\eta}
		+\tN{\ul{\Bu_h}}_{\boldsymbol{\Cv}}^2\big)
		+\frac{1}{2}\big\|\nu^{-\frac12}\nabla \ul{p_h}\big\|_{0,\Omega^n_\eta}^2,
		\label{discwell2}
	\end{align}
	where $C_2>0$ is a constant independent of $h$, $\tau$, and $\nu$.
	Moreover, by the relation between $\Vw_h$ and $\ul{\Bw_h}$ and Poincar\'{e}'s inequality, there exists a constant $C>0$ independent of $h,\tau$ such that
	\begin{align*}
		\big|\Aprod[\Omega^n_\eta]{\ul{\Bu_h},\ul{\Bw_h}}\big|
		\le\,& \N{\ul{\Bu_h}}_{0,\Omega^n_\eta} \NLtwo{\Vw_h}
		\le C\N{\ul{\Bu_h}}_{0,\Omega^n_\eta} \SNHonev{\Vw_h}\notag\\
		\le \,&C\N{\ul{\Bu_h}}_{0,\Omega^n_\eta} \tN{\ul{\Bw_h}}_{\boldsymbol{\Cv}}.
	\end{align*}
	Together with \eqref{discwell0} and Lemma~\ref{lem:Ah}, this implies
	\begin{align}
		\big|\Aprod[\Omega^n_\eta]{\ul{\Bu_h},\ul{\Bw_h}}\big|
		\le\,& C\N{\ul{\Bu_h}}_{0,\Omega^n_\eta} \tN{\ul{p_h}}_{\Cq}\notag \\
		\le\,& C_3\lambda_0^k\tau^{-1}\N{\ul{\Bu_h}}_{0,\Omega^n_\eta}^2 
		+\frac{\tau}{4\lambda_0^k}C_0\tN{\ul{p_h}}_{\Cq}^2,   \label{discwell3} \\
		\SN{\mathscr{A}_h^{n}(\ul{\Bu_h},\ul{\Bw_h})}
		\le\,& C\tN{\ul{\Bu_h}}_{\boldsymbol{\Cv}}\tN{\ul{p_h}}_{\Cq}
		\le C_4\tN{\ul{\Bu_h}}_{\boldsymbol{\Cv}}^2
		+\frac{1}{4}C_0\tN{\ul{p_h}}_{\Cq}^2.  \label{discwell4}
	\end{align}
	Since $\mathscr{A}_h^n(\ul{\Bu_h},\ul{\Bu_h})\ge 0.9\tN{\ul{\Bu_h}}_{\boldsymbol{\Cv}}^2$
	by Lemma~\ref{lem:Ah}, inserting \eqref{discwell2}--\eqref{discwell4} into \eqref{est:K} leads to
	\begin{align*}
		\mathscr{K}_1((\ul{\Bu_h},\ul{p_h}),(\ul{\Bv_h},\ul{q_h})) \ge \,&
		\big\{\lambda_0^k\tau^{-1}-C_3\alpha(\lambda_0^k/\tau)^2-C_2\gamma_1\big\}
		\N{\ul{\Bu_h}}_{0,\Omega^n_\eta}^2 \\
		&+(0.9-C_4\alpha-C_2\gamma_1\nu_2)\tN{\ul{\Bu_h}}_{\boldsymbol{\Cv}}^2
		+ \frac{1}{2}C_0\alpha\tN{\ul{p_h}}_{\Cq}^2.
	\end{align*}
	Suppose $\gamma_1$ is small enough such that $4C_2\gamma_1\le \min(\nu_2^{-1},\lambda_0^k\tau^{-1})$. Then \eqref{insup-K1} holds with $\alpha_0 = 0.5\min(1,C_0\alpha)$ if we take $\alpha=0.1 \min\{1, C_4^{-1},(C_3\lambda_0^k)^{-1}\tau\}$.
\end{proof}



\section{Modified Stokes-projection operator}

In this section we define a modified Stokes-projection operator from $\boldsymbol{\Cv}^n\times \Cq^n$
to $\boldsymbol{\Cv}^n_h\times\Cq^n_h$. It is a powerful tool for proving  the stability and error estimates of numerical solutions.
To begin with, we make the assumption on the regularity of the solution to Stokes interface problem: \vspace{1mm}

\begin{assumption}\label{ass-3}
	For any $\xibf\in\Ltwov[\Omega^n_\eta]$, the Stokes interface problem
	\begin{align*}
		\begin{cases}
			-\nu \Delta \Bz + \nabla r = \xibf, \quad
			\Div \Bz =0 \quad \text{in}\;\; \Omega^n_\eta,  \vspace{1mm}\\
			\jump{(\nu \nabla \Bz +r\bbI)\cdot\Bn} = \jump{\Bz}=0
			\quad \text{on}\;\; \Gamma_\eta^n,  \vspace{1mm} \\
			\Bz=0 \quad \text{on} \;\;\Sigma,
		\end{cases}
	\end{align*}
	has a unique solution which satisfies
	\begin{equation*}
		\sum_{i=1,2}\Big(\nu_i^{\frac12}\big\|\Bz\big\|_{\BH^2(\Omega^n_{\eta,i})}
		+\nu_i^{-\frac12}\big\|r\big\|_{H^1(\Omega^n_{\eta,i})}\Big)
		\lesssim \NLtwov[\Omega^n_\eta]{\xibf}.
	\end{equation*}
\end{assumption}

For the case that $\Omega^n_{\eta,1}$ and $\Omega^n_{\eta,2}$ have smooth boundaries, the solution to elliptic interface problem has the $H^2$-regularity in each sub-domain (cf. \cite{bab70}). Here we do not intend to prove the result for Stokes interface problem and just treat it as an assumption. Another issue involves the outer boundary $\Sigma=\partial\Omega$ which is not smooth, but Lipschitz continuous. Our method and theory can be extended straightforwardly to the case that $\Sigma$ is smooth. Since the paper is focused on dealing with the moving interface $\Gamma(t)$, we only consider the case that $\Sigma$ is fitted with $\Ct_h$ for simplicity.

\subsection{The definition}

For any $(\ul{\Bu},\ul{p})\in \boldsymbol{\Cv}^n\times \Cq^n$ and $f\in (\Cq^n_h)'$, the modified Stokes projection $\Cs^n(\ul{\Bu},\ul{p},f) \equiv (\ul{\Bu_h}, \ul{p_h})\in {\boldsymbol{\Cv}}_h^n\times \Cq_h^n$ is defined by the solution to the mixed problem
\begin{subequations}\label{def-proj}
	\begin{align}
		\mathscr{A}_h^n(\ul{\Bu_h}, \ul{\Bv_h}) + \mathscr{B}^n_0(\ul{\Bv_h},\ul{p_h})
		&= a_h^n(\ul{\Bu}, \ul{\Bv_h}) + \mathscr{B}^n_0(\ul{\Bv_h},\ul{p}),
		\quad \forall\,\ul{\Bv_h}\in {\boldsymbol{\Cv}}_h^n, \label{proj-u}\\
		\mathscr{B}^n_0(\ul{\Bu_h}, \ul{q_h}) - \mathscr{J}_p^n(\ul{p_h},\ul{q_h}) \label{proj-p}
		& =f(\ul{q_h}),\quad \forall\,\ul{q_h}\in\Cq^n_h,
	\end{align}
\end{subequations}
where $a_h^n(\ul{\Bu},\ul{\Bv})=
\Aprod[\Omega_{\eta}^n]{\nu\nabla\ul{\Bu}, \nabla\ul{\Bv}}
-\mathscr{F}^n(\ul{\Bu},\ul{\Bv})
+\mathscr{J}_0^n(\ul{\Bu},\ul{\Bv})$. We define two norms on $(\Cq^n_h)'$ as
\begin{equation}\label{eq:normf}
	\N{f}_{(\Cq^n_h)'} = \sup_{\ul{0}\ne\ul{q_h}\in\Cq^n_h}
	\frac{f(\ul{q_h})}{\tN{\ul{q_h}}_\Cq},\qquad
	\N{f}_{1,(\Cq^n_h)'}:=\sup_{\ul{0}\ne\ul{q_h}\in\Cq_h}
	\frac{f(\ul{q_h})}{\tN{\ul{q_h}}_{1,\Cq}}.
\end{equation}

\begin{theorem} \label{thm:infsup}
	Let the assumptions in Lemma~\ref{lem:infsup} be satisfied. Define
	\begin{align}
		\mathscr{K}_0((\ul{\Bu_h},\ul{p_h}),(\ul{\Bv_h},\ul{q_h})) =\,&
		\mathscr{A}_h^n(\ul{\Bu_h},\ul{\Bv_h})
		+\mathscr{B}^n_0(\ul{\Bv_h},\ul{p_h})
		-\mathscr{B}^n_0(\ul{\Bu_h},\ul{q_h})\notag\\
		&+ \mathscr{J}_p^n(\ul{p_h},\ul{q_h}).\label{K0}
	\end{align}
	For any $(\ul{\Bu_h},\ul{p_h})\in \boldsymbol{\Cv}_h^n\times\mathcal{Q}_h^n$, there exists a $\ul{\Bv_h}\in {\boldsymbol{\Cv}}_h^n$ such that
	\begin{align}\label{ieq:K0}
		\tN{\ul{\Bv_h}}_{\boldsymbol{\Cv}}^2\lesssim
		\tN{(\ul{\Bu_h},\ul{p_h})}_{{\boldsymbol{\Cv}},\Cq}^2 \lesssim
		\mathscr{K}_0((\ul{\Bu_h},\ul{p_h}),(\ul{\Bv_h},\ul{p_h})).
	\end{align}
\end{theorem}
\begin{proof}
	The proof is similar to that of Theorem~\ref{thm:alpha}.
	By Lemma~\ref{lem:infsup}, there exist a $\ul{\Bw_h}\in\boldsymbol{\Cv}^n_h$ satisfying $\jump{\ul{\Bw_h}}=0$ on $\Gamma^n_\eta$ and two positive constants $C_0,C_1$ independent of $\tau,h$ such that
	\begin{align*}
		C_0\tN{\ul{\Bw_h}}_{\boldsymbol{\Cv}}^2
		\le C_1\tN{\ul{p_h}}_{\Cq}^2 \le
		\Aprod[\Omega^n_\eta]{\Div\ul{\Bw_h},\ul{p_h}}
		+ \mathscr{J}^n_p(\ul{p_h},\ul{p_h}) .
	\end{align*}
	Set $\ul{\Bv_h} =\ul{\Bu_h}-\alpha\ul{\Bw_h}$ and $\ul{q_h}=\ul{p_h}$ for some $\alpha\in (0,1)$ to be specified. By Lemma~\ref{lem:Ah}, there exists a constant $C_2>0$ such that
	\begin{align*}
		\mathscr{K}_0((\ul{\Bu_h},\ul{p_h}),(\ul{\Bv_h},\ul{q_h}))
		=\,& \mathscr{A}_h^n(\ul{\Bu_h},\ul{\Bu_h}-\alpha\ul{\Bw_h})
		+\alpha\Aprod[\Omega^n_\eta]{\Div\ul{\Bw_h},\ul{p_h}}
		+ \mathscr{J}_p^n(\ul{p_h},\ul{p_h}) \\
		\ge\,& 0.9\tN{\ul{\Bu_h}}_{\boldsymbol{\Cv}}^2
		-C_2\alpha\tN{\ul{\Bu_h}}_{\boldsymbol{\Cv}}\tN{\ul{\Bw_h}}_{\boldsymbol{\Cv}}
		+C_1\alpha \tN{\ul{p_h}}_{\Cq}^2\\
		\ge\,& (0.9-0.5C_2^{2}C_1C_0^{-2}\alpha)\tN{\ul{\Bu_h}}_{\boldsymbol{\Cv}}^2
		+0.5C_1\alpha \tN{\ul{p_h}}_{\Cq}^2.
	\end{align*}
	Taking $\alpha = C_2^{-2}C_1^{-1}C_0^{2}$ leads to \eqref{ieq:K0}.
\end{proof}

\begin{theorem}\label{thm:stokes-proj}
	Problem \eqref{def-proj} has a unique solution which satisfies
	\ben
	\tN{(\ul{\Bu_h},\ul{p_h})}_{{\boldsymbol{\Cv}},\Cq}
	\lesssim \tN{(\ul{\Bu},\ul{p})}_{{\boldsymbol{\Cv}},\Cq} + \N{f}_{(\Cq^n_h)'}.
	\een
\end{theorem}
\begin{proof}
	Note that \eqref{def-proj} is equivalent to the variational problem
	\begin{align}\label{def-proj-K0}
		\mathscr{K}_0((\ul{\Bu_h},\ul{p_h}),(\ul{\Bv_h},\ul{q_h}))
		=a_h^n(\ul{\Bu}, \ul{\Bv_h}) + \mathscr{B}^n_0(\ul{\Bv_h},\ul{p})
		-f(\ul{q_h}),\;
		\forall\,(\ul{\Bv_h},\ul{q_h})\in {\boldsymbol{\Cv}}_h^n\times \Cq_h^n.
	\end{align}
	By Theorem~\ref{thm:infsup}, we have the inf-sup condition
	\begin{align}\label{ieq:infsup}
		\tN{(\ul{\Bu_h},\ul{p_h})}_{{\boldsymbol{\Cv}},\Cq}\lesssim
		\sup_{(\ul{\Bv_h},\ul{q_h})\in\boldsymbol{\Cv}_h^n\times \Cq_h^n}
		\frac{\mathscr{K}_0((\ul{\Bu_h},\ul{p_h}),(\ul{\Bv_h},\ul{q_h}))}
		{\tN{(\ul{\Bv_h},\ul{q_h})}_{{\boldsymbol{\Cv}},\Cq}}.
	\end{align}
	This implies that \eqref{def-proj-K0} has a unique solution. Moreover, combining \eqref{def-proj-K0} and \eqref{ieq:infsup} yields
	\begin{align*}
		\tN{(\ul{\Bu_h},\ul{p_h})}_{{\boldsymbol{\Cv}},\Cq}\lesssim\,&
		\sup_{(\ul{\Bv_h},\ul{q_h})\in\boldsymbol{\Cv}_h^n\times \Cq_h^n}
		\frac{a_h^n(\ul{\Bu}, \ul{\Bv_h}) + \mathscr{B}^n_0(\ul{\Bv_h},\ul{p})
			-f(\ul{q_h})}{\tN{(\ul{\Bv_h},\ul{q_h})}_{{\boldsymbol{\Cv}},\Cq}}\\
		\lesssim\;& \tN{(\ul{\Bu},\ul{p})}_{{\boldsymbol{\Cv}},\Cq} + \N{f}_{(\Cq^n_h)'}.
	\end{align*}
	The proof is finished.
\end{proof}
\vspace{1mm}

\subsection{Quasi-interpolation operators}

For $1\le m\le k$ and $D\subseteq\Omega$, let $\pi_{m,D}$: $\Hone[D]\to V_h(m,D)$ be the quasi-interpolation (or Scott-Zhang) operator which respects homogeneous Dirichlet boundary conditions \cite{sco90}. We have the well-known results: for any $K\in \Ct_h$ and $E\in \Ce_h$,
\begin{align}
	\N{\pi_{m,\Omega} v - v}_{H^l(K)}&\lesssim h^{\min \{m+1-l,s-l\}}\SN{v}_{H^{s}(D_K)}, \label{int-errK}\\
	\N{\pi_{m,\Omega} v - v}_{H^l(E)}&\lesssim h^{\min \{m+1/2-l,s-1/2-l\}}\SN{v}_{H^{s}(D_E)},\quad 0\leq l\leq s,
	\label{int-errE}
\end{align}
where $D_A$ is the union of all elements having non-empty intersection with $A$, for $A=K$ or $E$. Similarly, let $\pibf_{m,D}$ denote the interpolation operator on vector-valued functions. For any $\ul{\Bv}\in\Honev[\Omega^n_{h,1}]\times \Honev[\Omega^n_{h,2}]$ and any
$\ul{q}\in\Hone[\Omega^n_{h,1}]\times \Hone[\Omega^n_{h,2}]$, we define
\begin{align*}
	\ul{\pibf_m}(\ul{\Bv}) &= \big(\pibf_{m,\Omega^n_{h,1}}\Bv_{1},
	\pibf_{m,\Omega^n_{h,2}}\Bv_{2}\big)\in \boldsymbol{\Cv}^n_h,\\
	\ul{\pi_m}(\ul{q}) &= \big(\pi_{m,\Omega^n_{h,1}}q_{1},
	\pi_{m,\Omega^n_{h,2}}q_{2}\big)\in \Cq^n_h.
\end{align*}

\begin{lemma}\label{lem:scott-zhang}
	Suppose that the components of $\ul\Bv$ and $\ul{q}$ satisfy $\Bv_i\in\BH^{k+1}(\Omega^n_{h,i})$, $q_i\in H^{k}(\Omega^n_{h,i})$ for $i=1,2$. Let $\ul{\Bv_h}=\ul{\pibf_k}(\ul{\Bv})$ and $\ul{q_h}=\ul{\pi_{k-1}}(\ul{q})$. Then
	\begin{align}
		&\big\|\ul{\Bv}-\ul{\Bv_h}\big\|_{\boldsymbol{\Cv}}^2
		+\tN{\ul{\Bv}-\ul{\Bv_h}}_{\boldsymbol{\Cv}}^2
		+h^2\N{\nu^{\frac12}\Delta_h(\ul{\Bv} -\ul{\Bv_h})}_{0,\Omega^n_\eta}^2\notag \\
		\lesssim\; &h^{2k}\sum_{i=1,2}\nu_i\SN{\Bv_i}_{\BH^{k+1}(\Omega_{h,i}^n)}^2,
		\label{ieq:v-vh}\\
		&\big\|\ul{q}-\ul{q_h}\big\|_{\Cq}^2 +\tN{\ul{q}-\ul{q_h}}_{\Cq}^2 +
		h^2\sum_{i=1,2}\N{\nu^{-\frac12}\nabla(q_i-q_{h,i})}_{0,\Omega^n_{h,i}}^2 \notag \\
		\lesssim \;&h^{2k}\sum_{i=1,2}\nu_i^{-1}\SN{q_i}_{H^{k}(\Omega_{h,i}^n)}^2.
		\label{ieq:q-qh}
	\end{align}
\end{lemma}
\begin{proof}
	Since $q_i\in H^k(\Omega_{h,i}^n)$ and $\Bv_i \in \BH^{k+1}(\Omega_{h,i}^n)$, from \eqref{int-errE} we have
	\begin{align*}
		&\mathscr{J}_p^n(\ul{q}-\ul{q_h},\ul{q}-\ul{q_h}) \lesssim
		\sum_{\mbox{\tiny$\begin{array}{c}
					K\in\Ce_{i,B}^n,\\
					i=1,2
				\end{array}$}}h^{2k}
		\nu_i^{-1}\SN{q_i}_{H^{k}(D_E)}^2
		\lesssim h^{2k}\sum_{i=1}^2\nu_i^{-1}\SN{q_i}_{H^{k}(\Omega_{h,i}^n)}^2,\\
		&\mathscr{J}_u^n(\ul\Bv-\ul{\Bv_h},\ul\Bv-\ul{\Bv_h})
		\lesssim 
		\sum_{\mbox{\tiny$\begin{array}{c}
					K\in\Ce_{i,B}^n,\\
					i=1,2
				\end{array}$}}h^{2k}
		\nu_i\SN{\Bv_i}_{\BH^{k+1}(D_E)}^2
		\lesssim h^{2k}\sum_{i=1}^2\nu_i\SN{\Bv_i}_{\BH^{k+1}(\Omega_{h,i}^n)}^2.
	\end{align*}
	For any $K\in \Ct_{h,B}^n$ and $\Gamma_K =\Gamma_{\eta}^n \cap K$,
	from \eqref{ieq:trace0} we have
	\begin{align}
		&\avg{\nu}\N{\jump{\ul\Bv-\ul{\Bv_h}}}_{\BL^2(\Gamma_K)}^2
		\lesssim
		\sum_{i=1}^2
		\sum_{j=0}^1 h^{2j-1}
		\kappa_i\nu_i \SN{\Bv_i-\Bv_{h,i}}_{\BH^{j}(K)}^2 \notag \\
		&\avg{\nu}^{-1}\N{\Avg{\nu\partial_\Bn(\Bv-\ul{\Bv_h})}}_{\BL^2(\Gamma_K)}^2
		\lesssim\sum_{i=1}^2 
		\sum_{j=1}^2h^{2j-3}\kappa_i\nu_i\SN{\Bv_i-\Bv_{h,i}}_{\BH^{j}(K)}^2,\label{eq:v-vh} \\
		&\avg{\nu}^{-1}\N{\Avg{\ul{q}-\ul{q_h}}}_{L^2(\Gamma_K)}^2
		\lesssim\sum_{i=1}^2
		\sum_{j=0}^1 h^{2j-1}\kappa_i\nu_i^{-1}\SN{q_i-q_{h,i}}_{H^j(K)}^2,\label{eq:q-qh}
	\end{align}
	where $\SN{\cdot}_{\BH^0(K)}= \N{\cdot}_{\BL^2(K)}$.
	We conclude that
	\begin{align*}
		&h^{-1}\avg{\nu}\NLtwov[\Gamma^n_\eta]{\jump{\ul\Bv-\ul{\Bv_h}}}^2
		+h\avg{\nu}^{-1}\N{\Avg{\nu\partial_{\Bn}(\ul\Bv-\ul{\Bv_h})}}_{\BL^2(\Gamma_\eta^n)}^2\\
		&\hspace{41mm}\lesssim  h^{2k}\sum_{i=1,2}\nu_i \SN{\Bv_i}^2_{H^{k+1}(\Omega^n_{h,i})},\\
		&h\avg{\nu}^{-1}\N{ \Avg{\ul{q}-\ul{q_h}}}^2_{L^2(\Gamma_\eta^n)}
		\lesssim h^{2k} \sum_{i=1,2}\nu_i^{-1}\SN{q_i}^2_{H^k(\Omega^n_{h,i})}.
	\end{align*}
	The other terms in \eqref{ieq:v-vh}--\eqref{ieq:q-qh} can be estimated similarly by using
	\eqref{int-errK}.
\end{proof}

\subsection{Error estimates for modified Stokes projections}

Now we define the operator $\Cb^n_0$: $\boldsymbol{\Cv}^n\to (\Cq^n_h)'$ as follows:
for any $\ul\Bu\in\boldsymbol{\Cv}^n$, $\Cb^n_0\ul\Bu\in(\Cq^n_h)'$ satisfies
\ben
(\Cb^n_0\ul\Bu)(\ul{q_h})= \mathscr{B}^n_0(\ul{\Bu},\ul{q_h}),
\qquad \forall\,\ul{q_h}\in\Cq^n_h.
\een
For convenience in notation, we define
\ben
M_k(\ul{\Bu},\ul{p}) =\sum_{i=1}^2\big(\nu_i\SN{\Bu_i}_{\BH^{k+1}(\Omega_{h,i}^n)}^2
+\nu_i^{-1}\SN{p_i}_{H^k(\Omega_{h,i}^n)}^2\big).
\een

\begin{center}
	\fbox{\parbox{0.975\textwidth}
		{\begin{theorem}\label{thm:stokes-H1}
				Let Assumption~\ref{ass-2} be satisfied and let $(\ul{\Bu_h}, \ul{p_h})\equiv \Cs^n (\ul{\Bu},\ul{p},f)$ be the solution to problem~\eqref{def-proj}. Suppose
				$M_k(\ul{\Bu},\ul{p})< \infty$. Then
				\begin{align*}
					&\tN{(\ul{\Bu}-\ul{\Bu_h},\ul{p} -\ul{p_h})}_{{\boldsymbol{\Cv}},\Cq}
					+h\dN{\nu^{\frac12}\Delta_h(\ul\Bu -\ul{\Bu_h})}_{0,\Omega^n_\eta}
					+h\big\|\nu^{-\frac12}\nabla(\ul{p}-\ul{p_h})\big\|_{0,\Omega^n_\eta} \\
					\lesssim\,& h^k M_k^{\frac12}(\ul{\Bu},\ul{p}) + \N{f-\Cb^n_0\Bu}_{(\Cq^n_h)'}.
				\end{align*}
	\end{theorem}}}
\end{center}
\vspace{1mm}

\begin{proof}
	For convenience, we write
	\ben
	\begin{array}{lll}
		\ul{\Be_u}=\ul{\Bu}-\ul{\Bu_h},\quad
		&\ul{\xibf_h}=\ul{\Bu} -\ul{\pibf_k}(\ul{\Bu}),\quad
		&\ul{\etabf_h}=\ul{\pibf_k}(\ul{\Bu})-\ul{\Bu_h},\vspace{1mm}\\
		\ul{e_p}=\ul{p}-\ul{p_h},\quad
		&\ul{r_h}=\ul{p} -\ul{\pi_{k-1}}(\ul{p}),\quad
		&\ul{s_h}=\ul{\pi_{k-1}}(\ul{p})-\ul{p_h}.
	\end{array}
	\een
	Since $\Bu_i\in\BH^{k+1}(\Omega^n_{h,i})$ and $p_i\in H^{k}(\Omega^n_{h,i})$, we have
	$\mathscr{J}_\Bu^n(\ul{\Bu},\ul{\Bv_h}) =\mathscr{J}_p^n(\ul{p},\ul{q_h})=0$
	for all $\ul{\Bv_h}\in\boldsymbol{\Cv}^n_h$ and $\ul{q_h}\in\Cq^n_h$. From \eqref{def-proj-K0}, we have
	\begin{align*}
		\mathscr{K}_0((\ul{\Be_u},\ul{e_p}),(\ul{\Bv_h},\ul{q_h}))
		=\mathscr{B}^n_0(\ul\Bu,\ul{q_h})-f(q_h),\quad
		\forall\,(\ul{\Bv_h},\ul{q_h})\in {\boldsymbol{\Cv}}_h^n\times \Cq_h^n.
	\end{align*}
	By Theorem~\ref{thm:infsup}, there is a $\ul{\Bw_h}\in \boldsymbol{\Cv}^n_h$ which satisfies $\tN{\ul{\Bw_h}}_{\boldsymbol{\Cv}}\lesssim \tN{(\ul{\etabf_h},\ul{s_h})}_{{\boldsymbol{\Cv}},\Cq}$ and
	\begin{align*}
		\tN{(\ul{\etabf_h},\ul{s_h})}_{{\boldsymbol{\Cv}},\Cq}^2
		\lesssim\,&
		\mathscr{K}_0((\ul{\etabf_h},\ul{s_h}),(\ul{\Bw_h},\ul{s_h})) \\
		=\,& \mathscr{K}_0((\ul{\Be_u},\ul{e_p}),(\ul{\Bw_h},\ul{s_h}))
		-\mathscr{K}_0((\ul{\xibf_h},\ul{r_h}),(\ul{\Bw_h},\ul{s_h})) \\
		=\,&\mathscr{B}^n_0(\ul\Bu,\ul{s_h})-f(s_h)
		-\mathscr{K}_0((\ul{\xibf_h},\ul{r_h}),(\ul{\Bw_h},\ul{s_h}))\\
		\lesssim\,& \big(\N{f-\Cb^n_0\Bu}_{(\Cq^n_h)'}
		+\tN{(\ul{\xibf_h},\ul{r_h})}_{{\boldsymbol{\Cv}},\Cq}\big)
		\tN{(\ul{\Bw_h},\ul{s_h})}_{{\boldsymbol{\Cv}},\Cq} \\
		\lesssim\,&  \big(\N{f-\Cb^n_0\Bu}_{(\Cq^n_h)'}
		+\tN{(\ul{\xibf_h},\ul{r_h})}_{{\boldsymbol{\Cv}},\Cq}\big)
		\tN{(\ul{\etabf_h},\ul{s_h})}_{{\boldsymbol{\Cv}},\Cq},
	\end{align*}
	where we have used Lemma~\ref{lem:Ah} in the second inequality. It follows that
	\ben
	\tN{(\ul{\etabf_h},\ul{s_h})}_{{\boldsymbol{\Cv}},\Cq}\lesssim
	\N{f-\Cb^n_0\Bu}_{(\Cq^n_h)'}
	+\tN{(\ul{\xibf_h},\ul{r_h})}_{{\boldsymbol{\Cv}},\Cq}.
	\een
	Together with Lemma~\ref{lem:scott-zhang}, this yields the estimate for
	$\tN{(\ul{\Bu}-\ul{\Bu_h},\ul{p} -\ul{p_h})}_{{\boldsymbol{\Cv}},\Cq}$.

	Finally, from Lemma~\ref{lem:ch-up} and Lemma~\ref{lem:scott-zhang}, we have
	\begin{align*}
		\dN{\nu^{\frac 12}\Delta_h\ul{\Be_u}}_{0,\Omega^n_\eta}
		+\big\|\nu^{-\frac 12}\nabla\ul{e_p}\big\|_{0,\Omega^n_\eta}
		\lesssim\,& \big\|\nu^{\frac 12}\Delta_h\ul{\xibf_h}\big\|_{0,\Omega^n_\eta}
		+\big\|\nu^{-\frac 12}\nabla \ul{r_h} \big\|_{0,\Omega^n_\eta}\\
		&+h^{-1}\tN{(\ul{\etabf_h},\ul{s_h})}_{{\boldsymbol{\Cv}},\Cq}\\
		\lesssim \,&  h^{-1} \Big(h^kM_k^{\frac12}(\ul{\Bu},\ul{p})
		+\N{f-\Cb^n_0\Bu}_{(\Cq^n_h)'}\Big).
	\end{align*}
	The proof is finished.
\end{proof}
\vspace{1mm}

\begin{center}
	\fbox{\parbox{0.975\textwidth}
		{\begin{theorem}\label{thm:stokes-L2}
				Let Assumption~\ref{ass-2} and Assumption~\ref{ass-3} be satisfied.
				Then upon a hidden constant depending only $\Omega^n_{\eta,1}$ and $\Omega^n_{\eta,2}$, $(\ul{\Bu_h}, \ul{p_h})\equiv \Cs^n (\ul{\Bu},\ul{p},f)$ satisfies
				\begin{align*}
					\N{\ul{\Bu}-\ul{\Bu_h}}_{0,\Omega^n_\eta} \lesssim
					h \tN{(\ul{\Bu}-\ul{\Bu_h},\ul{p}-\ul{p_h})}_{{\boldsymbol{\Cv}},\Cq}
					+\N{f-\Cb^n_0\Bu}_{1,(\Cq^n_h)'}.
				\end{align*}
	\end{theorem}}}
\end{center}
\vspace{1mm}

\begin{proof}
	We shall use the duality technique to prove the lemma. Let $\Be_u\in\Ltwov[\Omega^n_\eta]$ be defined as $\Be_u :=\Bu_i-\Bu_{h,i}$ in $\Omega^n_{\eta,i}$ for $i=1,2$. Consider the auxiliary problem
	\begin{align}\label{stokes-eu}
		\begin{cases}
			-\nu \Delta \Bz + \nabla r = \Be_u, \quad
			\Div \Bz =0 \quad \text{in}\;\; \Omega^n_\eta,  \vspace{1mm}\\
			\jump{(\nu \nabla \Bz +r\bbI)\cdot\Bn} = \jump{\Bz}=0
			\quad \text{on}\;\; \Gamma_\eta^n,  \vspace{1mm} \\
			\Bz=0 \quad \text{on} \;\;\Sigma.
		\end{cases}
	\end{align}
	By Assumption~\ref{ass-3}, we have
	$\big\|\nu^{\frac12}\Bz\big\|_{\BH^2(\Omega^n_\eta)}
	+\big\|\nu^{-\frac12}r\big\|_{H^1(\Omega^n_\eta)}
	\lesssim \|\Be_u\|_{L^2(\Omega^n_\eta)}$.
	By Stein's extension theory (see \cite[Chapter~6]{ste70}), there exist $\Bz_i\in\BH^2(\Omega^n_{h,i})$ and $r_i\in H^1(\Omega^n_{h,i})$, $i=1,2$, which satisfy $\Bz_i=\Bz$ and $r_i=r$ in $\Omega^n_{\eta,i}$, and
	\begin{align}
		&\nu_i^{\frac12}\|\Bz_i\|_{\BH^2(\Omega^n_{h,i})}
		+\nu_i^{-\frac12}\|r_i\|_{H^1(\Omega^n_{h,i})} \notag \\
		\lesssim\;&
		\|\nu_i^{\frac12}\Bz\|_{H^2(\Omega^n_{\eta,i})} +\|\nu_i^{-\frac12}r\|_{H^1(\Omega^n_{\eta,i})}
		\lesssim \|\Be_u\|_{L^2(\Omega_{\eta}^n)}.\label{z-ext}
	\end{align}
	Define $\ul{\Bz}=(\Bz_1,\Bz_2)$, $\ul{r}=(r_1,r_2)$, $\ul{\Bz_h}=\ul{\pibf_1}(\ul{\Bz})$, $\ul{r_h}=\ul{\pi_1}(\ul{r})$, $\ul{\Be_z} =\ul{\Bz}-\ul{\Bz_h}$, and $\ul{e_r} =\ul{r}-\ul{r_h}$. From
	Lemma~\ref{lem:scott-zhang} and \eqref{z-ext}, we obtain
	\begin{align}\label{ieq:ez}
		\tN{(\ul{\Be_z},\ul{e_r})}_{\boldsymbol{\Cv},\Cq} +
		h\big\|\nu^{-\frac12}\nabla\ul{e_r}\big\|_{0,\Omega^n_\eta}
		\lesssim h\|\Be_u\|_{L^2(\Omega_\eta^n)}.
	\end{align}
	
	Write $\ul{\Be_u} =\ul{\Bu}-\ul{\Bu_h}$ and $\ul{e_p} =\ul{p}-\ul{p_h}$.
	The weak form of \eqref{stokes-eu} yields
	\begin{align*}
		a_h^n(\ul{\Bz},\ul{\Be_u}) +\mathscr{B}^n_0(\ul{\Be_u},\ul{r})
		= \Aprod[\Omega^n_\eta]{\ul{\Be_u},\ul{\Be_u}},\qquad
		\mathscr{B}^n_0(\ul{\Bz},\ul{e_p}) =0.
	\end{align*}
	Write $f_{\Bu}:=f-\Cb^n_0\Bu$ for convenience.
	From \eqref{def-proj}, we also have
	\begin{align*}
		&a_h^n(\ul{\Be_u},\ul{\Bz_h}) +\mathscr{B}^n_0(\ul{\Bz_h},\ul{e_p})
		=\mathscr{J}^n_u(\ul{\Bu_h},\ul{\Bz_h}),\\
		&\mathscr{B}^n_0(\ul{\Be_u}, \ul{r_h})
		+\mathscr{J}_p^n(\ul{p_h},\ul{r_h}) +f_\Bu(\ul{r_h}) =0.
	\end{align*}
	Combining the four equalities above and noting that $a^n_h$ is symmetric, we have
	\begin{align}
		\|\Be_u\|_{\BL^2(\Omega^n_\eta)}^2
		=\,& a_h^n(\ul{\Be_z},\ul{\Be_u})+ \mathscr{B}^n_0(\ul{\Be_z},\ul{e_p})
		+\mathscr{B}^n_0(\ul{\Be_u},\ul{e_r}) +\mathscr{J}^n_u(\ul{\Bu_h},\ul{\Bz_h})\notag \\
		& -\mathscr{J}_p^n(\ul{p_h},\ul{r_h}) -f_\Bu(\ul{r_h})\notag \\
		=\,& a_h^n(\ul{\Be_z},\ul{\Be_u})+ \mathscr{B}^n_0(\ul{\Be_z},\ul{e_p})
		+\mathscr{B}^n_0(\ul{\Be_u},\ul{e_r})-f_\Bu(\ul{r_h})\notag\\
		&+h \sum_{i=1,2}\sum_{E\in \Ce_{i,B}^n}\int_{E} \nu_i\jump{\partial_{\Bn}(\Bu_i-\Bu_{h,i})}\jump{\partial_{\Bn}(\Bz_i-\Bz_{h,i})} \notag \\
		&-h^3 \sum_{i=1,2}\sum_{E\in \Ce_{i,B}^n}\int_{E} \nu^{-1}_i
		\jump{\partial_{\Bn} (p_i - p_{h,i})} \jump{\partial_{\Bn} r_{h,i}},
		\label{ieq:eu}
	\end{align}
	where in the second equality, we have used the fact that $\Bz_{h,i}$ and $r_{h,i}$ are piecewise linear.

	From \eqref{int-errK} and \eqref{z-ext}, we know that $\nu_i^{-\frac12}\NHone[\Omega^n_{h,i}]{r_{h,i}}\lesssim \nu_i^{-\frac12}\NHone[\Omega^n_{h,i}]{r}
	\lesssim \|\Be_u\|_{\BL^2(\Omega^n_\eta)}$ for $h$ small enough.
	By arguments similar to the proof of Lemma~\ref{lem:Ah} and using Theorem~\ref{thm:stokes-H1} and inequality \eqref{ieq:ez}, we have
	\begin{align}
		&\SN{a_h^n(\ul{\Be_z},\ul{\Be_u}) +\mathscr{B}^n_0(\ul{\Be_z},\ul{e_p})
			+\mathscr{B}^n_0(\ul{\Be_u},\ul{e_r})}
		\lesssim h \tN{(\ul{\Be_u},\ul{e_p})}_{\boldsymbol{\Cv},\Cq}\|\Be_u\|_{\BL^2(\Omega^n_\eta)},
		\label{ieq:aBB} \\
		&\SN{f_\Bu(\ul{r_h})} \lesssim \N{f_\Bu}_{1,(\Cq^n_h)'} \tN{\ul{r_h}}_{1,\Cq}
		\lesssim \N{f_\Bu}_{1,(\Cq^n_h)'}\|\Be_u\|_{\BL^2(\Omega^n_\eta)}.
		\label{ieq:fu}
	\end{align}
	By the trace inequality \eqref{ieq:trace0} and the Cauchy-Schwarz inequality, we also have
	\begin{align}
		&h\sum_{i=1,2}\sum_{E\in \Ce_{i,B}^n}\int_{E}\nu_i\jump{\partial_{\Bn}(\Bu_i-\Bu_{h,i})}
		\jump{\partial_{\Bn}(\Bz_i-\Bz_{h,i})} \notag\\
		\lesssim\,&\mathscr{J}^n_\Bu(\ul{\Be_u},\ul{\Be_u})^{\frac12}
		\sum_{i=1,2}\nu^{\frac12}_i\big(\SNHonev[\Omega^n_{h,i}]{\Bz_i-\Bz_{h,i}}
		+h\SN{\Bz_i}_{\BH^2(\Omega^n_{h,i})}\big) \notag\\
		\lesssim\,& h \mathscr{J}^n_\Bu(\ul{\Be_u},\ul{\Be_u})^{\frac12}
		\|\Be_u\|_{\BL^2(\Omega^n_\eta)}.
		\label{ieq:jump-u}
	\end{align}
	Similarly, by norm equivalence and the estimates in \eqref{z-ext}--\eqref{ieq:ez}, we get
	\begin{align}
		&h^3\sum_{i=1,2}\sum_{E\in \Ce_{i,B}^n}\int_{E} \nu^{-1}_i
		\jump{\partial_{\Bn} (p_i - p_{h,i})} \jump{\partial_{\Bn} r_{h,i}} \notag \\
		\lesssim\,& h\mathscr{J}^n_p(\ul{e_p},\ul{e_p})^{\frac12}
		\sum_{i=1,2}\nu^{-\frac12}_i \SNHone[\Omega^n_{h,i}]{r_{h,i}}  \notag \\
		\lesssim\,& h \mathscr{J}^n_p(\ul{e_p},\ul{e_p})^{\frac12}\|\Be_u\|_{\BL^2(\Omega^n_\eta)}. \label{ieq:jump-p}
	\end{align}
	The proof is finished by inserting \eqref{ieq:aBB}--\eqref{ieq:jump-p} into \eqref{ieq:eu}.
\end{proof}
\vspace{1mm}

\begin{remark}
	If $f=\Cb^n_0\ul\Bu$ and $M_k(\ul\Bu,\ul p)<\infty$, Theorems~\ref{thm:stokes-H1} and \ref{thm:stokes-L2} indicate that
	\begin{align}\label{ieq:stokes-L2}
		\N{\ul{\Bu}-\ul{\Bu_h}}_{0,\Omega^n_\eta} \lesssim h
		\tN{(\ul{\Bu}-\ul{\Bu_h},\ul{p}-\ul{p_h})}_{{\boldsymbol{\Cv}},\Cq}
		\lesssim h^{k+1}M_k^{\frac12}(\ul\Bu,\ul p).
	\end{align}
\end{remark}

\subsection{Modified Stokes projections of discrete solutions}

Note from \eqref{eq:disc} that the discrete solutions are coupled with flow maps.
Numerical solutions in previous time steps are not piecewise polynomials any more in the current time step. In the subsequent analysis, the modified Stokes projections will be applied to them.

Following same lines in \cite[Appendix~A]{ma21}, we have the preliminary but useful estimates for the pull-back map $v_h\to v_h\circ\BX^{n,n-l}$ for a finite element function $v_h\in V_h(k,\Omega)$.  \vspace{1mm}

\begin{lemma}\label{lem:uhX}
	Suppose $\eta = O(\tau^{\max(k/3,1)})$. Write $v_h^{n-l,n}=v_{h}\circ\BX_\tau^{n,n-l}$ for $v_h\in V_{h}(k,\Omega)$, $0\le l\le k\le n$.
	There is a constant $C>0$ independent of $\tau$, $h$, and $n$ such that, for $\mu=0,1$, $i=1,2$, and $l\leq j\leq k$,
	\begin{align}
		&\big|v_{h}^{n-l,n}\big|_{\Hone[\Gamma^n_\eta]}^2\le C h^{-1}\SNHone{v_{h}}^2,\label{eq:v-X3}\\
		&\dN{v_{h}^{n-l,n}}_{\Ltwo[\Gamma^n_\eta]}^2
		\le (1+C\tau)\NLtwo[\Gamma^{n-l}_\eta]{v_h}^2 +C\tau^{k+2}h^{-2}\NHone{v_h}^2, \label{eq:v-X2}\\
		&\dN{\nabla^\mu v_{h}^{n-l,n}}_{\Ltwov[\Omega^n_{\eta,i}]}^2
		\le (1+C\tau)\dN{\nabla^\mu v_{h}^{n-j,n-l}}_{\Ltwov[\Omega^{n-l}_{\eta,i}]}^2\notag \\
		&\hspace{35mm}+C\tau^{k+2}h^{-1} \NLtwov{\nabla^\mu v_h}^2.	\label{ieq:vhXn}
	\end{align}  	
\end{lemma}
\vspace{1mm}

Note that $\Omega^n_{\eta,i}\ne \Omega_i(t_n)$ for $i=1,2$. We also need to estimate functions on $\Omega^n_{\eta,i}\oplus \Omega_i(t_n)$, where $A\oplus B = (A\backslash B)\cup (B\backslash A)$ for two sets $A$ and $B$.
\vspace{1mm}

\begin{lemma}\label{lem:uX}
	Suppose $\eta = O(\tau^{\max(k/3,1)})$. Then for $i=1,2$ and $|m-n|\le k$,
	\begin{align}
		&\NLtwo[\Omega_{\eta,i}^n\oplus \Omega_i(t_n)]{v}^2 \lesssim \tau^{k+1}\NHone[\Omega]{v}^2,
		\quad\;\;\, \forall\,v\in \Hone,\label{eq:v-X0 H1} \\
		&\NLtwo[\Omega^n_{\eta,i}\oplus \BX_\tau^{m,n}(\Omega^m_{\eta,i})]{v}^2
		\lesssim \tau^{k+2}\|v\|_{H^1(\Omega)}^2,
		\quad \forall\,v\in H^1(\Omega), \label{eq:v-X4} \\
		&\NLtwov[\Omega^n_{\eta,i}\backslash \Omega_i(t_n)]{\nabla v_{h}}^2
		\lesssim \tau^{k+1}h^{-1}\NLtwov[\Omega^n_{h,i}]{\nabla v_h}^2, \quad
		\forall\,v_h\in V_{h}(k,\Omega), \label{eq:v-X5}\\
		&\NLtwov[\Omega^n_{\eta,i}\oplus \BX^{m,n}_\tau(\Omega^m_{\eta,i})]{\nabla v_{h}}^2
		\lesssim \tau^{k+2}h^{-1}\NLtwov[\Omega^n_{h,i}]{\nabla v_h}^2, \;
		\forall\,v_h\in V_{h}(k,\Omega). \label{eq:v-X0}
	\end{align}
\end{lemma}
\begin{proof}
	The proof of \eqref{eq:v-X0} can be found in \cite[Appendix~A]{ma21}.
	It is left to prove \eqref{eq:v-X0 H1}--\eqref{eq:v-X5}.
	First we cite an important result from \cite[ Lemma~10 and (17)]{nic06}: for any Lipschitz domain $D$ and
	$w\in\Hone[D]$, there holds
	\begin{align}\label{Depsilon}
		\NLtwo[D_\varepsilon]{w} \le C\big(\sqrt{\varepsilon} \N{w}_{H^\kappa(D)}
		+ \varepsilon^{s}\N{w}_{H^s(D)}\big), \qquad 1/2 < \kappa\le s\le 1,
	\end{align}
	where $D_\varepsilon:=\{\Bx\in D: \mathrm{dist}(\Bx,\partial D)< \varepsilon\}$,
	$\varepsilon>0$ is a constant, and $C$ is a constant independent of $\varepsilon$. Using $\eta=O(\tau^{\max\{1,k/3\}})$ and Theorem~\ref{thm:chi},
	we infer that
	\begin{align}\label{dist-n}
		\text{dist}(\Gamma_\eta^n,\Gamma(t_n))=O(\tau^{k+1}).
	\end{align}
	Then using \eqref{Depsilon}, we immediately get \eqref{eq:v-X0 H1}, that is,
	\begin{align}
		\NLtwo[\Omega_{\eta,i}^n\oplus \Omega_i(t_n)]{v}^2 \lesssim
		\text{dist}(\Gamma_\eta^n,\Gamma(t_n))\NHone[\Omega]{v}^2 \lesssim
		\tau^{k+1}\NHone[\Omega]{v}^2.
	\end{align}
	Moreover, inequality \eqref{eq:v-X4} is a consequence of \eqref{Depsilon} and Theorem~\ref{thm:bdr}.

	Finally, from \eqref{dist-n} we find that
	\begin{align*}
		\NLtwov[\Omega^n_{\eta,i}\backslash \Omega_i(t_n)]{\nabla v_{h}}^2
		\lesssim \tau^{k+1}\sum_{K\in\Ct^n_{h,i}} h \NLinfv[K]{\nabla v_{h}}^2
		\lesssim \tau^{k+1} h^{-1} \NLtwov[\Omega^n_{h,i}]{\nabla v_{h}}^2.
	\end{align*}
	This is inequality \eqref{eq:v-X5}. The proof is finished.
\end{proof}
\vspace{1mm}

\begin{remark}\label{rem:Un}
	Using \eqref{ieq:vh-ext}, \eqref{avg-nu}, and Lemma~\ref{lem:uhX}, we easily get
	\begin{align}
		&\dN{\nu^{\frac12}\nabla^{\mu} \ul{\BU_h^{n-l,n}}}_{0,\Omega^n_\eta}^2
		\le  (1+C\tau)\dN{\nu^{\frac12}\nabla^{\mu}\ul{\BU_h^{n-l,n-1}}}_{0,\Omega_\eta^{n-1}}^2\notag \\
		&\hspace{35mm}+C\tau^{k+2}h^{1-2\mu} \mathscr{J}^n_\Bu(\ul{\Bu_h^{n-l}},\ul{\Bu_h^{n-l}}),
		\label{eq:temp3}\\
		&\avg{\nu}\NLtwov[\Gamma_\eta^n]{\jump{\ul{\BU_h^{n-l,n}}}}^2
		\le (1+C\tau)\avg{\nu}\NLtwov[\Gamma_\eta^{n-l}]{\jump{\ul{\Bu_h^{n-l}}}}^2 \notag \\
		&\hspace{40mm}+ C \tau^{k+2}h^{-2} \tN{\ul{\Bu_h^{n-l}}}_{\boldsymbol{\Cv}}^2,
		\label{eq:temp2}\\
		&\NLtwo[\Gamma_{\eta}^n]{\Avg{\nu\partial_\Bn \ul{\BU_h^{n-l,n}}}}^2
		\lesssim \avg{\nu} h^{-1}\tN{\ul{\Bu_h^{n-l}}}_{\boldsymbol{\Cv}}^2,
		\label{eq:temp4}
	\end{align}
	where the jump and average are defined as
	\begin{align*}
		\jump{\ul{\BU_h^{n-l,n}}} =\,& \Bu_{h,1}^{n-l}\circ \BX_\tau^{n,n-l}-\Bu_{h,2}^{n-l}\circ\BX_\tau^{n,n-l}, \\
		\Avg{\nu\partial_{\Bn} \ul{\BU_h^{n-l,n}}} =\,&
		\kappa_1\nu_1\partial_{\Bn} \big(\Bu_{h,1}^{n-l}\circ \BX_\tau^{n,n-l}\big)
		+\kappa_2\nu_2\partial_{\Bn} \big(\Bu_{h,2}^{n-l}\circ \BX_\tau^{n,n-l}\big).
	\end{align*}
	Together with \eqref{norm-equi0}, they show $\dN{\ul{\BU_h^{n-1,n}}}_{\boldsymbol{\Cv}} \lesssim
	\tN{\ul{\Bu_h^{n-1}}}_{\boldsymbol{\Cv}}$.
\end{remark}
\vspace{1mm}

\begin{lemma}\label{lem:BUhqh}
	Let $\ul{\Bu^{n-1}_h}$ be the discrete velocity at $t_{n-1}$ and $\ul{q_h}\in\Cq^n_h$. Extend $\Bu^{n-1}_{h,i}$ and $q_{h,i}$ from $\Omega^n_{h,i}$ to $\Omega$ according to the convention in \eqref{ieq:vh-ext} and denote their extensions still by $\Bu^{n-1}_{h,i}\in \BV_h(k,\Omega)$ and $q_{h,i}\in V_h(k-1,\Omega)$.
	Moreover, let $\hat q_{h,i} =\pi_{1,\Omega}\big(q_{h,i}\circ \BX_\tau^{n-1,n}\big)\in V_h(1,\Ct_h)$ be the Scott-Zhang interpolations of $q_{h,i}\circ \BX_\tau^{n-1,n}$ for $i=1,2$. Then $\ul{\hat q_h}=\big(\hat q_{h,1},\hat q_{h,2}\big)$ satisfies
	\begin{align*}
		\SN{\mathscr{B}^n_0(\ul{\BU_h^{n-1,n}},\ul{q_h}) -
			\mathscr{B}^{n-1}_0(\ul{\Bu_h^{n-1}},\ul{\hat q_h})}
		\lesssim h \tN{\ul{\Bu^{n-1}_h}}_{\boldsymbol{\Cv}}
		\tN{\ul{q_h}}_{1,\Cq}.
	\end{align*}
\end{lemma}
\begin{proof}
	For convenience, we denote the unit normals on $\Gamma^n_\eta$, $\Gamma^{n-1}_\eta$ by $\Bn_n$ and $\Bn_{n-1}$, respectively. Using the parametric representations of $\Gamma^n_\eta$ and $\Gamma^{n-1}_\eta$, we find that
	\begin{align*}
		\int_{\Gamma^n_\eta} \jump{\ul{\BU_h^{n-1,n}}\cdot \Bn_n}\Avg{\ul{q_h}}
		=& \int_0^L \jump{\ul{\BU_h^{n-1,n}}\cdot \Bn_n}
		\circ\chibf_n\Avg{\ul{q_h}}\circ\chibf_n \SN{\chibf_n'}
		=\sum_{j=0}^3 I_j,
	\end{align*}
	where
	\begin{align*}
		I_0 = & \int_0^L \left( \jump{\ul{\Bu_h^{n-1}}\cdot \Bn_{n-1}}
		\Avg{\ul{\hat q_h}}\right)\circ\chibf_{n-1} \SN{\chibf_{n-1}'}
		=\int_{\Gamma^{n-1}_\eta}\jump{\ul{\Bu_h^{n-1}}\cdot\Bn_{n-1}} \Avg{\ul{\hat q_h}}, \\
		I_1 = & \int_0^L (\jump{\ul{\BU_h^{n-1,n}}\cdot\Bn_n}\circ\chibf_n|\chibf_n'|
		-\jump{\ul{\Bu_h^{n-1}}\cdot \Bn_{n-1}}\circ\chibf_{n-1}
		\SN{\chibf_{n-1}'} )\Avg{\ul{q_h}}\circ\chibf_n,\\
		I_2 = &\int_0^L \jump{\ul{\Bu_h^{n-1}}\cdot\Bn_{n-1}}\circ\chibf_{n-1}
		\left(\Avg{\ul{q_h}}\circ\chibf_n
		-\Avg{\ul{q_h}\circ\BX_\tau^{n-1,n}}\circ\chibf_{n-1}\right)\SN{\chibf_{n-1}'},\\
		I_3=& \int_0^L \jump{\ul{\Bu_h^{n-1}}\cdot\Bn_{n-1}}\circ\chibf_{n-1} \left(\Avg{\ul{q_h}\circ\BX_\tau^{n-1,n}}
		-\Avg{\ul{\hat q_h}}\right)\circ\chibf_{n-1}\SN{\chibf_{n-1}'}.
	\end{align*}
	This shows that $\mathscr{B}^n_0(\ul{\BU_h^{n-1,n}},\ul{q_h})= \mathscr{B}^{n-1}_0(\ul{\Bu_h^{n-1}},\ul{\hat q_h}) +
	\sum_{j=1}^5I_j$, where
	\begin{align*}
		I_4 =\,&  -\sum_{i=1,2}\int_{\Omega^n_{\eta,i}}(\nabla\Bu_{h,i}^{n-1}) \circ \BX_\tau^{n,n-1}: (\bbJ_\tau^{n,n-1} -\bbI)^\top q_{h,i}, \\
		I_5 =\,& -\sum_{i=1,2}\int_{\Omega^{n-1}_{\eta,i}}
		\Div\Bu_{h,i}^{n-1}\hat q_{h,i}\big(\det{\bbJ_\tau^{n-1,n}}-1\big)\\
		&-\sum_{i=1,2}\int_{\Omega_{\eta,i}^n}
		\Div\Bu_{h,i}^{n-1}\big(q_{h,i}\circ \BX_\tau^{n-1,n}-\hat q_{h,i}\big)
		\det{\bbJ_\tau^{n-1,n}} \notag\\
		&-\sum_{i=1,2}\int_{\BX_\tau^{n,n-1}(\Omega^n_{\eta,i})\backslash\Omega^{n-1}_{\eta,i}}
		\Div\Bu_{h,i}^{n-1}\, q_{h,i}\circ \BX_\tau^{n-1,n} \det{\bbJ_\tau^{n-1,n}}\\
		&+\sum_{i=1,2}\int_{\Omega^{n-1}_{\eta,i}\backslash\BX_\tau^{n,n-1}(\Omega^n_{\eta,i})}
		\Div\Bu_{h,i}^{n-1}\, q_{h,i}\circ \BX_\tau^{n-1,n} \det{\bbJ_\tau^{n-1,n}}.
	\end{align*}
	Here $\bbA:\bbB=\sum_{i,j} \bbA_{i,j}\bbB_{i,j}$ stands for the Hadamard product of matrices $\bbA$ and $\bbB$.
	It suffices to estimate $I_0,\cdots,I_5$ term by term.

	By \eqref{ieq:Jtau} and \eqref{ieq:vh-ext}, the estimation for $I_4$ is easy and gives
	\begin{align*}
		\SN{I_4} \lesssim \tau \sum_{i=1,2}
		\big|\Bu_{h,i}^{n-1}\big|_{\Honev} \NLtwo[\Omega^n_{\eta,i}]{q_{h,i}}
		\lesssim h \tN{\ul{\Bu^{n-1}_h}}_{\boldsymbol{\Cv}}
		\big\|\nu^{-\frac12}\ul{q_h}\big\|_{0,\Omega^n_\eta}.  
	\end{align*}
	Thanks to \eqref{int-errK}, \eqref{ieq:Jtau}, and \eqref{ieq:vh-ext}, we have the error estimates
	\begin{align}\label{err:qh}
		\NLtwo{q_{h,i}\circ \BX_\tau^{n-1,n}-\hat q_{h,i}} 
		\lesssim& h\big|q_{h,i}\circ \BX_\tau^{n-1,n}\big|_{\Hone} \notag \\
		\lesssim& h\SNHone{q_{h,i}}\lesssim h\NHone[\Omega^n_{h,i}]{q_{h,i}}.
	\end{align}
	Using \eqref{ieq:Jtau}, \eqref{norm-equal-q}, \eqref{int-errK}, and Lemma~\ref{lem:uX}, we obtain
	\begin{align*}
		\SN{I_5} \lesssim\,& h\sum_{i=1,2}\big|\Bu^{n-1}_{h,i}\big|_{\Honev}\NHone{q_{h,i}}
		\lesssim h \tN{\ul{\Bu^{n-1}_h}}_{\boldsymbol{\Cv}}\tN{\ul{q_h}}_{1,\Cq}.  
	\end{align*}

	Note that $\Bn_n=(-\chi_{n,2}',\chi_{n,1}')/\SN{\chibf_n'}$. From Theorem~\ref{thm:bdr} and \eqref{ieq:Jtau}, we have
	\begin{align*}
		\N{\chibf_n-\chibf_{n-1}}_{\BC^1([0,L])}\leq
		&\N{\chibf_n-\BX_\tau^{n-1,n}\circ\chibf_{n-1}}_{\BC^1([0,L])}\\
		&+\N{\BX_\tau^{n-1,n}\circ\chibf_{n-1} - \chibf_{n-1}}_{\BC^1([0,L])}
		\lesssim \tau.
	\end{align*}
	This implies
	$\N{\Bn_n\circ\chibf_n-\Bn_{n-1}\circ\chibf_{n-1} |\chibf_n'|^{-1}|\chibf_{n-1}'|}_{\BL^\infty(0,L)}
	\lesssim \tau$.
	By Theorem~\ref{thm:bdr}, Lemma~\ref{lem:trace2}, and norm equivalence, $I_1$ can be estimated as follows
	\begin{align*}
		\SN{I_1} \lesssim\, &\int_0^L
		\SN{\jump{\ul{\Bu_h^{n-1}}(\BX_\tau^{n,n-1}\circ\chibf_n)}
			-\jump{\ul{\Bu_h^{n-1}}(\chibf_{n-1})}}
		\SN{\Avg{q_h}\circ\chibf_n}\SN{\chibf_n'} \notag\\
		&+\tau\int_0^L  \SN{\jump{\ul{\Bu_h^{n-1}}\cdot \Bn_{n-1}}\circ\chibf_{n-1}}
		\SN{\Avg{q_h}\circ\chibf_n}\SN{\chibf_n'} \notag\\
		\lesssim\,&\tau^{k+2}h^{-\frac12}\sum_{i=1,2}
		\SNHonev[\Omega^{n-1}_{h,i}]{\Bu_{h,i}^{n-1}}\NLtwo[\Gamma^n_\eta]{\Avg{\ul{q_h}}}\\
		&+\tau \NLtwov[\Gamma_\eta^{n-1}]{\jump{\ul{\Bu_h^{n-1}}}}
		\NLtwo[\Gamma^n_\eta]{\Avg{\ul{q_h}}} \notag \\
		\lesssim\,&h\, \tN{\ul{\Bu_h^{n-1}}}_{\boldsymbol{\Cv}} \tN{\ul{q_h}}_{\Cq}.  
	\end{align*}
	Similarly, by Theorem~\ref{thm:bdr} and inverse estimate, $I_2$ is estimated as follows
	\begin{align*}
		\SN{I_2}\lesssim\,&\int_0^L \SN{\jump{\ul{\Bu_h^{n-1}}\cdot\Bn_{n-1}}\circ\chibf_{n-1}}
		\Big|\Avg{\ul{q_h}}\circ\chibf_n -\Avg{\ul{q_h}\circ\BX_\tau^{n-1,n}}\circ\chibf_{n-1}\Big|
		|\chibf_{n-1}'| \notag\\
		\lesssim\,& \tau^{k+2}h^{-3/2}\NLtwo[\Gamma^{n-1}_\eta]{\jump{\ul{\Bu_h^{n-1}}}}
		\sum_{i=1,2}\NLtwo[\Omega^{n}_{h,i}]{q_{h,i}}  \notag \\
		\lesssim\,& h^{k+1} \tN{\ul{\Bu_h^{n-1}}}_{\boldsymbol{\Cv}} \tN{\ul{q_h}}_{\Cq}.  
	\end{align*}
	Finally, from \eqref{eq:q-qh} and the interpolation error estimate \eqref{int-errK}, we have
	\begin{align*}
		\SN{I_{3}}\lesssim\,& \NLtwo[\Gamma^{n-1}_\eta]{\jump{\ul{\Bu_h^{n-1}}}}
		\NLtwo[\Gamma^{n-1}_\eta]{\Avg{\ul{q_h}\circ\BX_\tau^{n-1,n}-\ul{\hat q_h}}}\\
		\lesssim\,& h \tN{\ul{\Bu_h^{n-1}}}_{\boldsymbol{\Cv}} \tN{\ul{q_h}}_{1,\Cq}. 
	\end{align*}
	The proof is finished.
\end{proof}

\begin{corollary}\label{cor:stokes-L2-1}
	Let Assumption~\ref{ass-2}, Assumption~\ref{ass-3},
	and the assumptions in Lemma~\ref{lem:uhX} be satisfied. Suppose
	$k\ge 2$ and $h=O(\tau)$.
	Then $\big(\ul{\hat{\BU}_h^{n-1,n}},\ul{\hat{p}_h^{n-1,n}}\big)
	=\Cs^n\big(\ul{\BU_h^{n-1,n}},\ul{0},0\big)$ satisfies
	\begin{align}
		\tN{\big(\ul{\hat{\BU}_h^{n-1,n}},\ul{\hat{p}_h^{n-1,n}}
			\big)}_{\boldsymbol{\Cv},\Cq}   \lesssim\,&
		\tN{\ul{\Bu_h^{n-1}}}_{\boldsymbol{\Cv}},\label{eq:stab-proj0}\\
		\N{\ul{\BU_h^{n-1,n}} - \ul{\hat{\BU}_h^{n-1,n}}}_{0,\Omega_\eta^n}
		\lesssim \,& \gamma_1\nu_2 h\Big\{\nu_2^{-\frac12}\big\|\Lambda^k\ul{\bbU_h^{n-1}} -\tau\ul{\Bf}^{n-1}\big\|_{0,\Omega^{n-1}_\eta}\notag \\
		&+h\big\|\nu^{-\frac12}\nabla\ul{p_h^{n-1}}\big\|_{0,\Omega^{n-1}_\eta}\Big\}
		\notag\\
		&+h\Big\{\tN{\Bu_h^{n-1}}_{\boldsymbol{\Cv}}
		+\mathscr{J}_p^{n-1}(\ul{p_h^{n-1}},\ul{p^{n-1}_h})^{\frac12}\Big\}.	
		\label{eq:stab-proj1}
	\end{align}
\end{corollary}
\begin{proof}
	The stability \eqref{eq:stab-proj0} comes directly from Theorem~\ref{thm:stokes-proj} and Remark~\ref{rem:Un}. It is left to prove \eqref{eq:stab-proj1}. From Theorem~\ref{thm:stokes-L2}, Remark~\ref{rem:Un}, and inequality \eqref{eq:stab-proj0}, we have
	\begin{align}
		\dN{\ul{\BU_h^{n-1,n}} - \ul{\hat{\BU}_h^{n-1,n}}}_{0,\Omega^n_\eta}
		\lesssim\,& h \tN{\ul{\BU_h^{n-1,n}}- \ul{\hat{\BU}_h^{n-1,n}}}_{\boldsymbol{\Cv}}	
		+\dN{\Cb^n_0\ul{\BU_h^{n-1,n}}}_{1,(\Cq^n_h)'} \notag\\
		\lesssim\,& h \tN{\Bu_h^{n-1}}_{\boldsymbol{\Cv}}	
		+\dN{\Cb^n_0\ul{\BU_h^{n-1,n}}}_{1,(\Cq^n_h)'}. \label{ieq:err-Un1}
	\end{align}
	For any $\ul{0}\ne \ul{q_h}\in\Cq^n_h$, we extend $q_{h,i}$ to $\Omega$ according to the convention in \eqref{ieq:vh-ext} and let $\hat q_{h,i} =\pi_{1,\Omega}\big(q_{h,i}\circ \BX_\tau^{n-1,n}\big)\in V_h(1,\Ct_h)$ be the Scott-Zhang interpolation for $i=1,2$. Setting $\ul{\hat q_h}=\big(\hat q_{h,1},\hat q_{h,2}\big)$ and using Lemma~\ref{lem:BUhqh}, we find that
	\begin{align}\label{eq:Bn1-1}
		\SN{\big(\Cb^n_0\ul{\BU_h^{n-1}}\big)(\ul{q_h})} \le
		\SN{\mathscr{B}^{n-1}_0(\ul{\Bu_h^{n-1}},\ul{\hat q_h})}
		+ Ch\tN{\ul{\Bu_h^{n-1}}}_{\boldsymbol{\Cv}} \tN{\ul{q_h}}_{1,\Cq}.
	\end{align}
	It suffices to estimate $\mathscr{B}^{n-1}_0(\ul{\Bu_h^{n-1}},\ul{\hat q_h})$.

	From equation \eqref{eq:disc-p}, we have
	\ben
	\mathscr{B}^{n-1}_0(\ul{\Bu_h^{n-1}},\ul{\hat q_h})
	=\mathscr{J}_p^{n-1}(\ul{p_h^{n-1}},\ul{\hat q_h})
	+ \mathscr{R}_h^{n-1}(\ul\bbU_h^{n-1},\ul{p_h^{n-1}};\ul{\hat q_h}).
	\een
	By norm equivalence and \eqref{int-errK}, the first term on the right-hand side satisfies
	\begin{align*}
		\SN{\mathscr{J}_p^{n-1}(\ul{p_h^{n-1}},\ul{\hat q_h})}
		\lesssim\,& h\mathscr{J}_p^{n-1}(\ul{p_h^{n-1}},\ul{p^{n-1}_h})^{\frac12}
		\sum_{i=1,2}\nu^{-\frac12}_i\SNHone[\Omega^n_{h,i}]{\hat q_{h,i}} \notag \\
		\lesssim\,& h \mathscr{J}_p^{n-1}(\ul{p_h^{n-1}},\ul{p^{n-1}_h})^{\frac12}
		\tN{\ul{q_h}}_{1,\Cq}.
	\end{align*}
	By inverse estimates and \eqref{int-errK}, the second term satisfies
	\begin{align*}
		\SN{\mathscr{R}_h^{n-1}(\ul\bbU_h^{n-1},\ul{p_h^{n-1}};\ul{\hat q_h})}
		\lesssim\,& \gamma_1\nu_2 h \left\{\nu_2^{-\frac12}\big\|\Lambda^k\ul{\bbU_h^{n-1}}
		-\tau\ul{\Bf}^{n-1}\big\|_{0,\Omega^{n-1}_\eta}
		+\tN{\ul{\Bu_h^{n-1}}}_{\boldsymbol{\Cv}}\right. \\
		&\left.+h\big\|\nu^{-\frac12}\nabla\ul{p_h^{n-1}}\big\|_{0,\Omega^{n-1}_\eta}
		\right\}\tN{\ul{q_h}}_{1,\Cq} .
	\end{align*}
	Together with \eqref{eq:Bn1-1} and \eqref{eq:normf}, they yield the desired result
	\begin{align*}
		\dN{\Cb^n_0\ul{\BU_h^{n-1,n}}}_{1,(\Cq^n_h)'}
		\lesssim \,& \gamma_1\nu_2 h\Big\{\nu_2^{-\frac12}\big\|\Lambda^k\ul{\bbU_h^{n-1}}
		-\tau\ul{\Bf}^{n-1}\big\|_{0,\Omega^{n-1}_\eta}\\
		&+h\big\|\nu^{-\frac12}\nabla\ul{p_h^{n-1}}\big\|_{0,\Omega^{n-1}_\eta}\Big\} \\
		&+h\Big\{\tN{\Bu_h^{n-1}}_{\boldsymbol{\Cv}}
		+\mathscr{J}_p^{n-1}(\ul{p_h^{n-1}},\ul{p^{n-1}_h})^{\frac12}\Big\}.
	\end{align*}
	The proof is finished.
\end{proof}

\section{Stability of numerical solutions}\label{sec:stab}

The purpose of this section is to prove the stability of numerical solutions. It also paves the way to a priori finite element error estimates. First we cite the telescope formulas for BDF schemes from \cite[Section~2 and Appendix~A]{liu13}.

\begin{lemma}\label{lem:tel}
	Suppose $1\le k\le 4$ and let $a_0=1$ and $a_1=\delta_{k,3}+\delta_{k,4}$,
	where $\delta_{ij}$ is the Kronecker delta. Define $\Lambda^0 \ul{\bbU^{n}_h}=\ul{\Bu_h^n}$ and
	\ben
	\Psibf^k_l (\ul{\bbU^{n}_h}) =\sum_{j=1}^{l} c_{l,j}^k\ul{\BU_{h}^{n+1-j,n}},\qquad
	\Phibf_l^k (\ul{\bbU^{n}_h})= \sum_{j=1}^{l} c_{l,j}^k \ul{\BU_{h}^{n-j,n}}.
	\een
	where the parameters $c_{l,j}^k$, $1\le j\le l\le k+1$, are given in \cite[Table~2.2]{liu13}. Then
	\begin{equation*}
		\big(\Lambda^k\ul{\bbU^{n}_h}\big)
		\sum_{l=0}^1 a_l \Lambda^l \ul{\bbU^{n}_h} =\sum_{l=1}^{k+1}
		\big[\Psibf_l^k (\ul{\bbU^{n}_h})\big]^2
		- \sum_{l=1}^k \big[\Phibf_l^k (\ul{\bbU^{n}_h})\big]^2 .
	\end{equation*}
\end{lemma}

\begin{lemma}\label{lem:phL2}
	Assume that the penalty parameter $\gamma_0$ in $\mathscr{J}^n_0$ is large enough. Let $\ul{p^n_h}$ be the discrete pressure at $t_n$. Then
	\begin{align*}
		\tN{\ul{p^n_h}}_{\Cq}^2\lesssim\,&  \mathscr{J}^n_p(\ul{p^n_h},\ul{p^n_h})
		+\frac{1}{\nu_2\tau^2}\sum_{j=0}^k \Big\{\dN{\ul{\Bu_h^{n-j}}}_{0,\Omega^{n-j}_{\eta}}^2
		+ h^{k+3}\nu_2^{-1}\tN{\ul{\Bu^{n-j}_h}}_{\boldsymbol{\Cv}}^2 \Big\} \\
		&+\tN{\ul{\Bu^n_h}}_{\boldsymbol{\Cv}}^2
		+\dN{\nu^{-\frac 12}\ul{\Bf^{n}}}_{0,\Omega^n_{\eta}}^2.
	\end{align*}
\end{lemma}
\begin{proof}
	By Lemma~\ref{lem:infsup}, there is a 
	$\ul{\Bv_h}\equiv (\Vv_h|_{\Omega^n_{h,1}},\Vv_h|_{\Omega^n_{h,2}})$,
	$\Vv_h\in\BV_{h,0}(k,\Omega)$, satisfies
	\begin{align}\label{est:p1}
		\dN{\nu^{\frac 12}\ul{\Bv_h}}_{0,\Omega^n_\eta}^2\lesssim
		\tN{\ul{\Bv_h}}_{\boldsymbol{\Cv}}^2
		\lesssim \tN{\ul{p^n_h}}_{\Cq}^2 \lesssim
		\Aprod[\Omega^n_\eta]{\Div\ul{\Bv_h},\ul{p^n_h}}
		+ \mathscr{J}^n_p(\ul{p^n_h},\ul{p^n_h}),
	\end{align}
	where we have used Poincar\'{e}'s inequality in the first inequality due to $\Vv_h\in\zbHonev$.
	From \eqref{eq:temp3} and \eqref{norm-equal-v}, we have
	\begin{align}\label{ieq:LamU}
		\dN{\Lambda^k\ul{\bbU_h^{n}}}_{0,\Omega^n_{\eta}}^2
		\lesssim \sum_{j=0}^k \Big\{\dN{\ul{\Bu_h^{n-j}}}_{0,\Omega^{n-j}_{\eta}}^2
		+\nu_2^{-1}h^{k+3}\mathscr{J}^n_\Bu\big(\ul{\Bu_h^{n-j}},\ul{\Bu_h^{n-j}}\big)\Big\}.
	\end{align}
	Using \eqref{eq:disc-u} and \eqref{eq:trace3}, we find that
	\begin{align*}
		\Aprod[\Omega^n_\eta]{\Div\ul{\Bv_h},\ul{p^n_h}}
		=\,& \tau^{-1}\Aprod[\Omega^n_\eta]{\Lambda^k \ul\bbU_h^n,\ul{\Bv_h}}
		+\mathscr{A}_h^n(\ul{\Bu_h^n},\ul{\Bv_h})\\
		&+\int_{\Gamma^n_\eta}\jump{\ul{\Bv_h}}\cdot\Bn \Avg{p^n_h}
		- \Aprod[\Omega^n_\eta]{\ul\Bf^n,\ul{\Bv_h}} \notag \\
		\lesssim\,&\tau^{-1}\dN{\nu^{\frac 12}\ul{\Bv_h}}_{0,\Omega^n_\eta}
		\sum_{j=0}^k \dN{\nu^{-\frac 12}\ul{\Bu_h^{n-j}}}_{0,\Omega^{n-j}_{\eta}}\\
		&+\sum_{j=0}^k \nu^{-1}_2h^{(k+3)/2}\mathscr{J}^n_\Bu\big(\ul{\Bu_h^{n-j}},\ul{\Bu_h^{n-j}}\big)^{\frac12}\\
		&+\Big(\tN{\ul{\Bu^n_h}}_{\boldsymbol{\Cv}}
		+\gamma_0^{-\frac12}\tN{\ul{p^n_h}}_{\Cq}\Big)
		\tN{\ul{\Bv_h}}_{\boldsymbol{\Cv}}\\
		&+\dN{\nu^{\frac 12}\ul{\Bv_h}}_{0,\Omega^n_\eta}
		\dN{\nu^{-\frac 12}\ul{\Bf^{n}}}_{0,\Omega^n_{\eta}}.
	\end{align*}
	The proof is finished by using \eqref{est:p1} and assuming that $\gamma_0$ is large enough.
\end{proof}
\vspace{1mm}

\begin{center}
	\fbox{\parbox{0.975\textwidth}
		{\begin{theorem}\label{thm:uh-stab}
				Suppose $2\leq k\leq 4$, $\tau\le h=O(\tau)$, and $\gamma_0,\gamma_1^{-1}$ are large enough.
				Let $\ul{\Bu_h^n}$ be the solution to the discrete problem \eqref{eq:disc}.
				There is an $h_0>0$ small enough such that,
				for any $h\in(0, h_0]$ and $m\ge k$,
				\begin{align*}
					&\big\|\ul{\Bu^m_h}\big\|^2_{0,\Omega^n_{\eta}}
					+\sum_{n=k}^m\tau \big(\tN{\ul{\Bu_h^n}}^2_{\boldsymbol{\Cv}}
					+\gamma_1\nu_2 h^2\tN{\ul{p_h}}_{1,\Cq}^2\big)\\
					\lesssim& \sum_{n=0}^m\tau\NLtwov[\Omega^n_\eta]{\ul\Bf^n}^2
					+\sum_{n=0}^{k-1}\big(\big\|\ul{\Bu_h^n}\big\|_{0,\Omega^n_\eta}^2 +
					\tau\tN{\ul{\Bu_h^n}}^2_{\boldsymbol{\Cv}}\big).
				\end{align*}
	\end{theorem}}}
\end{center}
\vspace{1mm}

\begin{proof}
	Without loss of generality, we only prove the lemma for $k=4$. The proofs for $k< 4$ are similar.
	The discrete problem can be written as follows
	\begin{align}\label{eq:Uh}
		&\Aprod[\Omega^n_{\eta}]{\Lambda^4 \ul{\bbU_h^n}, \ul{\Bv_h}}
		+ \tau \mathscr{K}_0((\ul{\Bu_h^n},\ul{p_h^n}),(\ul{\Bv_h},\ul{q_h})) \notag \\
		&+ \gamma_1\nu_2 \tau h^2\mathscr{B}^n_1(\ul{\Bu_h^n},\ul{q_h})
		+\gamma_1\nu_2 \tau h^2\Aprod[\Omega^n_\eta]{\nu^{-1}\nabla \ul{p_h^n},\nabla \ul{q_h}}
		\notag \\
		=\,&\tau\Aprod[\Omega^n_{\eta}]{\ul\Bf^n, \ul{\Bv_h}}
		-\gamma_1\nu_2 \tau h^2\sum_{i=1,2}\nu_i^{-1}
		\Big(\Bf^n_i-\sum_{j=1}^k \lambda_j^k \BU_h^{n-j,n},
		\nabla q_{h,i}\Big)_{\Omega^n_{\eta,i}}.
	\end{align}	
	For $k=4$, the coefficients $a_0=a_1=1$ in Lemma~\ref{lem:tel}. From \cite{liu13}, the coefficients $\lambda_0^k,\cdots,\lambda_k^k$ for the BDF-$k$ scheme also imply that $\Lambda^0\ul{\bbU_h^n} =\ul{\Bu^n_h}$
	and $\Lambda^1\ul{\bbU_h^n} =\ul{\Bu^n_h} - \ul{\BU_h^{n-1,n}}$. Then
	\ben
	\sum_{l=0,1} a_l\Lambda^l \ul{\bbU^{n}_h}
	= 2\ul{\Bu^n_h} - \ul{\BU_h^{n-1,n}} .
	\een
	Define $\big(\ul{\hat\BU_h^{n-1,n}}, \ul{\hat p_h^{n-1,n}}\big) =\Cs^n(\ul{\BU_h^{n-1,n}},\ul{0},0)$
	, take $(\ul{\Bv_h},\ul{q_h})=\big(2 \ul{\Bu_h^n} -\ul{\hat\BU_h^{n-1,n}},\ul{0}\big)$
	in \eqref{eq:Uh}. Using Lemma~\ref{lem:tel}, we have
	\begin{equation}\label{eq:stab0}
		\sum_{l=1}^5 \dN{\Psibf_l^4(\underline{\bbU_h^{n}})}_{0,\Omega^n_{\eta}}^2 -
		\sum_{l=1}^4 \dN{\Phibf_l^4(\underline{\bbU_h^{n}})}_{0,\Omega^n_{\eta}}^2
		+ \tau \mathscr{K}_0((\ul{\Bu_h^n},\ul{p_h^n});(\ul{\Bv_h},\ul{0}))
		= \tau (A_1 +A_2),
	\end{equation}
	where $A_1=\Aprod[\Omega^n_{\eta}]{\ul\Bf^n, \ul{\Bv_h}}$
	and $A_2 = \tau^{-1}\Aprod[\Omega^n_{\eta}]{\Lambda^4\ul{\bbU_h^{n}},\ul{\hat\BU_h^{n-1,n}} - \ul{\BU_h^{n-1,n}}}$. We are going to estimate each term in \eqref{eq:stab0} separately.
	
	By Corollary~\ref{cor:stokes-L2-1} and inequality \eqref{ieq:LamU}, we have
	\begin{align}
		\SN{A_1} \le \,& \N{\ul\Bf^n}_{0,\Omega^n_\eta}\Big\{2\dN{\ul{\Bu^n_h}}_{0,\Omega^n_\eta}
		+\dN{\ul{\BU_h^{n-1,n}}}_{0,\Omega^n_\eta}
		+\dN{\ul{\BU_h^{n-1,n}}- \ul{\hat\BU_h^{n-1,n}}}_{0,\Omega^n_\eta}\Big\}\notag \\
		\le \,& C\N{\ul\Bf^n}_{0,\Omega^n_\eta}\Big\{\dN{\ul{\Bu^{n-1}_h}}_{0,\Omega^{n-1}_\eta}
		+\dN{\ul{\Bu^n_h}}_{0,\Omega^n_\eta}
		\notag \\
		&+ h\tN{\ul{\Bu_h^{n-1}}}_{\boldsymbol{\Cv}} 
		+\gamma_1\nu_2^{\frac12}h\big\|\Lambda^k\ul{\bbU_h^{n-1}}
		-\tau\ul{\Bf}^{n-1}\big\|_{0,\Omega^{n-1}_\eta}\notag \\
		&+\gamma_1\nu_2 h^2\big\|\nu^{-\frac12}\nabla\ul{p_h^{n-1}}\big\|_{0,\Omega^n_\eta} 
		+h\mathscr{J}_p^{n-1}(\ul{p_h^{n-1}},\ul{p^{n-1}_h})^{\frac12}\Big\}\notag \\
		\le& C\Big(1+\gamma_1\nu_2^{\frac12}h\Big)\sum_{j=0}^5
		\Big\{\dN{\ul{\Bu_h^{n-j}}}^2_{0,\Omega^{n-j}_{\eta}}
		+\nu_2^{-1}h^4\tN{\ul{\Bu_h^{n-j}}}_{\boldsymbol{\Cv}}^2
		+\N{\ul\Bf^{n-j}}_{0,\Omega^{n-j}_\eta}^2 \Big\} \notag \\
		&+h^2\tN{\ul{\Bu_h^{n-1}}}_{\boldsymbol{\Cv}}^2+(\gamma_1\nu_2)^2 h^4\big\|\nu^{-\frac12}\nabla\ul{p_h^{n-1}}\big\|_{0,\Omega^{n-1}_\eta}^2
		+h^2\mathscr{J}_p^{n-1}(\ul{p_h^{n-1}},\ul{p^{n-1}_h}). \label{est:A1}
	\end{align}
	For a parameter $\varepsilon \in (0,1)$ to be specified later,
	$A_2$ can be estimated similarly as follows
	\begin{align}
		\SN{A_2}\le \,& C\dN{\Lambda^4\ul{\bbU_h^{n}}}_{0,\Omega^n_{\eta}}
		\Big\{\gamma_1\nu_2^{\frac 12}\big\|\Lambda^4\ul{\bbU_h^{n-1}} -\tau\ul{\Bf}^{n-1}\big\|_{0,\Omega^{n-1}_\eta}
		+ \tN{\ul{\Bu_h^{n-1}}}_{\boldsymbol{\Cv}}\notag \\
		& +\gamma_1\nu_2 h\big\|\nu^{-\frac 12}\nabla\ul{p_h^{n-1}}\big\|_{0,\Omega^{n-1}_\eta}
		+\mathscr{J}_p^{n-1}(\ul{p_h^{n-1}},\ul{p^{n-1}_h})^{\frac12}\Big\} \notag\\
		\le\,& \gamma_1\nu_2^{\frac 12}h^2\N{\ul\Bf^{n-1}}_{0,\Omega^{n-1}_\eta}^2
		+ \frac{C+C\gamma_1\nu_2^{\frac12}}{\varepsilon}
		\sum_{j=0}^5\Big\{\dN{\ul{\Bu_h^{n-j}}}^2_{0,\Omega^{n-j}_{\eta}}
		+\frac{h^4}{\nu_2}\tN{\ul{\Bu_h^{n-j}}}_{\boldsymbol{\Cv}}^2\Big\}\notag \\
		&+\varepsilon \Big\{\tN{\ul{\Bu_h^{n-1}}}_{\boldsymbol{\Cv}}^2
		+(\gamma_1\nu_2)^2 h^2\big\|\nu^{-\frac 12}\nabla\ul{p_h^{n-1}}\big\|_{0,\Omega^{n-1}_\eta}^2
		+\mathscr{J}_p^{n-1}(\ul{p_h^{n-1}},\ul{p^{n-1}_h})\Big\} .   \label{est:A2}
	\end{align}
	Note that $\Phibf_l^4 (\underline{\bbU_h^n})=\Psibf_l^4 (\underline{\bbU_h^{n-1}})\circ \BX_{\tau}^{n,n-1}$.
	By \eqref{ieq:vhXn} and \eqref{eq:temp3}, we have
	\begin{align}\label{stab-est1}
		\dN{\Phibf_l^4 (\underline{\bbU_h^n})}_{0,\Omega_{\eta}^n}^2
		\le \dN{\Psibf_l^4 (\underline{\bbU_h^{n-1}})}_{0,\Omega_{\eta}^{n-1}}^2
		+ C\tau \sum_{j=1}^l\Big\{\big\|\ul{\Bu_h^{n-j}}\big\|_{0,\Omega^{n-j}_\eta}^2
		+ \frac{h^4}{\nu_2}\tN{\ul{\Bu_h^{n-j}}}^2_{\boldsymbol{\Cv}}\Big\}.
	\end{align}

	Now it is left to estimate $\mathscr{K}_0^n((\ul{\Bu_h^n},\ul{p_h^n}),(\ul{\Bv_h},\ul{0}))$.
	Since $\mathscr{A}^n_h$ is symmetric, from \eqref{def-proj} and \eqref{eq:disc-p}, we know that
	\begin{align}
		\mathscr{K}_0^n((\ul{\Bu_h^n},\ul{p_h^n}),(\ul{\Bv_h},\ul{0}))=\,&
		2\mathscr{A}_h^n(\ul{\Bu_h^n},\ul{\Bu_h^n})
		+2\mathscr{B}^n_0(\ul{\Bu_h^n},\ul{p_h^n})
		- \mathscr{A}_h^n(\ul{\hat\BU_h^{n-1,n}},\ul{\Bu_h^n}) \notag \\
		&-\mathscr{B}_0^n(\ul{\hat\BU_h^{n-1,n}},\ul{p_h^n}) \notag \\ 
		=\,&2\mathscr{A}_h^n(\ul{\Bu_h^n},\ul{\Bu_h^n})
		+2\mathscr{J}_p^n(\ul{p_h^n},\ul{p^n_h})
		+2\mathscr{R}_h^n(\ul\bbU_h^n,\ul{p_h^n};\ul{p^n_h})\notag \\
		&-a_h^n(\ul{\BU_h^{n-1,n}},\ul{\Bu_h^n}) 
		+\mathscr{R}_h^n(\ul\bbU_h^n,\ul{p_h^n};\ul{\hat p_h^{n-1,n}}).   \label{stab-est2}
	\end{align}
	By Lemma~\ref{lem:ch-up}, \eqref{eq:temp3}, and \eqref{ieq:LamU}, we deduce that
	\begin{align}
		\mathscr{R}_h^n(\ul\bbU_h^n,\ul{p_h^n};\ul{p^n_h}) =\,& \gamma_1\nu_2 h^2
		\Aprod[\Omega^n_\eta]{\tau^{-1} \Lambda^4\ul{\bbU_h^n}- \nu \Delta_h \ul{\Bu_h^n}
			+\nabla \ul{p_h^n}-\ul{\Bf}^n,\nu^{-1}\nabla \ul{p^n_h}} \notag\\
		\ge\,&\frac{3}{4} \gamma_1\nu_2 h^2\big\|\nu^{-\frac12}\nabla \ul{p^n_h}\big\|_{0,\Omega^n_\eta}^2
		-C\gamma_1 h^2\N{\ul\Bf^n}_{0,\Omega^n_\eta}^2
		-C\gamma_1\nu_2\tN{\ul{\Bu_h^n}}_{\boldsymbol{\Cv}}^2\notag \\
		&- C \gamma_1 \sum_{j=0}^4 \Big(\dN{\ul{\Bu_h^{n-j}}}_{0,\Omega^{n-j}_{\eta}}^2
		+\nu_2^{-1}h^4\tN{\ul{\Bu_h^{n-j}}}_{\boldsymbol{\Cv}}^2\Big)
		.
		\label{est:Rnh1}
	\end{align}
	Similarly, using inverse estimates, \eqref{eq:stab-proj0}, and \eqref{ieq:LamU}, we have
	\begin{align}
		&\mathscr{R}_h^n(\ul\bbU_h^n,\ul{p_h^n};\ul{\hat p_h^{n-1,n}})
		=\, \gamma_1\nu_2 h^2\Aprod[\Omega^n_\eta]{\frac{1}{\tau} \Lambda^4\ul{\bbU_h^n}- \nu \Delta_h \ul{\Bu_h^n}
			+\nabla \ul{p_h^n}-\ul{\Bf}^n,\frac{1}{\nu}\nabla \ul{\hat p_h^{n-1,n}}} \notag\\
		\le \,& C\gamma_1 \tN{\ul{\hat p_h^{n-1,n}}}_{\Cq}
		\Big(\dN{\Lambda^4 \ul{\bbU_h^n}}_{0,\Omega_\eta^n}
		+h\dN{\ul\Bf^n}_{0,\Omega^n_\eta}\Big) \notag \\ 
		&+C\gamma_1\nu_2 \tN{\ul{\hat p_h^{n-1,n}}}_{\Cq}
		\Big(\tN{\ul{\Bu_h^n}}_{\boldsymbol{\Cv}}
		+\dN{\nu^{-\frac 12}\nabla \ul{p_h^n}}_{0,\Omega_\eta^n}\Big) \notag \\
		\le \,&  C \gamma_1 \sum_{j=0}^4\Big(\dN{\ul{\Bu_h^{n-j}}}_{0,\Omega^{n-j}_{\eta}}^2
		+\nu_2^{-1}h^4\tN{\ul{\Bu_h^{n-j}}}_{\boldsymbol{\Cv}}^2\Big) 
		+C\gamma_1\nu_2 \tN{\ul{\Bu_h^n}}_{\boldsymbol{\Cv}}^2\notag \\
		\,&+C \gamma_1 h^2\N{\ul\Bf^n}_{0,\Omega^n_\eta}^2 
		+\frac{1}{2}\gamma_1\nu_2 h^2\dN{\nu^{-\frac{1}{2}} \nabla \ul{p_h^n}}_{0,\Omega_\eta^n}^2.
		\label{est:Rnh2}
	\end{align}
	By \eqref{eq:trace2} and \eqref{eq:temp3}--\eqref{eq:temp4}, we have
	\begin{align*}
		&\SN{a^n_h(\ul{\BU_h^{n-1,n}},\ul{\Bu_h^n})}
		\le\,\big\|\nu^{\frac12}\nabla\ul{\Bu_h^n}\big\|_{0,\Omega^n_\eta}
		\Big\{\big\|\nu^{\frac12}\nabla\ul{\Bu_h^{n-1}}\big\|_{0,\Omega_\eta^{n-1}}
		+Ch\tN{\ul{\Bu_h^{n-1}}}_{\boldsymbol{\Cv}}\Big\} \\
		&+C\gamma_0^{-\frac12}\mathscr{J}^n_0(\ul{\Bu_h^{n}},\ul{\Bu_h^{n}})^{\frac12}
		\tN{\ul{\Bu_h^{n-1}}}_{\boldsymbol{\Cv}}\\
		&+C\tN{\ul{\Bu_h^n}}_{\boldsymbol{\Cv}}\Big\{
		\gamma_0^{-\frac 12}\mathscr{J}^{n-1}_0(\ul{\Bu_h^{n-1}},\ul{\Bu_h^{n-1}})^{\frac12}
		+Ch^{\frac 12}\tN{\ul{\Bu_h^{n-1}}}_{\boldsymbol{\Cv}} \Big\}\\
		&+(1+C\tau)\mathscr{J}^n_0(\ul{\Bu_h^{n}},\ul{\Bu_h^{n}})^{\frac12}\Big\{
		\mathscr{J}^{n-1}_0(\ul{\Bu_h^{n-1}},\ul{\Bu_h^{n-1}})^{\frac12} +
		C\sqrt{h\gamma_0}\tN{\ul{\Bu_h^{n-1}}}_{\boldsymbol{\Cv}} \Big\} \\
		\le\,& \left\{0.5+\varepsilon
		+C\varepsilon^{-1}(h+\gamma_0^{-1}+h\gamma_0)\right\}\Big\{
		\tN{\ul{\Bu_h^{n}}}_{\boldsymbol{\Cv}}^2
		+\tN{\ul{\Bu_h^{n-1}}}_{\boldsymbol{\Cv}}^2\Big\}.
	\end{align*}
	Choosing $\gamma_0$ large enough such that $2C\gamma_0^{-1}\leq \varepsilon^2$
	and choosing  $h$ small enough such that
	$2C(1+\gamma_0)h\le \varepsilon^2$, we find that
	\begin{align}\label{est:anh}
		\SN{a^n_h(\ul{\BU_h^{n-1,n}},\ul{\Bu_h^n})}
		\le (0.5+2\varepsilon)\Big(\tN{\ul{\Bu_h^{n}}}_{\boldsymbol{\Cv}}^2
		+\tN{\ul{\Bu_h^{n-1}}}_{\boldsymbol{\Cv}}^2 \Big).
	\end{align}
	
	Now we assume that $h$ and $\gamma_1$ are small enough such that
	$\nu_2^{-1}h^4\leq \varepsilon^2\ll 1 $ and $ \gamma_1\lesssim 1$.
	Then inserting \eqref{est:Rnh1}--\eqref{est:anh} into \eqref{stab-est2} and using Lemma~\ref{lem:Ah} and Lemma~\ref{lem:trace2}, we get
	\begin{align}
		\mathscr{K}_0^n((\ul{\Bu_h^n},\ul{p_h^n}),(\ul{\Bv_h},\ul{0}))\ge\,&
		(1.3-2\varepsilon-C\gamma_1\nu_2)\tN{\ul{\Bu_h^n}}^2_{\boldsymbol{\Cv}}
		-(0.5+2\varepsilon)\tN{\ul{\Bu_h^{n-1}}}^2_{\boldsymbol{\Cv}}\notag \\
		&-C\Big\{\sum_{j=0}^4\dN{\ul{\Bu_h^{n-j}}}_{0,\Omega_\eta^n}^2
		+\varepsilon\tN{\ul{\Bu_h^{n-j}}}_{\boldsymbol{\Cv}}^2
		+h^2\N{\ul{\Bf^n}}_{0,\Omega_\eta^n}^2\Big\} \notag \\
		&+2\mathscr{J}_p^n(\ul{p_h^n},\ul{p^n_h})
		+\gamma_1\nu_2h^2 \dN{\nu^{-\frac 12}\nabla \rho_h^n}_{0,\Omega_\eta^n}^2.\label{eq:K00}
	\end{align}
	By norm equivalence, we have
	\begin{equation}
		\dN{\nu^{\frac12}\nabla \ul{p_h^n}}_{0,\Omega_\eta^n}^2
		\geq \tN{\ul{p_h^n}}_{1,\Cq}^2 - \tN{\ul{p_h^n}}_{\Cq}^2 -Ch^{-2}\mathscr{J}_p^n(\ul{p_h^n},\ul{p_h^n}).
	\end{equation}
	Combining the above inequality with Lemma~\ref{lem:phL2} leads to
	\begin{align}
		\mathscr{K}_0^n((\ul{\Bu_h^n},\ul{p_h^n}),(\ul{\Bv_h},\ul{0}))&\ge\,
		(1.3-2\varepsilon-C\gamma_1\nu_2)\tN{\ul{\Bu_h^n}}^2_{\boldsymbol{\Cv}}
		-(0.5+2\varepsilon)\tN{\ul{\Bu_h^{n-1}}}^2_{\boldsymbol{\Cv}}\notag \\
		& -C\Big\{\sum_{j=0}^4\dN{\ul{\Bu_h^{n-j}}}_{0,\Omega_\eta^n}^2
		+\varepsilon\tN{\ul{\Bu_h^{n-j}}}_{\boldsymbol{\Cv}}^2
		+h^2\N{\ul{\Bf^n}}_{0,\Omega_\eta^n}^2\Big\}\notag \\
		& +\gamma_1\nu_2h^2 \tN{\ul{\rho_h^n}}_{1,\Cq}^2
		+(2-C\gamma_1\nu_2)\mathscr{J}_p^n(\ul{p_h^n},\ul{p^n_h})
		.    \label{est:K0}
	\end{align}
	Now substitute \eqref{est:A1}--\eqref{stab-est1} and \eqref{est:K0} into \eqref{eq:stab0} and take the sum of both sides over $n=4,\cdots,m$. After proper arrangements, we arrive at
	\begin{align*}
		&\sum_{l=1}^4 \dN{\Psibf_l^4(\underline{\bbU_h^{m}})}_{0,\Omega^m_{\eta}}^2
		+(0.8-C\varepsilon -C\gamma_1\nu_2 -Ch^2)
		\sum_{n=4}^m\tau\tN{\ul{\Bu_h^{n}}}_{\boldsymbol{\Cv}}^2 \\
		&\quad+(1-\gamma_1\nu_2 h^2 -\varepsilon\gamma_1\nu_2)
		\gamma_1\nu_2 h^2\sum_{n=4}^m \tau\tN{\ul{p^n_h}}_{1,\Cq}^2\\
		&\quad +(2-\varepsilon -C\gamma_1\nu_2)\sum_{n=4}^m\tau \mathscr{J}_p^n(\ul{p_h^n},\ul{p^n_h})  \\
		\le\,&C\sum_{n=4}^m\tau\big\|\ul{\Bu_h^n}\big\|_{0,\Omega^n_\eta}^2
		+C\sum_{n=0}^m\tau \dN{\ul\Bf^n}_{0,\Omega^n_\eta}^2
		+C\sum_{n=0}^3\Big(\dN{\ul{\Bu_h^n}}^2_{0,\Omega^n_{\eta}}
		+\tau\tN{\ul{\Bu_h^{n}}}_{\boldsymbol{\Cv}}^2\Big)
	\end{align*}
	From \cite[Table~2.2]{liu13}, we have $\Psibf_1^4(\underline{\bbU_h^{m}})=0.06\ul{\Bu^m_h}$.
	In the above inequality, we let $\varepsilon$, $\gamma_1$, and $h$ to be small enough
	such that 
	\ben
	C(\varepsilon+\gamma_1\nu_2)<0.3,\quad
	\varepsilon \gamma_1\nu_2<0.5, \quad Ch^2<0.1.
	\een 
	It results in
	\begin{align*}
		&\big\|\underline{\Bu_h^m}\big\|^2_{0,\Omega^m_{\eta}}
		+\sum_{n=4}^m\tau\tN{\ul{\Bu_h^{n}}}_{\boldsymbol{\Cv}}^2
		+\gamma_1\nu_2 h^2\sum_{n=4}^m \tau\tN{\ul{p^n_h}}_{1,\Cq}^2 \\
		\le\,&C\sum_{n=4}^m\tau\big\|\ul{\Bu_h^n}\big\|_{0,\Omega^n_\eta}^2
		+C\sum_{n=0}^m\tau \N{\ul\Bf^n}_{0,\Omega^n_\eta}^2
		+C\sum_{n=0}^3\Big(\dN{\ul{\Bu_h^n}}^2_{0,\Omega^n_{\eta}}
		+\tau\tN{\ul{\Bu_h^{n}}}_{\boldsymbol{\Cv}}^2\Big).
	\end{align*}
	The proof is finished by using Gronwall's inequality.
\end{proof}

\section{Finite element error estimates}
\label{sec:err}
The purpose of this section is to prove the error estimates between the exact solution and the finite element solution.
Throughout this section we assume the the index $r$ in Assumption~\ref{ass-1} satisfies $r\ge \max(4,k+2)$.

\subsection{Extension operator}

First we write $\Omega(t) = \Omega_1(t)\cup \Omega_2(t)$ and define
\begin{align*}
	Q_{T} =\,& \left\{(\Bx,t):\Bx\in \Omega(t),\,t\in [0,T]\right\}, \\
	Q_{i,T} =\,& \left\{(\Bx,t):\Bx\in \Omega_i(t),\,t\in [0,T]\right\}, \quad i=1,2.
\end{align*}
Given an integer $0\le m\le k+1$, we define the Bochner space
\begin{align*}
	&L^\infty\big(0,T;H^m(\Omega_i(t))\big)\\
	=\,&\big\{v_i\in L^2(Q_{i,T}):
	\esssup_{t\in [0,T]}\N{v_i(\BX(t;0,\cdot),t)}_{H^m(\Omega_i(0))}
	<\infty\big\} .
\end{align*}
For any Lipschitz domain $D\subset \Omega$, by \cite[Chapter 6]{ste70},
there is an extension operator $\Ce_{D}$: $H^{k+1}(D)\to H^{k+1}(\bbR^2)$ depending only on $D$ and $k$
such that
\begin{equation*}
	\left(\Ce_{D} w\right)|_{D} =w,\qquad
	\|\Ce_{D}w\|_{H^{k+1}(\bbR^2)}\lesssim \|w\|_{H^{k+1}(D)} \qquad
	\forall\,w\in H^{k+1}(D).
\end{equation*}
Since $\BX(t;0,\cdot)$: $\Omega_i(0)\to\Omega_i(t)$ is one-to-one, we have $w\circ\BX(t;0,\cdot)\in H^{k+1}(\Omega_i(0))$ for any $w\in H^{k+1}(\Omega_i(t))$. Let the extension operator from
$H^{k+1}(\Omega_i(t))$ to $H^{k+1}(\bbR^2)$ be defined as
\ben
\mathsf{E}_t^i w :=
\big[\mathsf{E}_0^i(w\circ\BX(t;0,\cdot))\big]\circ \BX(0;t,\cdot),\quad
\hbox{where}\;\; \mathsf{E}_0^i := \Ce_{\Omega_i(0)}.
\een
The global extension operator
$\mathsf{E}^i$: $L^\infty(0,T;H^{k+1}(\Omega_i(t)))\to L^\infty(0,T;H^{k+1}(\bbR^2))$
is defined as
\begin{equation}\label{def:extend}
	\hbox{for any}\; v\in L^\infty(0,T;H^{k+1}(\Omega_i(t))),\quad
	(\mathsf{E}^iv)(\cdot,t)=\mathsf{E}_t^iv(\cdot,t)\;\forall\, t\in [0,T].
\end{equation}
Following \cite[Lemma~6.1]{ma21} and arguments similar to \cite{leh19}, we have the following lemma.
\vspace{1mm}

\begin{lemma}\label{lem:extension}
	There exists a constant $C>0$ depending only on $\Omega_i(0)$
	and the exact flow map $\BX(t;0,\Omega_0)$ such that,
	for any $v\in L^\infty\big(0,T;H^{k+1}(\Omega_i(t))\big)\cap H^{k+1}(Q_{i,T})$,
	\begin{align*}
		\|\mathsf{E}^i v\|_{H^{k+1}(\bbR^2\times (0,T))} \leq\,& C \|v\|_{H^{k+1}(Q_{i,T})}, \\
		\|\mathsf{E}^i v\|_{L^\infty(0,T;H^{m}(\bbR^2))} \leq\,&
		C \|v\|_{L^\infty(0,T;H^m(\Omega_i(t)))}, \quad  1\leq m\leq k+1.
	\end{align*}
	Furthermore, for $v\in L^{\infty}(0,T;H^{m}(\Omega_i(t)))$ and
	$\partial_t v \in  L^{\infty}\big(0,T;H^{m-1}(\Omega_i(t))\big)$,
	for any $t\in [0,T]$, it holds
	\begin{equation*}
		\|\partial_t(\mathsf{E}^i v)\|_{H^{m-1}(\bbR^2)}
		\le C\big(\|v\|_{H^{m}(\Omega_i(t))} +\|\partial_t v\|_{H^{m-1}(\Omega_i(t))}\big),
		\quad 1\leq m\leq k+1.
	\end{equation*}
\end{lemma}

\subsection{The extended solution}

Let $\ul\Bu =(\Bu_1,\Bu_2)$, $\ul p=(p_1,p_2)$ be the true solution to \eqref{eq:model} and assume
\begin{align*}
	\begin{array}{ll}
		&\Bu_i\in \BL^\infty(0,T;\BH^{k+1}(\Omega_i(t)))\cap \BH^{k+1}(Q_{i,T}), \quad  \Div \Bv_i=0,
		\vspace{1mm}\\
		&p_i \in L^\infty(0,T;H^{k}(\Omega_i(t)))\cap H^{k}(Q_{i,T}), \quad 
		(p_1/\nu_1,1)_{\Omega_1(t)} + (p_2/\nu_2,1)_{\Omega_2(t)}=0.
	\end{array}
\end{align*}
Let $\tilde{\Bu}_i = \mathsf{E}^i\Bu_i$, $\tilde{p}_i = \mathsf{E}^ip_i$
be the extensions of $\Bu_i$ and $p_i$, respectively. Define
\begin{align*}
	\tilde\Bf_i := \frac{\D\tilde \Bu_i}{\D t}-\nu_i \Delta \tilde{\Bu}_i + \nabla \tilde{p}_i,
	\qquad
	\frac{\D\tilde \Bu_i}{\D t} :=
	\frac{\partial\tilde \Bu_i}{\partial t} +(\Bw\cdot\nabla)\tilde{\Bu}_i.
\end{align*}
It is easy to see that $\tilde\Bf_i = \Bf_i$ in $\Omega_i(t)$, $i=1,2$.
From Lemma~\ref{lem:extension}, we have
\begin{align}\label{ext-stab-up}
	\begin{cases}
		\N{\tilde \Bu_i}_{\BL^\infty(0,T;\BH^{k+1}(\Omega))}
		\lesssim \N{\Bu_i}_{\BL^\infty(0,T;\BH^{k+1}(\Omega_i(t)))},
		\vspace{1mm}\\
		\N{\tilde p_i}_{\BL^\infty(0,T;H^{k}(\Omega))} \lesssim
		\N{p_i}_{\BL^\infty(0,T;H^{k}(\Omega_i(t)))},
	\end{cases}
	\quad i=1,2.
\end{align}
For convenience, we introduce some symbols for quantities at discrete time steps
\ben
\ul{\tilde\Bu^n}=\big(\tilde{\Bu}_1^n,\tilde{\Bu}_2^n\big),\;
\ul{\tilde p^n}=\big(\tilde{p}_1^n,\tilde{p}_2^n\big), \;
\tilde{\Bu}_i^n=\tilde{\Bu}_i(\cdot,t_n),\;\tilde{p}_i^n=\tilde{p}_i(\cdot,t_n),\;
\tilde{\Bf}_i^n=\tilde{\Bf}_i(\cdot,t_n).
\een
Define $\tilde\BU_i^{m,n}:= \tilde{\Bu}_i^m\circ \BX_\tau^{n,m}$ and $\tilde\bbU_i^{n} :=\big[\tilde\BU_i^{n-k,n},\cdots,\tilde\BU_i^{n,n}\big]$.
Then the extended solutions $\tilde \Bu_i^n$ and $\tilde p_i^n$ satisfy the semi-discrete equations
\begin{align}
	&\frac{1}{\tau}\Lambda^k \tilde\bbU_i^n -\nu_i\Delta \tilde \Bu_i^n + \nabla \tilde p_i^n
	= \hat{\Bf}_i^n \quad \text{in}\; \; \Omega_{\eta,i}^n, \label{mom-ext}\\
	&\Div \tilde \Bu_i^n=0 \quad \text{in}\;\; \Omega_i(t_n)\backslash\bar\Omega_{\eta,i}^n,
	\label{div-ext}
\end{align}
where
\ben
\hat{\Bf}_i^n = \tilde\Bf_i^n +\BR_i^n, \qquad
\BR_i^n = \frac{1}{\tau}\Lambda^k \tilde\bbU_i^n -\frac{\D\tilde \Bu_i}{\D t}\Big|_{t=t_n}.
\een
However, the extended solution does not satisfy the continuity conditions on the approximate interface $\Gamma^n_\eta$, namely,
\ben
\nu_1 \partial_{\Bn} \tilde \Bu_1^n -\tilde p_1^n \Bn \ne
\nu_2 \partial_{\Bn} \tilde \Bu_2^n -\tilde p_2^n \Bn, \quad
\tilde \Bu_1^n \ne\tilde \Bu_2^n \quad \hbox{on}\;\;\Gamma^n_\eta.
\een

From \eqref{mom-ext}--\eqref{div-ext}, it is easy to see that the extended solution satisfies the discrete equations with modified right-hand sides
\begin{subequations} \label{eq:exten disc}
	\begin{align}
		&\tau^{-1}\Aprod[\Omega^n_\eta]{\Lambda^k \ul{\tilde\bbU}^n,\ul{\Bv_h}}
		+\mathscr{A}_h^n(\ul{\tilde\Bu^n},\ul{\Bv_h})
		+\mathscr{B}^n_0(\ul{\Bv_h},\ul{\tilde p^n})\notag \\
		=\;& \Aprod[\Omega^n_\eta]{\ul{\hat\Bf}^n,\ul{\Bv_h}} +\Cf^n(\ul{\Bv_h}),
		\qquad \forall \ul{\Bv_h}\in \Cv_h^n \label{eq:ext-u}\\
		&\mathscr{B}^n_0(\ul{\tilde\Bu^n},\ul{q_h})
		- \mathscr{J}_p^n(\ul{\tilde p^n},\ul{q_h})
		-\mathscr{C}_h^n(\ul{\tilde\bbU^n},\ul{\tilde p^n};\ul{q_h})\notag\\
		=\;&(\Cb_0^n\ul{\tilde\Bu^n})(\ul{q_h})
		-\gamma_1\nu_2 h^2\Aprod[\Omega_\eta^n]{\ul{\hat\Bf}^n,\nu^{-1}\nabla \ul{q_h}}, \qquad 
		\forall \ul{\Bq_h}\in \Cq_h^n \label{eq:ext-p}
	\end{align}
\end{subequations}
where
$\ul{\tilde\bbU^n}=\big(\tilde\bbU_1^n, \tilde\bbU_2^n\big)$,
$\hat{\ul\Bf}^n= \big(\hat{\Bf}_1^n,\hat{\Bf}_2^n\big)$, and
\begin{align*}
	&\Cf^n(\ul{\Bv_h}) = \int_{\Gamma_\eta^n} \jump{\nu\partial_{\Bn}\ul{\tilde\Bu}^n
		-\ul{\tilde p^n} \Bn}\Avgg{\ul{\Bv_h}}
	\notag \\
	&\hspace{20mm}+\int_{\Gamma_\eta^n}\jump{\ul{\tilde\Bu^n}} \big(\Avg{\nu \partial_\Bn \ul{\Bv_h}}
	+\frac{\gamma0}{h}\Avg{\nu}\jump{\ul{\Bv_h}}\big),\\
	&\mathscr{C}_h^n\big(\ul{\tilde\bbU^n},\ul{\tilde p^n};\ul{q_h})
	= \gamma_1\nu_2 h^2\Aprod[\Omega_\eta^n]{\tau^{-1}\Lambda^k \ul{\tilde\bbU^n}
		-\nu \Delta_h\ul{\tilde\Bu^n} +\nabla \ul{\tilde p^n},\nu^{-1}\nabla \ul{q_h}}.
\end{align*}

\subsection{Modified Stokes projections of the extended solution}

Using \eqref{def-proj} of the modified Stokes projection, we define the modified Stokes projection
$(\ul{\hat{\Bu}_h^n},\ul{\hat{p}_h^n})=\Cs^n(\ul{\tilde\Bu}^n, \ul{\tilde p}^n,\Cb_0^n \ul{\tilde\Bu^n})$.
The approximation errors can be split into two parts
\begin{align}\label{unpn}
	\ul{\tilde\Bu^n}- \ul{\Bu_h^n} =\ul{\thetabf^n} + \ul{\thetabf_h^n}, \qquad
	\ul{\tilde p^n} -\ul{p_h^n} = \ul{\rho^n} +\ul{\rho_h^n}.
\end{align}
where the error functions are defined as
\ben
\ul{\thetabf^n} = \ul{\tilde\Bu^n}-\ul{\hat{\Bu}_h^n}, \qquad
\ul{\rho^n} = \ul{\tilde p^n} -\ul{\hat{p}_h^n},\qquad
\ul{\thetabf_h^n} = \ul{\hat{\Bu}_h^n} - \ul{\Bu_h^n}, \qquad
\ul{\rho_h^n} = \ul{\hat{p}_h^n} -\ul{p_h^n}.
\een
From \eqref{ieq:stokes-L2}, we know that
\begin{align}\label{ieq:theta-rho-n}
	\dN{\ul{\thetabf^n}}_{0,\Omega^n_\eta} + h \tN{\ul{\thetabf^n}}_{\boldsymbol{\Cv}}
	+ h \tN{\ul{\rho^n}}_{\Cq}
	\lesssim h^{k+1}\wt M_k^n,
\end{align}
where by \eqref{ext-stab-up},
\begin{align*}
	\wt M_k^n:=\,&
	\sum_{i=1,2}\big(\nu_i |\tilde \Bu_i^n|_{\BH^{k+1}(\Omega)}^2
	+\nu_i^{-1}|\tilde p^n_i|_{H^k(\Omega)}^2\big) \\
	\le\,& C \sum_{i=1,2}\big(\N{\Bu_i}^2_{\BL^\infty(0,T;\BH^{k+1}(\Omega_i(t)))}
	+ \N{p_i}^2_{\BL^\infty(0,T;H^{k}(\Omega_i(t)))}\big).
\end{align*}
Let $\ul{\pi_{k-1}}(\ul{\tilde p^n})$ be the quasi-interpolation of $\ul{\tilde p^n}$.
Using \eqref{ieq:q-qh}, \eqref{ieq:theta-rho-n}, and inverse estimate, we have
\begin{align}
	\sum_{i=1}^2\dN{\nu_i^{-\frac12}\nabla\rho^n_{i}}_{0,\Omega^n_{h,i}}^2
	\lesssim & \sum_{i=1}^2\dN{\nu_{i}^{-\frac12}\nabla\big[\tilde p_i^n
		-\pi_{k-1}(\tilde p_i^n)\big]}_{0,\Omega^n_{h,i}}^2\notag\\
	&+\sum_{i=1}^2\dN{\nu_i^{-\frac12}\nabla\big[\pi_{k-1}(\tilde p_i^n)
		-\hat p_{h,i}^n\big]}_{0,\Omega^n_{h,i}}^2 \notag \\
	\lesssim\,& h^{2k-2}\wt M_k^n
	+h^{-2}\tN{\ul{\pi_{k-1}}(\ul{\tilde p^n}) -\ul{\hat p^n_h}}_{\Cq}^2 \notag \\
	\lesssim\,& h^{2k-2}\wt M_k^n
	+h^{-2}\tN{\ul{\rho^n}}_{\Cq}^2
	+h^{-2}\tN{\ul{\pi_{k-1}}(\ul{\tilde p^n}) -\ul{\tilde p^n}}_{\Cq}^2 \notag \\
	\lesssim\,&h^{2k-2}\wt M_k^n.
	\label{ieq:rho-grad}
\end{align}
It is left to estimate $\ul{\thetabf^n_h}$ and $\ul{\rho^n_h}$.

\subsection{Error estimates for $\ul{\thetabf^n_h}$ and $\ul{\rho^n_h}$}

For convenience in notation, we define
\ben
\begin{array}{ll}
	\ul{\Thetabf_h^{m,n}}=\ul{\thetabf_h^m}\circ\BX_\tau^{n,m},
	&\ul{\Thetabf^{m,n}}=\ul{\thetabf^n}\circ \BX_\tau^{n,m}, \vspace{1mm} \\
	\uuline{\Thetabf_h^n}:= \big[\ul{\Thetabf_h^{n-k,n}},\cdots,\ul{\Thetabf_h^{n,n}}\big],
	\quad
	&\uuline{\Thetabf^n}:=\big[\ul {\Thetabf^{n-k,n}},\cdots,\ul{ \Thetabf^{n,n}}\big].
\end{array}
\een
Subtracting \eqref{eq:disc} from \eqref{eq:exten disc} and using \eqref{def-proj},
we obtain the discrete formulation \vspace{1mm}
\begin{center}
	\fbox{\parbox{0.975\textwidth}{
			\begin{subequations}\label{eq:err0}
				\begin{align}
					&\frac{1}{\tau}\Aprod[\Omega_\eta^n]{\Lambda^k \uuline{\Thetabf_h^n},\ul{\Bv_h}}
					+ \mathscr{A}_h^n(\ul{\thetabf_h^n},\ul{\Bv_h})+ \mathscr{B}_{0}^n(\ul{\Bv_h},\ul{\rho_h^n})
					=\Ce^n_1(\ul{\Bv_h}),\; \forall\,\ul{\Bv_h}\in {\boldsymbol{\Cv}}^n_h\label{eq:err u}\\
					&\mathscr{B}_{0}^n(\ul{\thetabf_h^n},\ul{q_h})-\mathscr{J}_2^n(\ul{\rho_h^n},\ul{q_h})
					- \mathscr{C}_h^n(\uuline{\Thetabf_h^n},\ul{\rho_h^n};\ul{q_h})
					= \Ce_2^n(\ul{q_h}), \;\forall\,\ul{q_h}\in\Cq^n_h, \label{eq:err p}
				\end{align}
			\end{subequations}
			where
			\begin{align}
				&\Ce^n_1(\ul{\Bv_h})=-\frac{1}{\tau}\Aprod[\Omega_\eta^n]{\Lambda^k \uuline{\Theta}^n,\ul{\Bv_h}}
				+\Aprod[\Omega_\eta^n]{\ul{\hat{\Bf}^n}-\ul{\Bf}^n,\ul{\Bv_h}}
				+\Cf^n(\ul{\Bv_h}), \label{eq:Ce1}\\
				&\Ce_2^n(\ul{q_h}) = (\Cb_0^n\ul{\tilde\Bu^n})(\ul{q_h})-\gamma_1\nu_2
				h^2\Aprod[\Omega_\eta^n]{\ul{\hat{\Bf}^n}-\ul{\Bf}^n,\nu^{-1}\nabla \ul{q_h}}
				+\mathscr{C}_h^n(\uuline{\Thetabf^n},\ul{\rho^n};\ul{q_h}).\label{eq:Ce2}
	\end{align}}}
\end{center}
\vspace{2mm}

Clearly $\Ce_1^n$ provides a functional on $\boldsymbol{\Cv}^n_h$ and $\Ce_2^n$ provides a functional on $\Cq^n_h$.
Their norms are defined as
\begin{align}\label{norm-E}
	\N{\Ce_1^n}_{1,(\Cv_h^n)'} = \sup_{\ul{\Bv_h}\in\Cv_h^n}
	\frac{\Ce_1^n(\ul{\Bv_h})}{\big\|\nu^{\frac12} \ul{\Bv_h}\big\|_{0,\Omega_\eta^n}
		+\tN{\ul{\Bv_h}}_{\boldsymbol{\Cv}}},\quad
	\N{\Ce_2^n}_{1,(\Cq_h^n)'}=\sup_{\ul{q_h}\in\Cq_h^n}
	\frac{\Ce_2^n(\ul{q_h})}{\tN{\ul{q_h}}_{1,\Cq}}.
\end{align}

\begin{lemma}\label{lem:E1}
	Assume $k\ge 2$, $\eta = O(\tau^{\max(k/3,1)})$, $\gamma_0 h \le 1$, and $h=O(\tau)$. Then
	\begin{align}
		\N{\Ce_1^n}_{1,(\Cv_h^n)'} \lesssim \;&\tau^k\sum_{i=1,2}
		\nu_i^{-\frac12}\Big(\tau\dN{\tilde \Bf_i^n}_{\BH^1(\Omega)}
		+\sum_{j=0}^k\dN{\tilde \Bu_i^{n-j}}_{\BH^{k+1}(\Omega)}\Big) \notag\\
		&+\tau^{k-\frac12}\sum_{i=1,2} \nu_i^{-\frac12}
		\N{\tilde\Bu_i}_{\BH^{k+1}(\Omega\times(t_{n-k},t_n))}
		+\tau^k (\wt{M}_k^n)^{\frac12}.
		\label{eq:norm Ce1}
	\end{align}
\end{lemma}
\begin{proof}
	Changing variables of integrations and using \eqref{ieq:Jtau} and \eqref{ieq:stokes-L2} yield
	\begin{align}
		\frac{1}{\tau}\dN{\Lambda^k\uuline{\Thetabf^n}}_{0,\Omega_{\eta}^n}
		\lesssim \frac{1}{\tau}\sum_{j=0}^k\dN{(\Det \bbJ_\tau^{n-j,n})^{\frac12}
			\ul{\thetabf^{n-j}}}_{0,\Omega_\eta^{n-j}}
		\lesssim
		\tau^{k} 
		\sum_{j=0}^k\sum_{i=1}^2\dN{\tilde{\Bu}_i^{n-j}}_{k+1,\Omega}.\label{eq:theta pullback}
	\end{align}
	From \eqref{ieq:Jtau} and the error estimates in \eqref{ieq:Xmn}, we have
	\begin{align*}	
		\N{\BX^{n,n-j}_\tau-\BX^{m,n}}_{\BL^{\infty}(\Omega)} \lesssim \tau^{k+2}.
	\end{align*}
	Applying Taylor's formula to the right-hand side of the following equality
	\ben
	\BR_i^n = \frac{1}{\tau}\Lambda^k \tilde \bbU_i^n -\frac{\D\tilde\Bu_i}{\D t}
	+\frac{1}{\tau}\sum\limits_{j=0}^k \lambda_j^k
	\big(\tilde \Bu_i^{n-j}\circ\BX_\tau^{n,n-j} -\tilde \Bu_i^{n-j}\circ\BX^{n,n-j}\big),
	\een
	we have
	\begin{align}
		\NLtwov[\Omega_{\eta,i}^n]{\BR_i^n} \lesssim
		&\bigg\|\sum_{l=1}^k \int_{t_{n-l}}^{t_n} \frac{(t_n-\xi)^k}{\tau}
		\frac{\D^{k+1}\tilde\Bu_i}{\D t^{k+1}}\D\xi\bigg\|_{0,\Omega_{\eta,i}}
		+\tau^{k+1}\sum_{j=0}^k \SN{\tilde \Bu_i^{n-j}}_{\BH^1(\Omega)}\notag \\
		\lesssim\;& \tau^{k-\frac12}\dN{\tilde\Bu_i}_{\BH^{k+1}(\Omega\times(t_{n-k},t_n))}
		+\tau^{k+1} \sum_{j=0}^k\dN{\tilde \Bu_i^{n-j}}_{\BH^1(\Omega)}.
		\label{Ri}
	\end{align}
	Write $\ul{\tilde \Bf^n}=(\tilde \Bf_1^n,\tilde \Bf_2^n)$, 
	and note that $\tilde\Bf_i^n = \Bf_i^n$ in $\Omega_i(t_n)$.
	Using Lemma~\ref{lem:uX} and norm equivalence, we deduce that
	\begin{align*}
		\Aprod[\Omega_\eta^n]{\ul{\tilde \Bf^n}-\ul{\Bf^n},\ul{\Bv_h}}
		=&\,\sum_{i=1,2}(\tilde\Bf_i^n - \Bf_{3-i}^n,\Bv_{h,i})_{\Omega_{\eta,i}^n\backslash\Omega_i(t_n)}\\
		\lesssim &\, \tau^{k+1}\sum_{i=1,2}
		\N{\nu_i^{-\frac 12}\tilde \Bf_i^n}_{\BH^1(\Omega)}
		\TN{\ul{\Bv_h}}_{\Cv}.
	\end{align*}
	Using $\hat{\Bf}_i^n = \tilde\Bf_i^n +\BR_i^n$ and the triangular inequality, we obtain
	\begin{align}
		&\Aprod[\Omega_\eta^n]{\ul{\hat{\Bf}^n}-\ul{\Bf^n},\ul{\Bv_h}}
		\lesssim\,\bigg\{\tau^{k+1}\sum_{i=1,2}\nu_i^{-\frac 12}\Big(\dN{\tilde \Bf_i^n}_{\BH^1(\Omega)}
		+\sum_{j=0}^k\dN{\tilde \Bu_i^{n-j}}_{\BH^1(\Omega)}\Big)\notag \\
		&\; +\tau^{k-\frac{1}{2}}\sum_{i=1,2}\nu_i^{-\frac12}
		\N{\tilde\Bu_i}_{\BH^{k+1}(\Omega\times(t_{n-k},t_n))}\bigg\}
		\Big(\dN{\nu^{\frac 12}\ul{\Bv_h}}_{0,\Omega_\eta^n}
		+\tN{\ul{\Bv_h}}_{\boldsymbol{\Cv}}\Big). \label{eq:hat f-f}
	\end{align}
	Let $\Bn = (-\chi_{n,2}',\chi_{n,1}')/\SN{\chibf_n'}$ and $\hat{\Bn} = (-\hat\chi_{n,2}',\hat\chi_{n,1}')/\SN{\hat\chibf_n'}$ denote the unit normal vectors to the approximate interface $\Gamma^n_\eta$ and the exact interface $\Gamma(t_n)$, respectively.
	By Theorem~\ref{thm:chi}, we have
	\ben
	\N{\Bn\circ\chibf_{n}'-\hat{\Bn}\circ\hat{\chibf}_n'
		\SN{\hat\chibf_n'}/\SN{\chibf_n'}}_{\BL^\infty(0,L)} \lesssim \tau^k.
	\een
	Since $\jump{(\nu \nabla \ul{\Bu^n}-\ul{p^n}\bbI)}\cdot \hat{\Bn}=\textbf{0}$ on $\Gamma(t_n)$
	and $\N{\hat{\chibf}_{n}-\chibf_{n}}_{\BL^\infty([0,L])}\lesssim \tau^{k+1}$,
	by \eqref{chi-err} and the trace theorem, we obtain
	\begin{align}
		&\dN{\jump{\nu \partial_{\Bn}\ul{\tilde\Bu^n} - \ul{\tilde p^n} \Bn}}^2_{\Ltwo[\Gamma_\eta^n]}
		=\int_0^L \Big|\jump{(\nu \nabla \ul{\tilde \Bu^n} - \ul{\tilde p^n} \bbI)
			\circ \chibf_{n}}\cdot(\Bn\circ \chibf_n)\Big|^2\SN{\chibf_{n}'}\notag \\
		\lesssim& \int_0^L 
		\Big|\jump{(\nu \nabla \ul{\tilde \Bu^n}
			-\ul{\tilde p^n} \bbI)\circ \chibf_{n}}\cdot(\Bn\circ \chibf_{n})\notag\\
		&\hspace{5mm}	-\jump{(\nu \nabla \ul{\tilde \Bu^n} -\ul{\tilde p^n} \bbI)\circ \chibf_{n}}\cdot
		(\hat{\Bn}\circ \hat{\chibf}_{n})\SN{\hat\chibf_n'}/\SN{\chibf_n'}
		\Big|^2\SN{\chibf_{n}'}\notag\\
		&+\int_0^L\Big|\jump{(\nu \nabla \ul{\tilde \Bu^n} - \ul{\tilde p^n}\bbI)\circ\chibf_{n}}\cdot
		(\hat{\Bn}\circ \hat{\chibf}_n)\notag \\
		&\hspace{5mm}-\jump{(\nu \nabla \ul{\tilde \Bu^n} -\ul{\tilde p^n} \bbI)\circ \hat{\chibf}_{n}}\cdot (\hat{\Bn}\circ\hat{\chibf}_{n})\Big|^2\SN{\hat\chibf_{n}'}^2/\SN{\chibf_n'} \notag \\
		\lesssim \;&\tau^{2k} \sum_{i=1}^2 (\nu_i\N{\tilde{\Bu}_i^n}_{\BH^3(\Omega)}^2
		+\N{\tilde{p}_i^n}_{\BH^2(\Omega)}^2).\label{eq:jump Gn}
	\end{align}
	Using the trace theorem again, we have
	\begin{align*}
		&\NLtwo[\Gamma_\eta^n]{\Avgg{\ul{\Bv_h}}}^2 \lesssim
		\NHonev[\Omega_{\eta,1}^n]{\kappa_2 \Bv_{h,1}}^2
		+\NHonev[\Omega_{\eta,2}^n]{\kappa_1 \Bv_{h,2}}^2
		\lesssim \sum_{i=1}^2\nu_i
		\NHonev[\Omega_{\eta,i}^n]{\Bv_{h,i}}^2.
	\end{align*}
	Similarly, since $\jump{\ul \Bu^n\circ \hat{\chibf}_n}=0$ on $\Gamma(t_n)$, we deduce from Theorem~\ref{thm:chi} that
	\begin{equation}\label{eq: GammaEta u}
		\NLtwo[\Gamma_\eta^n]{\jump{\ul{\tilde \Bu^n}}}^2
		\lesssim \int_0^L |\jump{\ul{\tilde \Bu^n} \circ \chibf_n}
		-\jump{\ul{\tilde \Bu^n} \circ \hat{\chibf}_n}|^2\SN{\chibf_n'}
		\lesssim \tau^{2k+2} \sum_{i=1}^2 \N{\tilde \Bu_i^n}_{\BH^2(\Omega)}^2.
	\end{equation}
	Combining \eqref{eq:jump Gn} and \eqref{eq: GammaEta u} and using the assumption that $\gamma_0 h\le 1$, we get
	\begin{equation}\label{eq:F Gamma}
		\SN{\Cf^n(\ul{\Bv_h})} \lesssim \tau ^k  (\wt M_k^n)^{\frac 12}
		\big(\big\|\nu^{\frac12}\ul{\Bv_h}\big\|_{0,\Omega^n_\eta}
		+ \tN{\ul{\Bv_h}}_{\boldsymbol{\Cv}}\big) .
	\end{equation}
	The proof is finished by inserting \eqref{eq:theta pullback}, \eqref{eq:hat f-f}, and \eqref{eq:F Gamma} into \eqref{eq:Ce1} and using the Cauchy-Schwarz inequality.
\end{proof}

\begin{lemma}\label{lem:E2}
	Let the assumptions in Lemma~\ref{lem:E1} be satisfied and assume $\tau\le\nu_2$. Then
	\begin{align}
		\N{\Ce_2^n}_{1,(\Cq_h^n)'}
		\lesssim \;&  \gamma_1\nu_2 \tau^{k+2}\sum_{i=1,2}
		\nu^{-\frac 12}_i\Big(\sum_{j=0}^k\dN{\tilde\Bu_i^{n-j}}_{\BH^{k+1}(\Omega))}
		+\dN{\tilde\Bf_i^n}_{\BH^1(\Omega)}\notag \\
		&+\tau^{-\frac 12}\dN{\tilde \Bu_i}_{\BH^{k+1}(\Omega\times(t_{n-k},t_n))} \Big)
		+\gamma_1\nu_2 \tau^{k+1}(\wt M_k^n)^{\frac 12}\notag \\
		&+\tau^{k+1}\sum_{i=1,2}\nu^{\frac 12}_i \N{\tilde \Bu_i^n}_{\BH^2(\Omega)}.
		\label{eq:norm Ce2}
	\end{align}
\end{lemma}
\begin{proof}
	Since $\tilde \Bu_i^n$ may not be divergence-free in $\Omega_{\eta,i}^n\backslash \Omega_i(t_n)$, by \eqref{eq:v-X3}, \eqref{eq: GammaEta u}, and the trace theorem, we have
	\begin{align}\label{eq:Bou}
		(\Cb_0^n\ul\Bu^n)(q_h) =\,& \sum_{i=1,2} (\Div \tilde \Bu^n_i,q_{h,i})_{\Omega_{\eta,i}^n\backslash\Omega_i(t_n)}
		+\int_{\Gamma_\eta^n} \jump{\ul{\tilde\Bu^n}}\cdot \Bn\Avg{\ul{q_h}} \notag \\
		\lesssim\,&  \tau^{k+1}\sum_{i=1,2}
		\dN{\tilde \Bu_i^n}_{\BH^2(\Omega)}\N{\ul{q_h}}_{1,\Omega_\eta^n}.
	\end{align}
	Using Lemma~\ref{lem:uX}, \eqref{Ri}, and arguments similar to \eqref{eq:hat f-f}, we obtain
	\begin{align}
		&\Aprod[\Omega_\eta^n]{\ul{\hat{\Bf}^n} - \ul{\Bf^n},\nu^{-1}\nabla\ul{q_h}}
		\lesssim \sum_{i=1}^2
		\nu_i^{-1}\big( \tau^{k+\frac 12}\dN{\tilde\Bf_i^n}_{1,\Omega}
		+\N{\BR^n_i}_{0,\Omega_{\eta,i}^n}\big)
		\dN{\nabla q_{h,i}}_{0,\Omega_{h,i}^n}\notag \\
		\lesssim\,& \tau^k \sum_{i=1}^2\nu_i^{-1}\Big(\tau^{\frac 12}
		\dN{\tilde\Bf_i^n}_{\BH^1(\Omega)}
		+\tau^{-\frac 12}\N{\tilde\Bu_i}_{\BH^{k+1}(\Omega\times(t_{n-k},t_n))}\Big)
		\NLtwov[\Omega^n_{h,i}]{\nabla q_{h,i}} \notag\\
		&+\tau^{k+1} \sum_{j=0}^k\sum_{i=1}^2\nu_i^{-1}
		\dN{\tilde\Bu_i^{n-j}}_{\BH^1(\Omega)}
		\NLtwov[\Omega^n_{h,i}]{\nabla q_{h,i}}. \label{eq:hat f-f qh}
	\end{align}

	Then using Theorems~\ref{thm:stokes-H1} and \ref{thm:stokes-L2}, we arrive at
	\begin{align}
		&\mathscr{C}_h^n(\uuline{\Thetabf^n},\ul{\rho}^n;\ul{q_h^n})\notag \\
		\lesssim&\, \gamma_1\nu_2 \tau^{k+1}\Big(
		\sum_{j=0}^k \sum_{i=1}^2\nu^{-\frac 12}_i \tau\dN{\tilde\Bu_i^{n-j}}_{k+1,\Omega}
		+(\wt M_k^n)^{\frac 12} \Big)
		\dN{\nu^{-\frac 12}\nabla \ul{q_h}}_{0,\Omega_\eta^n}. \label{eq:Ch theta}
	\end{align}
	The proof is finished by inserting \eqref{eq:Bou}--\eqref{eq:Ch theta} into \eqref{eq:Ce2}.
\end{proof}


\begin{lemma}\label{lem:stab of rho}
	Assume that the penalty parameter $\gamma_0$ in $\mathscr{J}_0^n$ is large enough.
	Then
	\begin{align*}
		\tN{\ul{\rho_h^n}}_{\Cq}^2\lesssim\,& \mathscr{J}_p^n(\ul{\rho_h^n},\ul{\rho_h^n})
		+ (\nu_2 \tau^2)^{-1}\dN{\Lambda^k \uuline{\Thetabf_h^n}}_{0,\Omega_\eta^n}^2
		+\tN{\ul{\thetabf_h^n}}_{\boldsymbol{\Cv}}^2+\N{\Ce_1^{n}}_{1,(\Cv_h^n)'}^2.
	\end{align*}
\end{lemma}
\begin{proof}
	In view of \eqref{eq:err u},
	the proof is parallel to to that of Lemma~\ref{lem:phL2}. We omit the details.
\end{proof}
\vspace{1mm}

\begin{lemma}\label{lem:stokes-L2-theta}
	Let Assumption~\ref{ass-2}, Assumption~\ref{ass-3}
	and the assumptions in Lemma~\ref{lem:E1} be satisfied.
	Then the modified Stokes projection
	$\big(\ul{\hat{\Thetabf}_h^{n-1,n}},\ul{\hat{\rho}_h^{n-1,n}}\big)
	=\Cs^n\big(\ul{\Thetabf_h^{n-1,n}},\ul{0},0\big)$ satisfies
	\begin{align}
		&\tN{\big(\ul{\hat{\Thetabf}_h^{n-1,n}},\ul{\hat{\rho}_h^{n-1,n}}\big)}_{\boldsymbol{\Cv},\Cq}
		\lesssim\,	\tN{\ul{\thetabf_h^{n-1}}}_{\boldsymbol{\Cv}},\label{eq:err-proj0}\\
		&\dN{\ul{\Thetabf_h^{n-1,n}} - \ul{\hat{\Thetabf}_h^{n-1,n}}}_{0,\Omega_\eta^n}
		\lesssim \, \gamma_1\nu_2 h\Big\{\dN{\nu^{-\frac12}\Lambda^k\uuline{\Thetabf_h^{n-1}}}_{0,\Omega^{n-1}_\eta}\notag \\
		&\hspace{40mm}+h\big\|\nu^{-\frac12}\nabla\ul{\rho_h^{n-1}}\big\|_{0,\Omega^n_\eta}\Big\}
		+\N{\Ce_2^{n-1}}_{1,(\Cq_h^{n-1})'}	\notag\\
		&\hspace{40mm}+h\Big\{\tN{\ul{\thetabf_h^{n-1}}}_{\boldsymbol{\Cv}}
		+\mathscr{J}_p^{n-1}(\ul{\rho_h^{n-1}},\ul{\rho^{n-1}_h})^{\frac12}\Big\}.\label{eq:err-proj1}
	\end{align}
\end{lemma}
\begin{proof}
	The stability of $\ul{\hat{\Thetabf}_h^{n-1,n}}$ and $\ul{\hat{\rho}_h^{n-1,n}}$ comes directly from Theorem~\ref{thm:stokes-H1} and Remark~\ref{rem:Un}.\vspace{1mm}
	Following the proof in Corollary~\ref{cor:stokes-L2-1}, we have
	\begin{equation*}
		\dN{\ul{\Thetabf_h^{n-1,n}}-\ul{\hat{\Thetabf}_h^{n-1,n}}}_{0,\Omega_\eta^n}
		\lesssim h\tN{\thetabf_h^{n-1}}_{\boldsymbol{\Cv}}
		+\sup_{\ul{q_h}\in \Cq_h^n} \frac{\mathscr{B}_0^{n-1}(\ul{\thetabf_h^{n-1}},\ul{\hat{q}_h})}{\tN{\ul{q_h}}_{1,\Cq}},
	\end{equation*}
	where we have extended $q_{h,1},q_{h,2}$ to $\Omega$ by the convention in \eqref{ieq:vh-ext}, and
	$\hat q_{h,i} =\pi_{1,\Omega}\big(q_{h,i}\circ \BX_\tau^{n-1,n}\big)\in V_h(1,\Ct_h)$ is the Scott-Zhang interpolation.

	From \eqref{eq:err p}, we have
	\begin{equation*}
		\mathscr{B}_0^n(\ul{\thetabf_h^{n-1}},\ul{\hat{q}_h}) = \Ce_2^{n-1}(\ul{\hat{q}_h}) +
		\mathscr{J}_p^{n-1}(\ul{\rho_h^{n-1}},\ul{\hat{q}_h})
		+\mathscr{C}_h^{n-1}(\uuline{\Thetabf_h^{n}},\ul{\rho_h^{n-1}},\ul{\hat{q}_h}).
	\end{equation*}
	By arguments similar to the proof of Corollary~\ref{cor:stokes-L2-1}, we have
	\begin{align*}
		\big|\mathscr{J}_p^{n-1}(\ul{\rho_h^{n-1}},\ul{\hat{q}_h})\big|
		\lesssim\,& h \mathscr{J}_p^{n-1}(\ul{\rho_h^{n-1}},\ul{\rho_h^{n-1}})^{\frac12}
		\tN{\ul{q_h}}_{1,\Cq}, \\
		\big|\mathscr{C}_h^{n-1}(\uuline{\Thetabf_h^{n}},\ul{\rho_h^{n-1}},\ul{\hat{q}_h})\big|
		\lesssim\,& \gamma_1\nu_2 h \Big(\tN{\ul{\thetabf_h^{n-1}}}_{\boldsymbol{\Cv}}
		+\dN{\nu^{-\frac 12}\Lambda^k \uuline{\Thetabf_h^{n-1}}}_{0,\Omega_\eta^{n-1}}\\
		&+h\dN{\nu^{-\frac 12}\nabla \ul{\rho_h^{n-1}}}_{0,\Omega_\eta^n}\Big)
		\tN{\ul{q_h}}_{1,\Cq}.
	\end{align*}
	Together with \eqref{eq:norm Ce2} and \eqref{int-errK}, they yields the desired result
	\begin{align*}
		\sup_{\ul{q_h}\in \Cq_h^n} \frac{\mathscr{B}_0^{n-1}(\ul{\thetabf_h^{n-1}},\ul{\hat{q}_h})}{\tN{\ul {q_h}}_{1,\Cq}}
		\lesssim\,& \gamma_1\nu_2h\Big(\tN{\ul{\thetabf_h^{n-1}}}_{\boldsymbol{\Cv}}
		+\dN{\nu^{-\frac 12}\Lambda^k \uuline{\Thetabf_h^{n-1}}}_{0,\Omega_\eta^{n-1}}\\
		&+h\dN{\nu^{-\frac 12}\nabla \ul{\rho_h^{n-1}}}_{0,\Omega_\eta^n}\Big) 
		+h \mathscr{J}_p^{n-1}(\ul{\rho_h^{n-1}},\ul{\rho_h^{n-1}})^{\frac12}\\
		&+\N{\Ce_2^{n-1}}_{1,(\Cq_h^{n-1})'}.
	\end{align*}
	The proof is finished.
\end{proof}

\begin{theorem}\label{thm:rhoh-err}
	Let the assumptions in Theorem~\ref{thm:uh-stab} and Lemma~\ref{lem:E1} be satisfied.
	Assume $\max(\tau,h)\le \gamma_1\nu_2$ and that the penalty parameter $\gamma_0$ in $\mathscr{J}^n_0$ satisfies
	$\gamma_0^{-1}+\gamma_0h \ll 1$.
	there holds for any $k\leq m\leq N$,
	\begin{align*}
		&\dN{\ul{\thetabf_h^m}}_{0,\Omega_\eta^m}^2
		+\sum_{n=k}^m\tau\Big(\tN{\ul{\thetabf_h^n}}^2_{\boldsymbol{\Cv}}
		+\gamma_1\nu_2 h^2\tN{\ul{\rho_h^n}}_{1,\Cq}^2
		+\mathscr{J}_p^n(\ul{\rho_h^n},\ul{\rho_h^n})\Big) \notag\\
		\lesssim& \sum_{n=0}^{k-1}\tau \Big(\tN{\ul{\thetabf_h^n}}^2_{\boldsymbol{\Cv}}
		+\tN{\ul{\rho_h^n}}^2_{\Cq}\Big)
		+ \sum_{n=k}^m\Big(\tau\N{\Ce_1^n}_{1,(\Cv_h^n)'}^2
		+ (\tau \gamma_1\nu_2)^{-1} \dN{\Ce_2^n}_{1,(\Cq_h^n)'}^2\Big).
	\end{align*}
\end{theorem}
\begin{proof}
	We only prove the theorem for $k=4$. The proofs for	other cases are similar.
	Let
	\ben
	\big(\ul{\hat{\Thetabf}_h^{n-1,n}},\ul{\hat{\rho}_h^{n-1,n}}\big)
	=\Cs^n\big(\ul{\Thetabf_h^{n-1,n}},\ul{0},0\big),
	\qquad \ul{\Bv_h} =2\ul{\thetabf_h^{n}}-\ul{\hat{\Thetabf}_h^{n-1,n}},\qquad
	\ul{q_h} = \ul{0}.
	\een
	From the telescope formula in Lemma~\ref{lem:tel}, we have
	\begin{align}\label{eq:err0-theta}
		\sum_{l=1}^5 \dN{\Psibf_l^4(\uuline{\Thetabf_h^n})}_{0,\Omega^n_{\eta}}^2 -
		\sum_{l=1}^4 \dN{\Phibf_l^4(\uuline{\Thetabf_h^n})}_{0,\Omega^n_{\eta}}^2
		+ \tau \mathscr{K}_0((\ul{\thetabf_h^n},\ul{\rho_h^n});(\ul{\Bv_h},\ul{0}))
		= \tau (B_1 +B_2),
	\end{align}
	where $B_1=\Ce_1^n(\ul{\Bv_h})$, $B_2^n=\tau^{-1}\Aprod[\Omega_\eta^n]
	{\Lambda^4 \uuline{\Thetabf_h^n},\ul{\hat{\Thetabf}_h^{n-1,n}}-\ul{\Thetabf_h^{n-1,n}}}$, and
	\begin{align*}
		\mathscr{K}_0((\ul{\thetabf_h^n},\ul{\rho_h^n});(\ul{\Bv_h},\ul{0}))
		=\,& 2\mathscr{A}_h^n(\ul{\thetabf_h^n},\ul{\thetabf_h^n})-a_h^n(\ul{\Thetabf_h^{n-1,n}},\ul{\thetabf_h^n})
		+2\mathscr{J}_p^n(\ul{\rho_h^n},\ul{\rho_h^n})\notag\\
		&+\mathscr{C}_h^n(\uuline{\Thetabf_h^n},\ul{\rho_h^n};2\ul{\rho_h^n}+\ul{\hat{\rho}_h^{n-1,n}})
		+\Ce_2^n(2\ul{\rho_h^n}+\ul{\hat{\rho}_h^{n-1,n}}).
	\end{align*}

	Using arguments similar to \eqref{est:Rnh1}-\eqref{eq:K00},
	for some $\varepsilon\in (0,1)$ to be specified later, we have
	\begin{align}
		\mathscr{K}_0((\ul{\thetabf_h^n},\ul{\rho_h^n});(\ul{\Bv_h},\ul{0})) \geq\,&
		(1.3-2\varepsilon-C\gamma_1\nu_2)\tN{\ul{\thetabf_h^n}}_{\boldsymbol{\Cv}}^2
		-(0.5+2\varepsilon)\tN{\ul{\thetabf_h^{n-1}}}_{\boldsymbol{\Cv}}^2\notag\\
		&	+2\mathscr{J}_p^n(\ul{\rho_h^n},\ul{\rho_h^n})
		+ \gamma_1\nu_2 h^2\dN{\nu^{-\frac 12}\nabla \ul{\rho_h^n}}^2_{0,\Omega_\eta^n}\notag \\
		&-C\dN{\Lambda^4\uuline{\Thetabf_h^n}}_{0,\Omega_\eta^n}^2
		+\Ce_2^n(2\ul{\rho_h^n}+\ul{\hat{\rho}_h^{n-1,n}}),\label{eq:err K0}	
	\end{align}
	where we have used the assumption that $h$ and $\gamma_0^{-1}$ are small enough such that
	\ben
	C(\gamma_0^{-1}+h\gamma_0)\leq \varepsilon^2\ll 1.
	\een
	From Lemma~\ref{lem:tel} and arguments similar to \eqref{stab-est1}, we have
	\begin{align}
		\dN{\Phibf_l^4 (\uuline{\Thetabf_h^n})}_{0,\Omega_{\eta}^n}^2
		\le\,& \dN{\Psibf_l^4 (\uuline{\Thetabf_h^{n-1}})}_{0,\Omega_{\eta}^{n-1}}^2\notag\\
		&+ C\tau \sum_{j=1}^l\Big\{\big\|\ul{\thetabf_h^{n-j}}\big\|_{0,\Omega^{n-j}_\eta}^2
		+h^4\nu_2^{-1}\tN{\ul{\thetabf_h^{n-j}}}^2_{\boldsymbol{\Cv}}\Big\}. 	\label{eq:err Thetabf back}
	\end{align}
	Inserting  \eqref{eq:err K0} and \eqref{eq:err Thetabf back} into \eqref{eq:err0-theta} and using Cauchy-Schwarz inequality, we obtain
	\begin{align}
		&\sum_{l=1}^4 \dN{\Psibf_l(\uuline{\Thetabf_h^{n}})}_{0,\Omega_\eta^n}^2
		+(1.3-2\varepsilon -C\gamma_1\nu_2) \tau \tN{\ul{\thetabf_h^n}}^2_{\boldsymbol{\Cv}}
		+   \gamma_1\nu_2  \tau h^2 \dN{\nu^{-\frac 12}\nabla \ul{\rho_h^n}}_{0,\Omega_\eta^n}^2\notag \\
		&\hspace{5mm}+ 2\tau\mathscr{J}_p^n(\ul{\rho_h^n},\ul{\rho_h^n})   \notag \\
		\leq\;& \sum_{l=1}^4 \dN{\Psibf_l(\uuline{\Thetabf_h^{n-1}})}_{0,\Omega_\eta^{n-1}}^2
		+C\tau \sum_{j=1}^4\Big\{\big\|\ul{\thetabf_h^{n-j}}\big\|_{0,\Omega^{n-j}_\eta}^2
		+ \nu_2^{-1}h^4\tN{\ul{\thetabf_h^{n-j}}}^2_{\boldsymbol{\Cv}}\Big\}\notag\\
		&+ (0.5 + 2\varepsilon)\tau \tN{\ul{\thetabf_h^{n-1}}}_{\boldsymbol{\Cv}}^2 
		+C(1+\varepsilon^{-1})\tau\dN{\Lambda^4 \uuline{\Thetabf_h^n}}_{0,\Omega_\eta^n}^2
		+C\tau \varepsilon^{-1}\N{\Ce_1^n}_{1,(\Cv_h^n)'}^2\notag \\
		&+\frac{C}{\tau \gamma_1\nu_2\varepsilon} \dN{\Ce_2^n}_{1,(\Cq_h^n)'}^2
		+\varepsilon \tau^{-1}\dN{\uline{\hat{\Thetabf}_h^{n-1,n}}
			-\uline{\Thetabf_h^{n-1,n}}}_{0,\Omega_\eta^n}^2 \notag \\
		&+ \tau \varepsilon \Big(\gamma_1\nu_2 h^2\tN{2\ul{\rho_h^n}+\ul{\hat{\rho}_h^{n-1,n}}}^2_{1,\Cq}
		+\dN{\nu^{\frac 12}\ul{\Bv_h}}_{0,\Omega_\eta^n}^2
		+ \tN{\ul{\Bv_h}}_{\boldsymbol{\Cv}}^2\Big).
		\label{err0}
	\end{align}
	
	Next we estimate the right-hand side of \eqref{err0}.
	From \eqref{ieq:vhXn} and \eqref{eq:temp3}, we have
	\begin{align}
		\tN{\Lambda^4 \uuline{\Thetabf_h^n}}_{0,\Omega_\eta^{n-j}}^2
		\lesssim\,& \sum_{j=0}^4 \Big(\dN{\ul{\thetabf_h^{n-j}}}_{0,\Omega_\eta^n}^2 +
		h^5\nu_2^{-1}\tN{\ul{\thetabf_h^{n-j}}}_{\boldsymbol{\Cv}}^2\Big),
		\label{Theta-est1}\\
		\dN{\nu^{\frac 12}\ul{\Thetabf_h^{n-1,n}}}_{0,\Omega_\eta^n}^2 \leq\,&
		(1+C\tau)\dN{\ul{\thetabf_h^{n-1}}}_{0,\Omega_\eta^{n-1}}^2
		+C\tau^7 \tN{\ul{\thetabf_h^{n-1}}}_{\boldsymbol{\Cv}}^2.
		\label{Theta-est2}
	\end{align}
	Moreover, using \eqref{eq:err-proj0}, the fact that $\nu\le 1$, and inverse estimates, we find that
	\begin{align}
		&\tN{2\ul{\rho_h^n}+\ul{\hat{\rho}_h^{n-1,n}}}_{1,\Cq}
		\leq 2\tN{\ul{\rho_h^n}}_{1,\Cq} + Ch^{-1}\tN{\ul{\thetabf_h^{n-1}}}_{\boldsymbol{\Cv}},
		\label{eq: rho 1} \\
		&\tN{\ul{\Bv_h}}_{\boldsymbol{\Cv}} \leq  2\tN{\ul{\thetabf_h^{n}}}_{\boldsymbol{\Cv}}
		+\tN{\ul{\hat{\Thetabf}_h^{n-1,n}}}_{\boldsymbol{\Cv}}
		\leq  2\tN{\ul{\thetabf_h^{n}}}_{\boldsymbol{\Cv}}
		+C\tN{\ul{\thetabf_h^{n-1}}}_{\boldsymbol{\Cv}}, \label{eq: vh v} \\
		&\dN{\nu^{\frac 12}\ul{\Bv_h}}_{0,\Omega_\eta^n}
		\leq 2\dN{\nu^{\frac 12}\ul{\thetabf_h^n}}_{0,\Omega_\eta^n}
		+\dN{\ul{\hat{\Thetabf}_h^{n-1,n}}-\ul{\Thetabf_h^{n-1,n}}}_{0,\Omega_{\eta}^n}
		+\dN{\nu^{\frac 12}\ul{\Thetabf_h^{n-1,n}}}_{0,\Omega_{\eta}^n}.
		\label{eq:vh L2}
	\end{align}
	Substitute all estimates above into \eqref{err0} and take the sum over
	$4\leq n\leq m$. By the assumption $h\le \gamma_1\nu_2$ and proper arrangements, we end up with
	\begin{align}
		&\dN{\Psibf_1^4 (\ul{\Thetabf_h^m})}_{0,\Omega_\eta^m}^2
		+\sum_{n=4}^m\tau\Big\{(0.8-C\varepsilon -C\gamma_1\nu_2) \tN{\ul{\thetabf_h^n}}^2_{\boldsymbol{\Cv}}\notag \\
		&+\gamma_1\nu_2 h^2\dN{\nu^{-\frac 12}\nabla \ul{\rho_h^n}}_{0,\Omega_\eta^n}^2
		+ 2\mathscr{J}_p^n(\ul{\rho_h^n},\ul{\rho_h^n})\Big\}\notag \\
		\leq\,& \sum_{n=0}^3\sum_{l=1}^4 \dN{\Psibf_l(\uuline{\Thetabf_h^{n}})}_{0,\Omega_\eta^n}^2
		+ C(1+\varepsilon^{-1})\sum_{n=0}^m\tau \big\|\ul{\thetabf_h^n}\big\|_{0,\Omega^n_\eta}^2\notag \\
		&+ 4\varepsilon \tau \gamma_1\nu_2 h^2 \TN{\ul{\rho_h^n}}_{1,\Cq}^2
		+\big(\varepsilon\tau^{-1}+C\varepsilon\tau\big)
		\dN{\uline{\hat{\Thetabf}_h^{n-1,n}}-\uline{\Thetabf_h^{n-1,n}}}_{0,\Omega_\eta^n}^2\notag \\
		&+ C \tau\varepsilon^{-1}\N{\Ce_1^n}_{1,(\Cv_h^n)'}^2
		+ C(\gamma_1\nu_2\varepsilon\tau)^{-1} \dN{\Ce_2^n}_{1,(\Cq_h^n)'}^2 . \label{err:finial}
	\end{align}
	Then combining \eqref{eq:err-proj1} and \eqref{Theta-est1} leads to
	\begin{align}
		\dN{\uline{\hat{\Thetabf}_h^{n-1,n}}-\uline{\Thetabf_h^{n-1,n}}}_{0,\Omega_\eta^n}^2
		\leq&\,  C\gamma_1^2 \nu_2h^2\sum_{j=1}^5\Big(\dN{\ul{\thetabf_h^{n-j}}}_{0,\Omega_\eta^{n-j}}^2
		+h^4\tN{\ul{\thetabf_h^{n-j}}}_{\boldsymbol{\Cv}}^2\Big)\notag\\
		&+h^2\mathscr{J}_p^{n-1}(\ul{\rho_h^{n-1}},\ul{\rho^{n-1}_h})
		+\N{\Ce_2^{n-1}}_{1,(\Cq_h^{n-1})'}^2
		\notag\\	
		&+h^2\tN{\ul{\thetabf_h^{n-1}}}^2_{\boldsymbol{\Cv}}
		+ C(\gamma_1\nu_2)^2h^4\dN{\nu^{-\frac 12}\nabla\ul{\rho_h^{n-1}}}^2_{0,\Omega_\eta^n} . 
		\label{Theta-est3}
	\end{align}
	Using the estimate of $\ul{\rho_h^n}$ in Lemma~\ref{lem:stab of rho}, we obtain
	\begin{align}
		&\dN{\nu^{-\frac 12}\nabla \ul{\rho_h^n}}_{0,\Omega_\eta^n}^2 \geq	
		\tN{\ul{\rho_h^n}}_{1,\Cq}^2-\tN{\ul{\rho_h^n}}_{\Cq}^2-Ch^{-2}
		\mathscr{J}_p^n(\ul{\rho_h^n},\ul{\rho_h^n}) \notag \\
		\geq&\, \tN{\ul{\rho_h^n}}_{1,\Cq}^2
		- Ch^{-2}\mathscr{J}_p^n(\ul{\rho_h^{n}},\ul{\rho_h^{n}})
		-C \tN{\ul{\thetabf_h^n}}_{\boldsymbol{\Cv}}^2
		-C\N{\Ce_1^{n}}_{1,(\Cv_h^n)'}^2
		\notag \\ 
		&-C\nu_2^{-1}h^{-2}
		\sum_{j=0}^4 \left( \dN{\ul{\thetabf_h^{n-j}}}_{0,\Omega_\eta^{n-j}}^2
		+h^{4}\nu_2^{-1} \tN{\ul{\thetabf_h^{n-j}}}_{\boldsymbol{\Cv}}^2\right). \label{eq:rho H1}
	\end{align}
	 Inserting \eqref{Theta-est3}  and \eqref{eq:rho H1} into \eqref{err:finial} yields
	\begin{align*}
		&\dN{\Psibf_1^4 (\ul{\Thetabf_h^m})}_{0,\Omega_\eta^m}^2
		+\sum_{n=4}^m\tau\Big\{(0.8-C_0\varepsilon -C_0\gamma_1\nu_2) \tN{\ul{\thetabf_h^n}}^2_{\boldsymbol{\Cv}}\\
		&+(2-C_0\varepsilon) \mathscr{J}_p^n(\ul{\rho_h^n},\ul{\rho_h^n}) 
		+ (1-C_0\varepsilon)\gamma_1\nu_2 h^2\dN{\nu^{-\frac 12}\nabla \ul{\rho_h^n}}_{0,\Omega_\eta^n}^2\Big\} \notag\\
		\leq\,& C_1(1+\frac{1}{\varepsilon})\sum_{n=0}^m\dN{\ul{\thetabf_h^{n}}}_{0,\Omega_\eta^n}^2
		+ \frac{C_1}{\varepsilon}\sum_{n=4}^m\Big(\tau\N{\Ce_1^n}_{1,(\Cv_h^n)'}^2
		+ (\tau \gamma_1\nu_2)^{-1} \dN{\Ce_2^n}_{1,(\Cq_h^n)'}^2\Big)\\
		&+C_1\sum_{n=0}^3\tau \Big(\tN{\ul{\thetabf_h^n}}^2_{\boldsymbol{\Cv}}
		+\tau \mathscr{J}_p^n(\ul{\rho_h^n},\ul{\rho_h^n})
		+h^4\dN{\nu^{-\frac 12}\nabla \ul{\rho_h^n}}_{0,\Omega_\eta^n}^2\Big).
	\end{align*}
	Choose $\varepsilon = 0.2/C_0$, $\gamma_1\nu_2 \le 0.2/C_0$, and $h \le 0.2/C_0$ and
	note that $\Psibf_1^4(\uuline{\Thetabf_h^m})=0.06\ul{\thetabf_h^m}$. The proof is finished upon using
	Gronwall's inequality and inverse estimates.
\end{proof}

\subsection{The main theorem}

Now we are ready to present the main theorem of this paper.
\textit{\begin{center}
		\fbox{\parbox{0.975\textwidth}{
				\begin{theorem}\label{thm:uh-err}
					Let the assumptions in Theorem~\ref{thm:rhoh-err} be satisfied and
					suppose that the pre-calculated initial solutions satisfy for 
					$j=0,1,\cdots,k-1$,
					\begin{equation}\label{init-err}
						\dN{\ul{\tilde\Bu^j}-\ul{\Bu_h^j}}^2_{0,\Omega_\eta^j}
						+\tau \tN{\ul{\tilde\Bu^j}-\ul{\Bu_h^j}}_{\boldsymbol{\Cv}}^2
						+ \tau^2\tN{\ul{\tilde p_j} -\ul{p_h^j}}^2_{1,\Cq}
						\le C_0\tau^{2k}.
					\end{equation}
					Then upon a hidden constant independent of $\eta,\tau,h$, and material parameters, such that for $ k\leq m\leq N$
					\begin{align*}
						\dN{\ul{\tilde\Bu^m}-\ul{\Bu_h^m}}_{0,\Omega^m_\eta}^2
						+\sum_{n=k}^{m}\tau \Big(\tN{\ul{\tilde\Bu^n}-\ul{\Bu_h^n}}_{\boldsymbol{\Cv}}^2
						+\gamma_1\nu_2h^2 \tN{\ul p^n-\ul{p_h^n}}^2_{1,\Cq}\Big)
						\lesssim C_N\tau^{2k},
					\end{align*}
					where
					\begin{align*}
						C_N:=\,&  C_0  + \sum_{n=1}^N\tau\wt M_k^n
						+ \nu_2^{-1}\sum_{n=1}^N\tau\sum_{i=1,2}\Big(h^2\dN{\tilde \Bf_i^n}^2_{\BH^2(\Omega)}
						+\nu_i\N{\tilde \Bu_i^n}_{\BH^2(\Omega)}^2\Big)\\
						&+ \sum_{i=1,2} \nu_i^{-1}\N{\tilde\Bu_i}^2_{\BH^{k+1}(\Omega\times(0,T))}.
					\end{align*}
		\end{theorem}}}
\end{center}}
\vspace{1mm}

\begin{proof}
	From \eqref{unpn}--\eqref{ieq:theta-rho-n} and assumption \eqref{init-err}, there exists a constant $C>0$ independent of $\eta$, $\tau$, and $h$, such that
	for $0\leq j\leq k-1$,
	\begin{align*}
		\begin{cases}
			\dN{\ul{\thetabf^j_h}}_{0,\Omega_\eta^j}
			\le \dN{\ul{\tilde\Bu^j}-\ul{\Bu_h^j}}_{0,\Omega_\eta^j}
			+ \dN{\ul{\thetabf^j}}_{0,\Omega_\eta^j}
			\le C_0\tau^k + C\,(\wt M_k^j)^{\frac12}h^{k+1},
			\vspace{1mm} \\
			\tN{\ul{\thetabf^j_h}}_{\boldsymbol{\Cv}}
			\le \tN{\ul{\tilde\Bu^j}-\ul{\Bu_h^j}}_{\boldsymbol{\Cv}}
			+ \tN{\ul{\thetabf^j}}_{\boldsymbol{\Cv}}
			\le C_0\tau^{k-1/2} +C\,(\wt M_k^j)^{\frac12}h^k,
			\vspace{1mm} \\
			\tN{\ul{\rho^j_h}}_{1,\Cq}
			\le \tN{\ul{\tilde p^j}-\ul{p_h^j}}_{1,\Cq}
			+ \tN{\ul{\rho^j}}_{1,\Cq}
			\le C_0\tau^{k-1} +C\, (\wt M_k^j)^{\frac12}h^{k-1}.
		\end{cases}
	\end{align*}
	From Theorem~\ref{thm:rhoh-err} and Lemmas~\ref{lem:E1}--\ref{lem:E2},
	we conclude that
	\begin{align*}
		\dN{\ul{\thetabf_h^m}}_{0,\Omega_\eta^m}^2
		+\sum_{n=k}^m\tau\Big(\tN{\ul{\thetabf_h^n}}^2_{\boldsymbol{\Cv}}
		+\gamma_1 \nu_2 h^2\TN{\ul{\rho_h^n}}_{1,\Cq}^2
		+\mathscr{J}_p^n(\ul{\rho_h^n},\ul{\rho_h^n})\Big)
		\lesssim C_N\tau^{2k}.    
	\end{align*}
	Using \eqref{ieq:theta-rho-n}--\eqref{ieq:rho-grad}, we finish the proof.
\end{proof}



\section{Numerical experiment}
\label{sec:num}

Now we use a numerical experiment on severely deforming domain to verify the convergence orders of the finite element method for $k=3,4$. The segment size for interface-tracking is set by
$\eta =0.5\tau^{\max(1,k/3)}$ in Algorithm~\ref{alg:spline} for computing $\BX_\tau^{n-1,n}$.
In order to capture large deformations of the domain accurately, 
we apply the cubic MARS (Mapping and Adjusting Regular Semialgebraic sets) algorithm in real computations \cite{zha18}.
The algorithm is a slight modification of Algorithm~\ref{alg:spline} by creating new markers or removing old markers adaptively.

The whole domain is taken as $\Omega=(0,1)\times(0,1)$ and the time interval is $[0,T]$ with $T=1.5$.
The initial sub-domain $\Omega_1(0)$ is a disk of radius $0.15$ and centering at $(0.5,0.75)$.
The flow velocity which drives the interface $\Gamma(t)$ is taken as
\begin{equation*}
	\Bw(\Bx,t) =\cos(\pi t/3) \left(\sin^2 (\pi x_1) \sin (2\pi x_2),
	-\sin^2 (\pi x_2 )\sin (2\pi x_1)\right)^\top,
\end{equation*}
At the final time, $\Omega_{1}(T)$  is stretched into a snake-like domain (see Fig.~\ref{fig:Omegat}).

To test the high-order error estimates, we set the true solution by smooth velocity and smooth pressure in
each sub-domain
\begin{align*}
	&
	\begin{cases}
		\Bu_1(\Bx,t)=\cos t\big(\cos(\pi x_1) \sin(\pi x_2),\, -\sin (\pi x_1) \cos(\pi x_2)\big)^\top, \\
		p_1(\Bx,t) = \cos(0.5\pi x_1)\sin(0.5\pi x_2),
	\end{cases}
	\quad \Bx\in\Omega_1(t), \\
	&
	\begin{cases}
		\Bu_2(\Bx,t)= e^{x_1}\big(\sin(\pi x_2+\pi t),\, \pi^{-1}\cos(\pi x_2+\pi t)\big)^\top,\\
		p_2(\Bx,t) = \sin(0.5\pi x_1)\cos(0.5\pi x_2).\,
	\end{cases}
	\quad \Bx\in\Omega_2(t).
\end{align*}
The coefficients of viscosity are taken as $\nu_1 =1$ and $\nu_2=10^{-3}$.
The body forces in \eqref{eq:Oseen0} are defined as
\ben
\Bf_i(t) = \partial_t \Bu_i + (\Bw\cdot\nabla)\Bu_i -\nu_i \Delta \Bu_i + \nabla p_i
\quad \hbox{in}\;\; \Omega_i(t),\quad i=1,2.
\een
However, the jumps $\jump{\ul{\Bu}}$ and $\jump{\nu \partial_{\Bn} \ul{\Bu}-\ul{p}\Bn}$ do not match the interface conditions in \eqref{eq:Oseen2}. To remedy the situation, we define
\ben
\Bg_0^n = \Bu_1-\Bu_2, \quad
\Bg_1^n = \nu_1 \partial_\Bn(\Bu_1-\Bu_2) - (p_1-p_2) \Bn \quad \hbox{on}\;\;\Gamma^n_\eta,
\een
and replace the right-hand sides of \eqref{eq:disc-u} and \eqref{eq:disc-p}, respectively, with
\ben
\Aprod[\Omega^n_\eta]{\ul\Bf^n,\ul{\Bv_h}} +
\int_{\Gamma^n_\eta}\big(\Bg^n_1\cdot\Avgg{\ul{\Bv_h}}
-\Bg^n_0\cdot \avg{\nu\partial_\Bn\ul{\Bv_h}}\big),
\qquad \int_{\Gamma^n_\eta}\Bg^n_0\cdot\Bn \avg{\ul{q_h}}.
\een
The approximation errors are measured with the following quantities
\begin{align*}
	\begin{array}{ll}
		e_{u,0} = \dN{\ul{\Bu^N}-\ul{\Bu_h^{N}}}_{0,\Omega_\eta^N},
		&e_{u,1} = \Big(\sum\limits_{n=k}^N\tau\dN{\nu^{\frac12}\nabla(\ul{\Bu^n}
			-\ul{\Bu_h^n})}_{0,\Omega_\eta^n}^2\Big)^{\frac12},   \vspace{1mm} \\
		e_{p,0} = \Big(\sum\limits_{n=k}^N \tau\dN{\nu^{-\frac12}(\ul{p^n}
			-\ul{p_h^n})}_{0,\Omega_\eta^n}^2\Big)^{\frac12},
		&e_{p,1} =\Big(\sum\limits_{n=k}^N \tau \dN{\nu^{-\frac12}
			\nabla(\ul{p^n}-\ul{p_h^n})}_{0,\Omega_\eta^n}^2\Big)^{\frac12}.
	\end{array}
\end{align*}
To simplify the computation, we set pre-calculated initial values by the exact solution, namely, $\ul{\Bu^j_h}=\ul\Bu(t_j)$ for $0\le j\le k-1$. Throughout the section, we set $\gamma_0=10^3$, $\gamma_1=1$, and $h=\tau$.

Numerical results for $k=3,4$ are shown in Tables~\ref{tab:k=3 velocity}--\ref{tab:k=4 pressure}. They show that optimal convergence is obtained for both the third- and fourth-order methods. Moreover, Tables~\ref{tab:k=3 pressure} and \ref{tab:k=4 pressure} show that the discrete pressure is of the $k^{\rm th}$-order convergence under both the weighted $L^2$-norm and the weighted $H^1$-norm. Finally, Fig.~\ref{fig:Omegat} shows the shapes of tracked interfaces at $t_n=0.0$, $0.5$, $1.0$, and $1.5$, respectively. We find that the interface has a large deformation at $t_n=T$. Nevertheless, our numerical experiment shows that optimal convergence is obtained for such a challenging simulation.

\begin{table}[http!]
	\center
	\setlength{\tabcolsep}{3mm}
	\begin{tabular}{ l|cc|cc}
		\hline
		$h=\tau$     &$e_{u,0}$       &order   &$e_{u,1}$       &order      \\ \hline
		1/16   &6.478e$-$03  &---   &6.259e$-$03  &--- \\
		1/32   &8.120e$-$04  &3.00  &8.679e$-$04  &2.85  \\
		1/64   &9.941e$-$05  &3.03  &1.119e$-$04  &2.96  \\
		1/128  &1.228e$-$05  &3.02  &1.413e$-$05  &2.99  \\  \hline
	\end{tabular}
	\caption{Convergence orders of the discrete velocity for $k=3$.} \label{tab:k=3 velocity}
\end{table}
\begin{table}[http!]
	\center
	\setlength{\tabcolsep}{3mm}
	\begin{tabular}{ l|cc|cc}
		\hline
		$h=\tau$     &$e_{p,0}$       &order   &$e_{p,1}$       &order      \\ \hline
		1/16    &1.83e$-$01  &---   &1.47e$+$00  &--- \\
		1/32    &2.41e$-$02  &2.93  &1.97e$-$01  &2.90  \\
		1/64    &3.07e$-$03  &2.97  &2.51e$-$02  &2.97  \\
		1/128   &3.87e$-$04  &2.99  &3.18e$-$03  &2.98  \\  \hline
	\end{tabular}
	\caption{Convergence orders of the discrete pressure for $k=3$.} \label{tab:k=3 pressure}
\end{table}
\begin{table}[http!]
	\center
	\setlength{\tabcolsep}{3mm}
	\begin{tabular}{ l|cc|cc}
		\hline
		$h=\tau$     &$e_{u,0}$    &order   &$e_{u,1}$       &order    \\ \hline
		1/16   &1.46e$-$03  &---   &4.20e$-$03  &---\\
		1/32   &8.70e$-$05  &4.07  &2.71e$-$04  &3.95  \\
		1/64   &5.17e$-$06  &4.07  &1.67e$-$05  &4.02  \\
		1/128  &3.16e$-$07  &4.03  &1.03e$-$06  &4.02  \\  \hline
	\end{tabular}
	\caption{Convergence orders of the discrete velocity for $k=4$.} \label{tab:k=4 velocity}
\end{table}

\begin{table}[http!]
	\center
	\setlength{\tabcolsep}{3mm}
	\begin{tabular}{ l|cc|cc}
		\hline
		$h=\tau$     &$e_{p,0}$       &order   &$e_{p,1}$    &order    \\ \hline
		1/16   &6.80e$-$02  &---   &6.51e$-$01  &---   \\
		1/32   &4.73e$-$03  &3.85  &4.56e$-$02  &3.84  \\
		1/64   &3.08e$-$04  &3.94  &2.97e$-$03  &3.94  \\
		1/128  &1.96e$-$05  &3.97  &1.89e$-$04  &3.97  \\  \hline
	\end{tabular}
	\caption{Convergence orders of the discrete pressure for $k=4$.} \label{tab:k=4 pressure}
\end{table}

\begin{figure}[http!]
	\centering
	\includegraphics[width=0.2\textwidth]{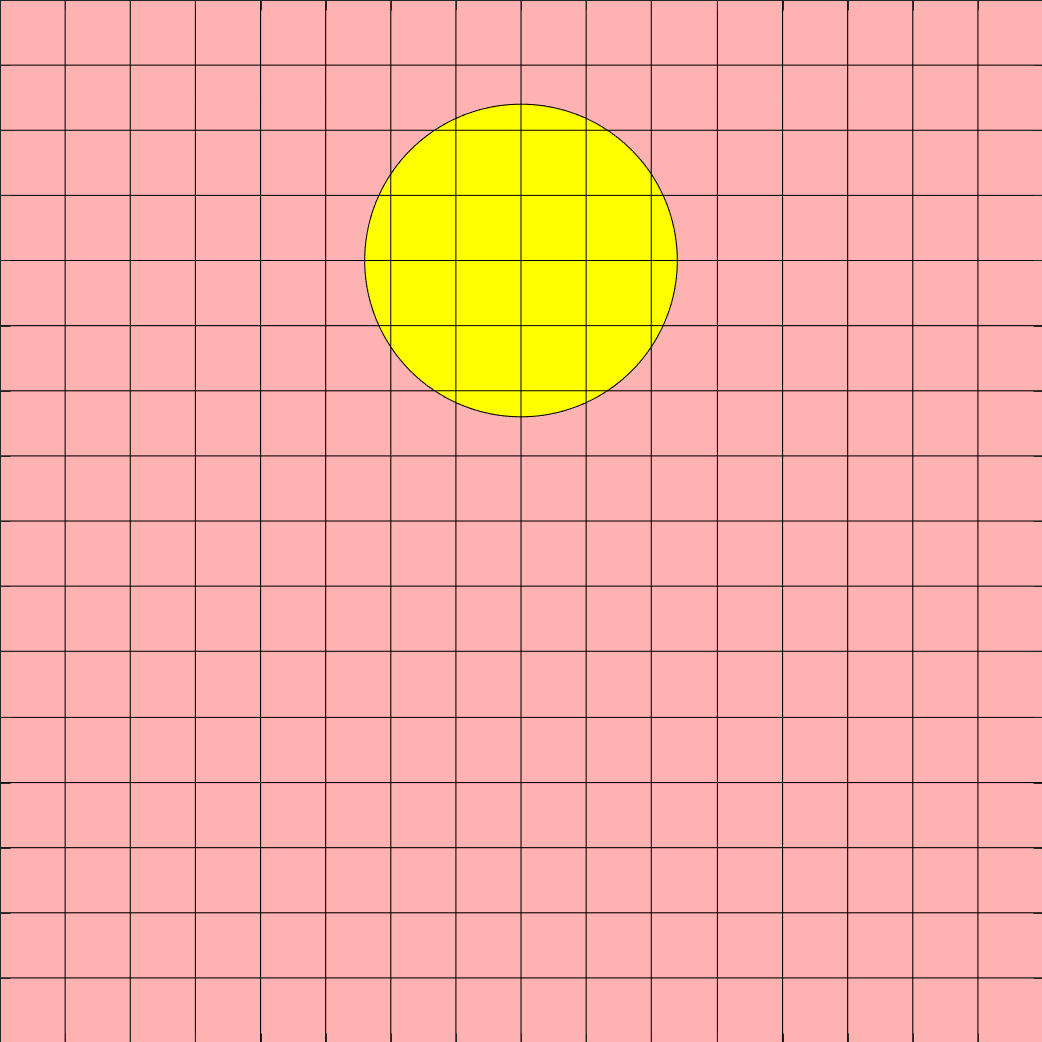}
	\includegraphics[width=0.2\textwidth]{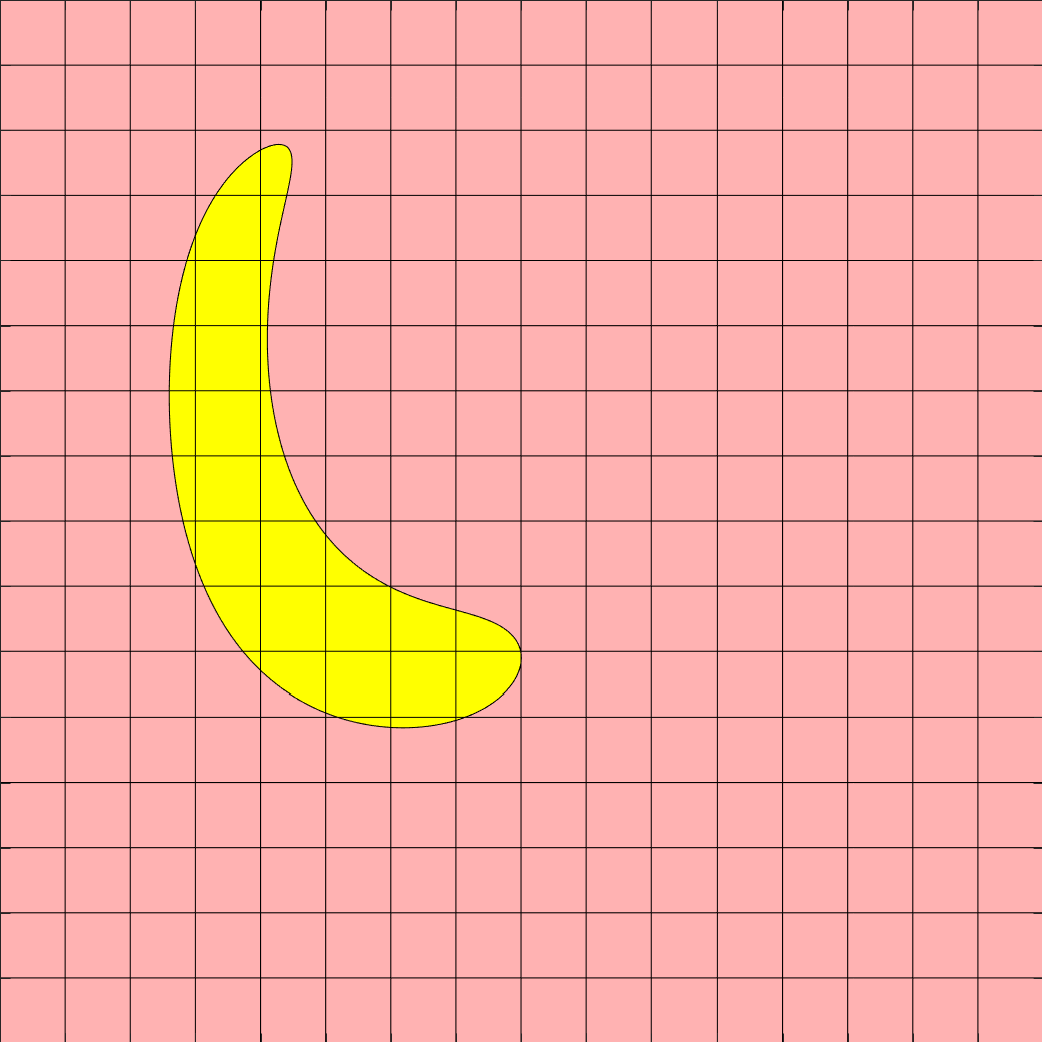}
	\includegraphics[width=0.2\textwidth]{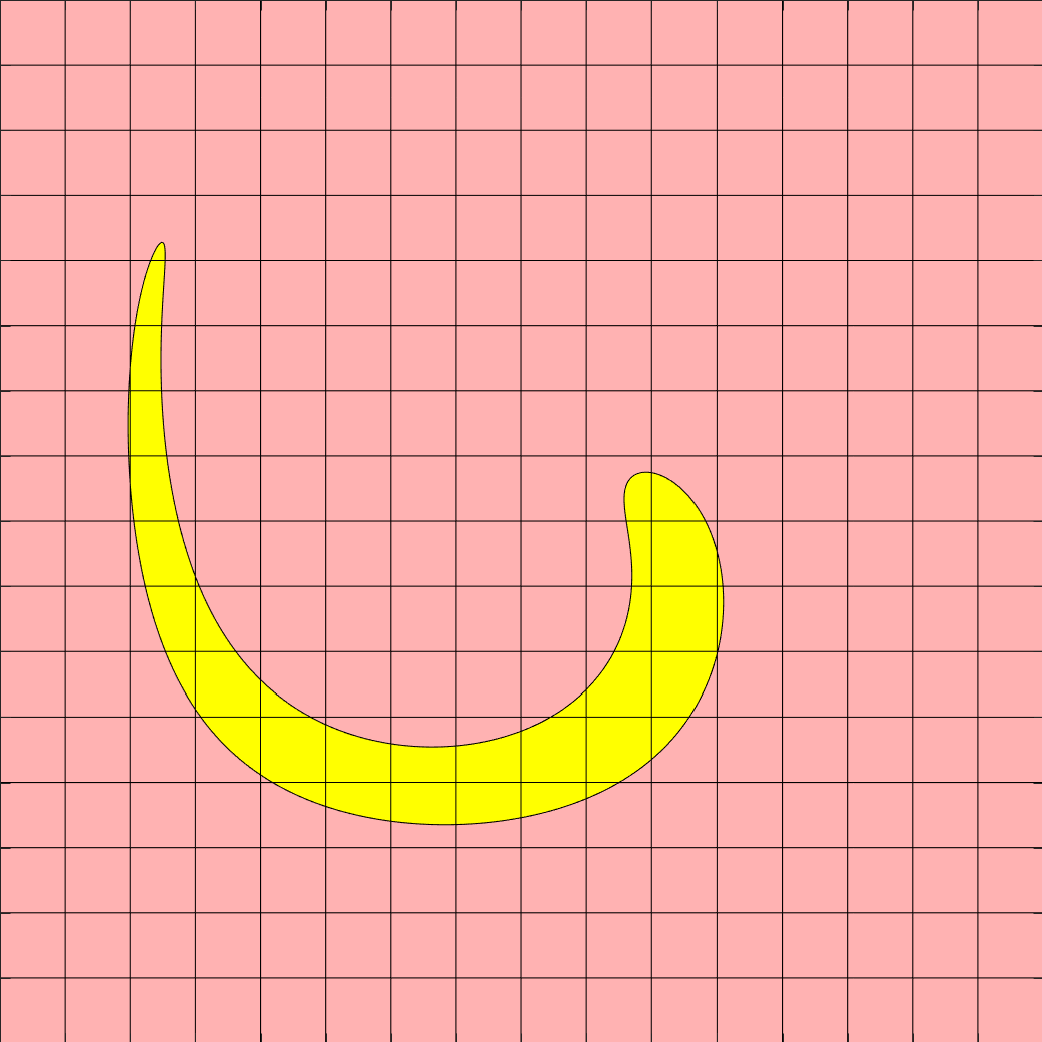}
	\includegraphics[width=0.2\textwidth]{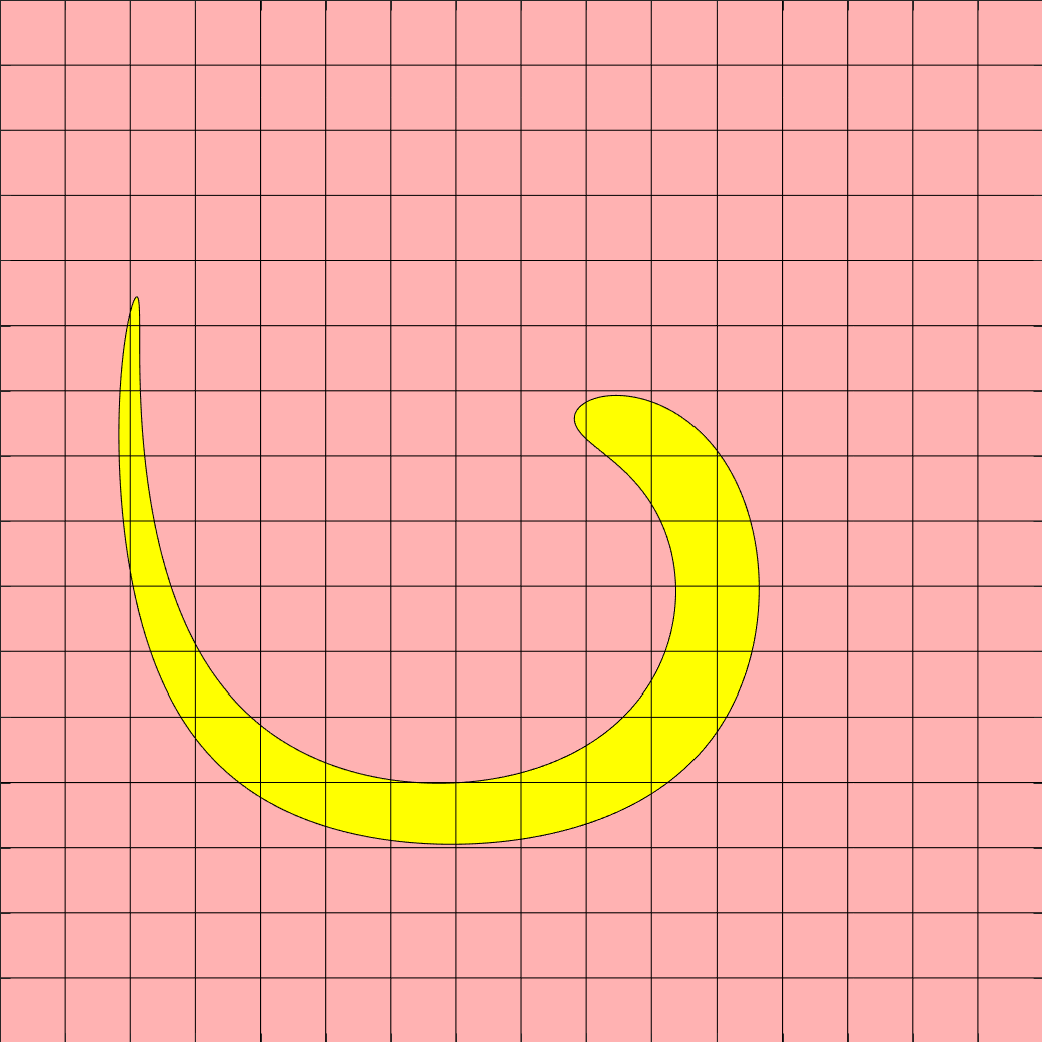}	
	\caption{Tracked interfaces $\Gamma^n_\eta$ at $t_n=0.0$, $0.5$, $1.0$, and $1.5$, respectively ($h=1/16$).}	\label{fig:Omegat}
\end{figure}



\normalem

\bibliographystyle{abbrv}

\end{document}

%% file: macros.tex
\providecommand{\Div}{\operatorname{div}}          
\providecommand{\curl}{\operatorname{{\bf curl}}}  

\providecommand*{\Dist}[2]{\operatorname{dist}({#1};{#2})}   
\providecommand*{\Dist}[2]{\Dist{#1}{#2}}
\providecommand{\Det}{\operatorname{det}}                    
\providecommand*{\Dom}[1]{{\rm Dom}({#1})}                   
\providecommand*{\Span}[1]{\operatorname{Span}\left\{{#1}\right\}}     

\providecommand{\esssup}{\operatorname*{esssup}}

\newcommand{\Vv}{{\mathbf{v}}}
\newcommand{\Vw}{{\mathbf{w}}}

\newcommand{\Ba}{{\boldsymbol{a}}}
\newcommand{\Bb}{{\boldsymbol{b}}}

\newcommand{\Be}{{\boldsymbol{e}}}
\newcommand{\Bf}{{\boldsymbol{f}}}
\newcommand{\Bg}{{\boldsymbol{g}}}

\newcommand{\Bn}{{\boldsymbol{n}}}

\newcommand{\Bp}{{\boldsymbol{p}}}
\newcommand{\Bq}{{\boldsymbol{q}}}

\newcommand{\Bu}{{\boldsymbol{u}}}
\newcommand{\Bv}{{\boldsymbol{v}}}
\newcommand{\Bw}{{\boldsymbol{w}}}
\newcommand{\Bx}{{\boldsymbol{x}}}

\newcommand{\Bz}{{\boldsymbol{z}}}

\newcommand{\BC}{{\boldsymbol{C}}}

\newcommand{\BH}{{\boldsymbol{H}}}

\newcommand{\BL}{{\boldsymbol{L}}}

\newcommand{\BR}{{\boldsymbol{R}}}

\newcommand{\BU}{{\boldsymbol{U}}}
\newcommand{\BV}{{\boldsymbol{V}}}
\newcommand{\BW}{{\boldsymbol{W}}}
\newcommand{\BX}{{\boldsymbol{X}}}

\newcommand{\etabf}{\boldsymbol{\eta}}
\newcommand{\thetabf}{\boldsymbol{\theta}}

\newcommand{\xibf}{\boldsymbol{\xi}}
\newcommand{\pibf}{\boldsymbol{\pi}}

\newcommand{\phibf}{\boldsymbol{\phi}}

\newcommand{\chibf}{\boldsymbol{\chi}}

\newcommand{\Thetabf}{\boldsymbol{\Theta}}

\newcommand{\Psibf}{\boldsymbol{\Psi}}

\newcommand{\Phibf}{\boldsymbol{\Phi}}

\newcommand{\Cb}{\mathcal{B}}

\newcommand{\Ce}{\mathcal{E}}
\newcommand{\Cf}{\mathcal{F}}
\newcommand{\Cg}{\mathcal{G}}

\newcommand{\Cl}{\mathcal{L}}

\newcommand{\Cp}{\mathcal{P}}
\newcommand{\Cq}{\mathcal{Q}}

\newcommand{\Cs}{\mathcal{S}}
\newcommand{\Ct}{\mathcal{T}}

\newcommand{\Cv}{\mathcal{V}}

\newcommand{\bbA}{\mathbb{A}}
\newcommand{\bbB}{\mathbb{B}}

\newcommand{\bbI}{\mathbb{I}}
\newcommand{\bbJ}{\mathbb{J}}

\newcommand{\bbL}{\mathbb{L}}

\newcommand{\bbR}{\mathbb{R}}

\newcommand{\bbU}{\mathbb{U}}

\newcommand{\bbW}{\mathbb{W}}

\providecommand*{\wt}[1]{\widetilde{#1}}

\newcommand*{\N}[1]{\left\|{#1}\right\|}     
\newcommand*{\dN}[1]{\big\|{#1}\big\|}     
\newcommand*{\tN}[1]{\big\|\hspace{-0.09em}\big|{#1}\big|\hspace{-0.09em}\big\|}     
\newcommand*{\TN}[1]{\Big\|\hspace{-0.09em}\Big|{#1}\Big|\hspace{-0.09em}\Big\|}     
\newcommand*{\SN}[1]{\left|{#1}\right|}      

\newcommand*{\Lp}[2][\defaultdomain]{L^{#2}({#1})}
\newcommand*{\Lpv}[2][\defaultdomain]{\BL^{#2}({#1})}
\newcommand*{\NLp}[3][\defaultdomain]{\N{#2}_{\Lp[#1]{#3}}}
\newcommand*{\NLpv}[3][\defaultdomain]{\N{#2}_{\Lpv[#1]{#3}}}

\newcommand*{\Ltwo}[1][\defaultdomain]{\Lp[#1]{2}}
\newcommand*{\Ltwov}[1][\defaultdomain]{\Lpv[#1]{2}}
\newcommand*{\NLtwo}[2][\defaultdomain]{\NLp[#1]{#2}{2}}
\newcommand*{\NLtwov}[2][\defaultdomain]{\NLpv[#1]{#2}{2}}

\newcommand*{\Linfv}[1][\defaultdomain]{\BL^{\infty}({#1})}

\newcommand*{\NLinfv}[2][\defaultdomain]{\N{#2}_{\Linfv[{#1}]}}

\newcommand*{\Hm}[2][\defaultdomain]{H^{#2}({#1})}
\newcommand*{\Hmv}[2][\defaultdomain]{\BH^{#2}({#1})}
\newcommand*{\bHm}[3][\defaultdomain]{H_{#3}^{#2}({#1})}

\newcommand*{\bHmv}[3][\defaultdomain]{\BH_{#3}^{#2}({#1})}

\newcommand*{\Hone}[1][\defaultdomain]{\Hm[#1]{1}}
\newcommand*{\Honev}[1][\defaultdomain]{\Hmv[#1]{1}}

\newcommand*{\zbHone}[1][\defaultdomain]{\bHm[#1]{1}{0}}
\newcommand*{\zbHonev}[1][\defaultdomain]{\bHmv[#1]{1}{0}}
\newcommand*{\NHone}[2][\defaultdomain]{{\N{#2}}_{\Hone[{#1}]}}
\newcommand*{\NHonev}[2][\defaultdomain]{{\N{#2}}_{\Honev[{#1}]}}
\newcommand*{\SNHone}[2][\defaultdomain]{{\SN{#2}}_{\Hone[{#1}]}}
\newcommand*{\SNHonev}[2][\defaultdomain]{{\SN{#2}}_{\Honev[{#1}]}}

\newcommand*{\jump}[2][]{\left \llbracket{#2}\right\rrbracket_{#1}}

\newcommand*{\avg}[2][]{\{\hskip -3.5pt\{{#2}\}\hskip -3.5pt\}_{#1}}
\newcommand*{\Avg}[2][]{\big\{\hskip -4.4pt\big\{{#2}
\big\}\hskip -4.4pt\big\}_{#1}}

\newcommand*{\avgg}[2][]{\left\{\hskip -3.5pt\left\{\hskip -3.5pt\left\{{#2}
\right\}\hskip -3.5pt\right\}\hskip -3.5pt\right\}_{#1}}

\newcommand*{\Avgg}[2][]{\left\{\hskip -4.5pt\left\{\hskip -4.5pt\left\{{#2}
\right\}\hskip -4.5pt\right\}\hskip -4.5pt\right\}_{#1}}

\newcommand*{\Aprod}[2][]{\big\langle{#2}\big\rangle_{#1}}

\newcommand{\D}{\mathrm{d}}

\newcommand{\ol}{\overline}
\newcommand{\ul}{\underline}

\newcommand{\be}{\begin{eqnarray}}
\newcommand{\ee}{\end{eqnarray}}

\newcommand{\ben}{\begin{eqnarray*}}
\newcommand{\een}{\end{eqnarray*}}


%% file: Oseen.bbl
\begin{thebibliography}{100}
	\bibitem{adj19}
	S. Adjerid and K. Moon.
	An immersed discontinuous Galerkin method for acoustic wave propagation in inhomogeneous media. SIAM J. Sci. Comput., 41 (2019), pp. A139--A162.
	
	\bibitem{bab70}
	I. Babu\v{s}ka.  The finite element method for elliptic equations with discontinuous coefficients. Computing, 5 (1970), pp. 207--213.
	
	\bibitem{bec09}
	R. Becker, E. Burman and P. Hansbo.
	A Nitsche extended finite element method for incompressible elasticity with discontinuous modulus of elasticity. Comput. Methods Appl. Mech. Engrg., 198 (2009), pp. 3352--3360.
	
	\bibitem{bes12}
	M. Besier and W. Wollner.
	On the pressure approximation in nonstationary incompressible flow simulations on dynamically varying spatial meshes.
	Internat. J. Numer. Methods Fluids, 69 (2012), pp. 1045--1064.
	
	\bibitem{bor17}
	S.P.A. Bordas, E. Burman, M.G. Larson and M.A. Olshanskii.
	Geometrically Unfitted Finite Element Methods and Applications.
	Lecture Notes in Computational Science and Engineering, vol. 121, Springer, 2017.
	
	\bibitem{bre12}
	W.P. Breugem.
	A second-order accurate immersed boundary method for fully resolved simulations of particle-laden flows. J. Comput. Phys., 231 (2012), pp. 4469--4498.
	
	\bibitem{bre91}
	F. Brezzi and R. S. Falk.  Stability of high-order Hood-Tayler methods. SIAM J. Numer. Anal., 28 (1991), pp. 581--590.
	
	\bibitem{bur10}
	E. Burman and P. Hansbo.
	Fictitious domain finite element methods using cut elements: I. A stabilized Lagrange multiplier method. Comput. Methods Appl. Mech. Engrg., 199 (2010), pp. 2680--2686.
	
	\bibitem{bur12}
	E. Burman and P. Hansbo.
	Fictitious domain finite element methods using cut elements: II. A stabilized Nitsche method. Appl. Numer. Math., 62 (2012), pp. 328--341.
	
	\bibitem{bur15}
	E. Burman, S. Claus, P. Hansbo and M.G. Larson.
	A. Massing, CutFEM: discretizing geometry and partial differential equations.
	Internat. J. Numer. Methods Engrg. 104 (2015), pp. 472-501.
	
	\bibitem{cai11}
	Z.~Cai, X.~Ye and S.~Zhang.
	Discontinuous Galerkin Finite Element Methods for Interface Problems: A Priori and A Posteriori Error Estimations. SIAM J. Numer. Anal., 49 (2011), pp. 1761--1787.
	
	\bibitem{chor1979}
	A. J. Chorin and J. E. Marsden.
	A mathematical introduction to fluid mechanics.
	Springer-Verlag, New York, Heidelberg, Berlin, 1979.
	
	
	\bibitem{fri09}
	T. P. Fries and A. Zilian.
	On time integration in the XFEM. Internat. J. Numer. Methods Engrg., 79 (2009), pp. 69--93.
	
	\bibitem{glo01}
	R. Glowinski, T. W. Pan, T. I. Hesla, D. D. Joseph and J. P\'{a}riaux.
	A fictitious domain approach to the direct numerical simulation of incompressible viscous flow past moving rigid bodies: application to particulate flow. J. Comput. Phys., 169 (2001), pp. 363--426.
	
	
	\bibitem{guo21}
	R. Guo.
	Solving parabolic moving interface problems with dynamical immersed spaces on unfitted meshes:fully discrete analysis. SIAM J. Numer. Anal., 59 (2021), pp. 797--828.
	
	\bibitem{guz18}
	J. Guzm\'{a}n and M. Olshanskii.
	Inf-sup stability of geometrically unfitted Stokes finite elements. Math. Comp., 87 (2018), pp. 2091--2112.
	
	\bibitem{han02}
	A. Hansbo and P. Hansbo.  An unfitted finite element method, based on Nitsche's method, for elliptic interface problems. Comput. Methods Appl. Mech. Engrg., 191 (2002), pp. 5537--5552.
	
	
	\bibitem{han14}
	P.Hansbo, M. G. Larson and S. Zahedi.
	A cut finite element method for a Stokes interface problem.
	Applied Numerical Mathematics, 85 (2014), pp. 90--114.
	
	\bibitem{hua17}
	P. Huang, H. Wu, and Y. Xiao.  An unfitted interface penalty finite element method for elliptic interface problems. Comput. Methods Appl. Mech. Engrg., 323 (2017) pp. 439--460.
	
	\bibitem{jar09}
	H. Jaroslav and R. Yves .
	A new fictitious domain approach inspired by the extended finite element method. SIAM J.Numer. Anal. 47 (2009), pp. 1474--1499.
	
	\bibitem{leh13}
	C. Lehrenfeld and R. Arnold.
	Analysis of a Nitsche XFEM-DG discretization for a class of two-phase mass transport problems. SIAM J.Numer. Anal. 51 (2013), pp. 958--983.
	
	\bibitem{leh15}
	C. Lehrenfeld.
	The Nitsche XFEM-DG space-time method and its implementation in three space dimensions. SIAM J. Sci. Comput., 37 (2015), pp. A245-A270.
	
	\bibitem{leh19}
	C. Lehrenfeld and M.A. Olshanaskii.
	An Eulerian finite element method for PDEs in time-dependent domains. ESAIM Math. Model. Numer. Anal., 53 (2019), pp. 585--614.
	
	
	\bibitem{lev94}
	R.J. LeVeque and Z. Li.
	The immersed interface method for elliptic equations with discontinuous coefficients and singular sources. SIAM J.Numer. Anal. 31 (1994), pp. 1019--1044.
	
	\bibitem{li98}
	Z. Li.  The immersed interface method using a finite element formulation. Appl. Numer. Math., 27 (1998), pp. 253--267.
	
	\bibitem{li03}
	Z. Li, T. Lin, and X. Wu.  New Cartesian grid methods for interface problems using the finite element formulation. Numer. Math., 96 (2003), pp. 61--98.
	
	\bibitem{li06}
	Z. Li and K. Ito.
	The Immersed Interface Method: Numerical Solutions of PDEs Involving Interfaces and Irregular Domains. Frontiers in Applied Mathematics, Society for Industrial and Applied Mathematics, 2006.
	
	\bibitem{lin09}
	T. Lin, Y. Lin and X. Zhang.
	Partially penalized immersed finite element methods for elliptic interface problems. SIAM J.Numer. Anal. 79 (2009), pp. 69--93.
	
	\bibitem{liu20}
	H. Liu, L. Zhang, X. Zhang, and W. Zheng.  Interface-penalty finite element methods for interface problems in $H^1$, $\BH(\curl)$, and $\BH(\Div)$. J. Comput. Appl. Math., 367 (2020), 113137.
	
	
	\bibitem{liu13}
	J.~Liu.  Simple and efficient ALE methods with provable temporal accuracy up to fifth order for the Stokes equations on time-varying domains.
	SIAM J. Numer. Anal., 51 (2013), pp. 743--772.	
	
	\bibitem{lou21}
	Y. Lou and C. Lehrenfeld.
	Isoparametric unfitted BDF-Finite element method for PDEs on evolving domains.
	arXiv:2105.09162\,.
	
	
	\bibitem{ma21}
	C. Ma, Q. Zhang, and W. Zheng.  A high-order fictitious-domain method for the advection-diffusion equation on time-varying domain. arXiv:2104.01870\,.
	
	\bibitem{mas14}
	A. Massing, M.G. Larson, A. Logg, and M.E. Rognes.  A stabilized Nitsche fictitious domain method for the Stokes problem. J. Sci. Comput., 61 (2014), pp. 604--628.
	
	\bibitem{mit05}
	R. Mittal and G. Iaccarino.  Immersed boundary methods.
	Ann. Rev. Fluid Mech., 37 (2005), pp. 239--261.
	
	\bibitem{nic06}
	S. Nicaise and S. A. Sauter.  Efficient numerical solution of Neumann problems on complicated domains.
	Calcolo, 43 (2006), pp. 95--120.
	
	\bibitem{ols06}
	M. A. Olshanskii and A. Reusken.  Analysis of a Stokes interface problem.
	Numer. Math., 103 (2006), pp. 129--149.
	
	\bibitem{pes77}
	C. S. Peskin.  Numerical analysis of blood flow in the heart. J. Comput. Phys., 25 (1977), pp. 220--252.
	
	\bibitem{ste70}
	E. Stein.  Singular Integrals and Differentiability Properties of Functions. Princeton University Press, Princeton, New Jersey, 1970.
	
	\bibitem{sco90}
	L. R. Scott and S. Zhang.  Finite element interpolation of nonsmooth functions satisfying boundary conditions. Math. Comp., 54 (1990), pp. 483-¨C493.
	
	
	\bibitem{vou18}
	I. Voulis and A. Reusken.
	A time dependent Stokes interface problem: well-posedness and space-time finite element discretization.
	ESAIM Math. Model. Numer. Anal., 52 (2018), pp. 2187--2213.
	
	\bibitem{wah20}
	H. v. Wahl, T. Richter and C. Lehrenfeld.
	An unfitted Eulerian finite element method for the time-dependent Stokes problem on moving domains. arXiv:2002.02352\,.
	
	
	\bibitem{wu19}
	H. Wu and Y. Xiao.  An unfitted hp-interface penalty finite element method for elliptic interface problems. J. Comp. Math., 37 (2019), pp. 316--339.
	
	\bibitem{zha18}
	Q.~Zhang and A.~Fogelson.
	Fourth-and higher-order interface tracking via mapping and adjusting regular semianalytic sets represented by cubic splines. SIAM J. Sci. Comput., 40 (2018), pp. A3755--A3788.
	
	\bibitem{zha16}
	Q.~Zhang and A.~Fogelson.
	MARS: An analytic framework of interface tracking via mapping and adjusting regular semialgebraic sets.
	SIAM J. Numer. Anal., 54 (2016), pp. 530--560.
	
	\bibitem{zun13}
	P. Zunino.
	Analysis of backward Euler/extended finite element discretization of parabolic problems with moving interfaces.
	Comput. Methods Appl. Mech. Engrg., 258 (2013), pp. 152--165.				
\end{thebibliography}
